\pdfoutput=1

\documentclass[10pt]{article}

\makeatletter
\renewcommand{\maketitle}{\bgroup\setlength{\parindent}{0pt}
\begin{flushleft}
  \textbf{\@title}

  \@author
\end{flushleft}\egroup
} %%%%%
\def\footnoterule{\kern-3\p@
  \hrule \@width 2in \kern 2.6\p@} % the \hrule is .4pt high
\makeatother

\usepackage{preamble}
\usepackage{courier}
\usepackage[normalem]{ulem}
\usepackage[square,numbers]{natbib}

\newtheorem*{theorem1}{Theorem 1}
\newtheorem*{theorem2}{Theorem 2}
\newtheorem*{theorem3}{Theorem 3}
\newtheorem*{theorem4}{Theorem 4}
\newtheorem*{theorem5}{Theorem 5}

\title{{\Large \textbf{Some results on NIP groups and their Ellis groups}}}
\author{Atticus Stonestrom, Notre Dame\ \ \ \ \ \ \ \ \ \ \ \ {\footnotesize (Email: \texttt{atticusstonestrom@yahoo.com})}}
\date{}

\begin{document}
\maketitle

{\small\noindent\textbf{Abstract:} This paper has several parts. We begin by developing a theory of `piecewise (strong) f-genericity' in NIP groups, where we call a definable set piecewise (strong) f-generic if some union of finitely many translates of it is (strong) f-generic. Our main result here is that, in an NIP group, the definable sets that are not piecewise (strong) f-generic form an ideal. Our hope is that the corresponding piecewise (strong) f-generic types can provide a substitute in arbitrary NIP groups for the (strong) f-generic types of definably amenable NIP groups, and in the rest of the paper we give several applications. As a warmup, we give a new and slightly improved proof of Gismatullin's theorem on the `existence' of $G^\infty$ in NIP groups.

Two of the applications deal with the Ellis group of an NIP group. Let $T$ be an NIP theory, $G$ a definable group of $T$, and $M$ a model of $T$. By the `Ellis group of $G(M)$' we mean the Ellis group of the $G(M)$-flow of global types concentrated on $G$ and finitely satisfiable in $M$. In our first result we show that the size of the Ellis group of $G(M)$ is bounded above by $2^{|T|}$, independent of the choice of $M$, giving a substantial step forwards on the open question of whether the isomorphism type of the Ellis group of $G(M)$ is independent of the choice of $M$. In fact we obtain much stronger structural information on the Ellis group than just a bound on its size, showing that, for any elementary substructure $M_0\preccurlyeq M$, the natural restriction map from the Ellis group of $G(M)$ to $S_G(M_0)$ is injective. A crucial tool is the recent theorem of Chernikov-Gannon-Krupi\'nski and Basso-Zucker that the $\tau$-topology on the Ellis group is Hausdorff. This is also the starting point for our second application, which gives much more precise information on the finer structure of this group.

In the second application, we show that, if $T$ and $M$ are countable and the formulas of $T$ have uniformly bounded VC-codensity (a natural assumption satisfied by o-minimal theories, the theories of the $p$-adics, and the theories of algebraically closed valued fields), then the Ellis group of $G(M)$ has `finite Archimedean rank', meaning that its connected component is profinite-by-Lie, or equivalently that it is an inverse limit of compact Lie groups of bounded dimension. This result is unsurprisingly greatly inspired by Hrushovski's theorem that certain automorphism groups associated to NIP formulas and definable measures are of finite Archimedean rank, though there are some important differences of setting and methods.

Along the way to our second application, we give a general result about VC-sets in compact Hausdorff groups, that we hope will be of broader interest. Suppose $G$ is a compact Hausdorff group and that there is a countable family $(B_i:i\in I)$ of Borel subsets of $G$ such that (1) every open subset of $G$ is a union of some of the $B_i$, and (2) there is a fixed $\delta$ such that, for each $i\in I$, the family of translates $(gB_i:g\in G)$ has VC-density at most $\delta$. Then we prove that $G$ has finite Archimedean rank; more precisely, we show that $G$ is an inverse limit of compact Lie groups of dimension at most $(4\delta)^2$. This part of the paper can be read independently of the rest, and does not require any model theory.

In our last result we use our techniques to obtain a `local' result valid in arbitrary NIP groups: we show that, for any `bi-invariant' formula $\phi(x,y)$, the group $G/G^{00}_\phi$ has finite Archimedean rank. More precisely, if the VC-codensity of $\phi(x,y)$ is at most $\delta$, then $G/G^{00}_\phi$ is an inverse limit of compact Lie groups of dimension at most $(4\delta)^2$. This connects to, though is different than, a question of Hrushovski's. \newline

\noindent\textbf{Acknowledgements:} This material developed out of an (ongoing) project of mine to give a structure theory for NIP approximate groups (in fact this paper can be seen as a `part 1' of that project), and I would like to thank Artem Chernikov, Gabe Conant, Krzysztof Krupi\'nski, Samaria Montenegro, Alf Onshuus, Devrim Pekmezci, Frank Wagner, and Julia Wolf for many helpful conversations and comments on both that project and this one. Above all I am especially grateful to my PhD supervisor Anand Pillay, who has been a constant and invaluable source of feedback and encouragement.

In spring of 2025 in the Notre Dame model theory seminar we read the paper \cite{hrushovski} together, and I learned the proof of Hrushovski's theorem on local $G/G^{00}$ then in order to present it in the seminar; I am grateful to the participants of the seminar for numerous clarifying discussions, and in particular to Yuyan He, Anand Pillay, and Nick Ramsey for their presentations on the other sections of that paper.  \newline

\noindent\textbf{AI Declaration:}
I am deeply opposed to the ways AI is now becoming heavily used in mathematics research, and \uline{I strongly request that the central new concept I develop in this paper, piecewise strong f-genericity, not be investigated in any capacity with AI, nor used in any paper that includes proofs produced by AI or otherwise involves substantial AI contributions.} If you have any question about this paper please reach out to me instead of asking AI.

AI played no role in the vast majority of the paper, though it did do one crucial thing. My first and only interaction with it was while finishing Section \ref{boundedness of the ellis group section}. I had already proved Theorem \ref{downstairs lemma dynamical language} and Lemma \ref{downstairs lemma for ideal}, and strongly suspected that, together with Fact \ref{tau-topology is hausdorff}, they should imply Theorem \ref{restriction map ellis group theorem}, which is what I was aiming to prove. Unlike the other topics that appear in this paper, I am very much not an expert in topological dynamics. In fact, to give a complete disclaimer, I have not (yet) learned the details of Glasner's structure theorem for tame, metrizable, minimal flows from \cite{glasner_tameness}, and I don't (yet) understand the details of the proof of Fact \ref{tau-topology is hausdorff} from \cite{chernikov_gannon_krupinski} and \cite{basso_zucker}; instead I have used Fact \ref{tau-topology is hausdorff} as a `black box' throughout this project. (This is the only result in the paper that I use as a black box, and with everything else I cite to use I am intimately familiar with the proofs. This is not the sort of thing that I would have made explicit in the past, but now with the advent of AI it seems important to be explicit about these things.) I am of course completely comfortable and familiar with the basic definitions of topological dynamics, but I have never learned any of the deeper structure theory. Thus, though I suspected Theorem \ref{restriction map ellis group theorem} should follow from Lemma \ref{downstairs lemma for ideal} from a standard argument present in the topological dynamics literature, I was not (and still am not) familiar enough with the literature to know where to look. Hence, after having been stuck for a while trying to prove it myself, I decided out of curiosity to try discussing it with AI. (At the time my views on AI were nascent and not yet so clearly formed.) After some discussion, to my surprise, the AI claimed to give me a proof, which I did not really understand, using equicontinuous factors and Theorem 8.13 in \cite{basso_zucker}. The AI cited a specific computation from the proof of Theorem 8.13 in \cite{basso_zucker}, which in turns cites pages 205-207 from Chapter 14 in \cite{auslander_book}. While trying to understand this computation, I began by reading the relevant material from Chapter 14 in \cite{auslander_book}, and immediately realized that Lemma 12 from page 205, which was completely elementary, was exactly what I needed to make my original naive attempts to directly show that Lemma \ref{downstairs lemma for ideal} implies Theorem \ref{restriction map ellis group theorem} work. So in fact the AI's argument used vastly more machinery than necessary and I end up not using it. However, I would almost surely have never found Lemma 12 from \cite{auslander_book} without the AI alerting me that the computation in \cite{basso_zucker} was relevant to my question, and so I need to give a significant acknowledgement to the AI for that. I deeply regret not instead discussing the question with an expert in topological dynamics, who would undoubtedly have been able to quickly answer it. No other part of this paper involved the use of AI, neither in the mathematical content nor in the writing. \newline

\noindent\textbf{Notation:} Throughout, $T$ will denote a complete theory in a first-order language $L$, with a monster model $\mathfrak{C}$. \uline{For notational convenience, we will assume throughout that $\dcl(\varnothing)$ is a model, which we denote by $M_0$.} Throughout, $G$ will be a $\varnothing$-definable group of $T$, and for any model $M\preccurlyeq\mathfrak{C}$ we use $G(M)$ to denote the $M$-points of $G$. As usual we will often notationally identify $G$ with $G(\mathfrak{C})$ when it does not cause confusion. For a parameter set $C\subseteq\mathfrak{C}$, we use $S_G(C)$ to denote the space of all complete types over $C$ that are concentrated on $G$. For tuples $a,b$, and a small parameter set $C\subset\mathfrak{C}$, we write $a\equiv_C b$ to mean $\tp(a/C)=\tp(b/C)$. We write $\ind$ for non-forking independence and $\ind^u$ for coheir independence. We write $[n]$ to mean $\{1,\dots,n\}$, so that in particular $[0]=\varnothing$. Finally, we will assume throughout almost the entire paper that $T$ is NIP, except when explicitly stated otherwise. %} %Finally, we use one piece of non-standard notation: given parameter sets $A,B$, instead of using $AB$ to denote their union, we write $(A,B)$, and likewise we use $(a,b)$ to denote concatenation of tuples $a,b$. This is to avoid ambiguity with respect to the group operation.

} %small commented out

\section{Introduction} 
The development of a structure theory for NIP groups, ie for groups definable in NIP theories, has historically often required doing at least one of two things: 1) studying groups definable in subclasses of NIP structures of particular interest, the main ones being o-minimal structures, the $p$-adics, and algebraically closed valued fields (all of which have been enormously fruitful and successful endeavours), or 2) working in arbitrary NIP theories but putting some structural restrictions on the definable groups one studies. (There is of course much intersection between the two approaches, and note that the latter approach originated from the former, in particular in the solution from \cite{hrushovski_peterzil_pillay} to the highly influential `Pillay's conjecture' on groups definable in o-minimal structures.) For approach (2), the most general of the assumptions under which a rich structure theory has been developed is that of definable amenability, which appears already in \cite{hrushovski_peterzil_pillay}; in analogue with the notion of amenability for a discrete group, a definable group is said to be `definably amenable' if there is a translation-invariant finitely-additive probability measure on the Boolean algebra of its definable subsets. Some of the principal examples of definably amenable NIP groups are NIP groups that are amenable as abstract groups (so for example abelian or more generally solvable groups), stable groups, and compact groups definable in o-minimal expansions of the reals or in the $p$-adics. The structure theory of definably amenable NIP groups is extremely satisfactory and well-developed, with the main arc of its development being completed over the course of the papers \cite{hrushovski_peterzil_pillay}, \cite{hrushovski_pillay}, and \cite{chernikov_simon}. Nevertheless, many important examples of NIP groups are not definably amenable, the most basic being $\mathrm{SL}_2$ as a definable group in the reals or in the $p$-adics. So it is a natural goal to try to develop some structure theory for arbitrary definable groups in arbitrary NIP structures; we aim in this paper to make some progress in this goal.

In this paper we complete an important step, which is to identify and prove the existence of an appropriate ideal of `small' definable sets in arbitrary NIP groups. A proper ideal of formulas naturally gives rise to `wide' types, ie types that do not imply any formula in the ideal, and the use of such wide types, for appropriate ideals, is a highly persistent theme in the model theory of definable groups; the most famous example are perhaps the `generic types' of stable groups, where a definable subset of a group is said to be generic if finitely many translates of it cover the group. Two other very important examples are `f-genericity' and `strong f-genericity', where we say a definable subset $\phi(x,b)$ of a group is `f-generic' if no left translate of it forks over any model containing $b$, and `strong f-generic' if no left translate of it forks over any model.\footnote[1]{Disclaimer: the term `f-generic' was introduced by Newelski and Petrykowski in \cite{newelski_petrykowski}, where it was used to mean what we are called strong f-genericity; this was also the definition of `f-generic' used by Hrushovski and Pillay in \cite{hrushovski_pillay}. The term `strong f-generic' was introduced and notion of `f-generic' renamed by Chernikov and Simon in \cite{chernikov_simon}, whose terminology we follow.} When forking and dividing coincide over models, the first notion has a purely combinatorial characterization; $\phi(x,b)$ is f-generic if and only if it does not `$G$-divide', where a subset $D$ of an abstract group $G$ is said to $G$-divide if there is some $k\in\omega$ and an infinite sequence $(g_i:i\in\omega)$ of elements of $G$ such that $\bigcap_{i\in s}g_iD=\varnothing$ for all $s\subseteq[n]$ of a size $k$. The notion of `strong f-genericity', however, does not have such a combinatorial description, and it involves the model-theoretic notion of forking quite essentially.

In a stable group, a definable set is generic if and only if it is f-generic if and only if it is strong f-generic, and the non-generic sets form an ideal. A fundamental result of Pillay from \cite{pillay_simple} is that, for a group definable in a simple theory, a definable set is f-generic if and only if it is strong f-generic, and the non-f-generic sets form an ideal.{\footnote[1]{Disclaimer: in \cite{pillay_simple} and in most of the literature on simple theories, what we are calling `strong f-generic' is referred to as `generic', but in modern hindsight/terminology it is more suitable to reserve the term `generic' for definable sets finitely many translates of whom cover the group. Note that definable groups in simple theories need not admit generic types in this sense; for example if $F$ is a pseudofinite field, then the set $S=\{a^2:a\in F\}$ of squares is a definable subset of $(F,+)$ such that neither it nor its complement is generic.} On the other hand, for an NIP group, a fundamental insight of Hrushovski and Pillay from \cite{hrushovski_pillay} is that the good behavior of strong f-genericity is equivalent to definable amenability; this insight was very fruitfully developed upon by Chernikov and Simon in \cite{chernikov_simon} to resolve a number of open problems on definably amenable NIP groups. Altogether, by results from \cite{hrushovski_pillay}, \cite{chernikov_simon}, and my own paper \cite{stonestrom}, an NIP group is definably amenable if and only if the non-f-generic sets form an ideal if and only if the non-strong-f-generic sets form an ideal. On the other hand, even in the definably amenable NIP setting, f-genericity and strong f-genericity need not coincide, as shown in \cite{chernikov_simon}.

Of a slightly different flavor, insofar as it does not deal with model-theoretic `dividing lines' like simplicity or NIP, is Hrushovski's notion of an S1 translation-invariant ideal, most famously applied in the paper \cite{hrushovski_stable_groups} to obtain enormous breakthroughs in the study of finite approximate groups. Translation-invariance of an ideal of definable subsets of a group $G$ means the obvious thing: if $\phi(x,b)$ is small then so is $g\phi(x,b)$ for all $g\in G$. The S1 condition is a bit more subtle; we say that the ideal is S1 if, for any $b_1,b_2$ starting an indiscernible sequence, if $\phi(x,b_1)\wedge\phi(x,b_2)$ is small then so is $\phi(x,b_1)$. (This is basically the `chain condition' from simple theories.) Existence of an S1 translation-invariant ideal has significant consequences and applications; alongside \cite{hrushovski_stable_groups}, some examples of different flavors can be found in \cite{montenegro_onshuus_simon} and in \cite{palacin}, \cite{amador_daniel_1}, and \cite{amador_daniel_2}.

On the other hand, it is easy to see that an S1 ideal must contain every forking formula (Lemma 2.9 in \cite{hrushovski_stable_groups}), and so a translation-invariant S1 ideal must contain every formula that is not strong f-generic. In particular, if there do not exist strong f-generic types, then there cannot exist translation-invariant S1 ideals, and so in particular in a non-definably amenable NIP group there is no hope of having a translation-invariant S1 ideal. Nevertheless, in this paper, we will show that in any NIP group there are a pair of ideals, one translation-invariant and one S1, which have strong applications. %(Outside NIP, our hopes here are also partly inspired by the aforementioned paper \cite{montenegro_onshuus_simon} of Montenegro, Onshuus, and Simon, which shows that, in an NTP$_2$ group admitting strong f-generic types, the non-strong-f-generic formulas form an S1 translation-invariant ideal, and combines this with a generalization of the `stabilizer theorem' of \cite{hrushovski_stable_groups} to obtain structural results on groups with strong f-generics in bounded PRC fields.)

With this context we can now state our first main result. \uline{Suppose for notational convenience that $\dcl(\varnothing)=M_0$ is a model}, so that a definable subset of a group $G$ is strong f-generic if and only if no left translate of it forks over $M_0$. Motivated by the terminology of a `piecewise syndetic' set in ergodic theory, let us call a definable set `piecewise (strong) f-generic' if a union of finitely many translates of it is (strong) f-generic, and `piecewise strong f-generic with witnesses in $M_0$' if a union of finitely many translates of it by elements of $G(M_0)$ is strong f-generic. Now our first main theorem can be stated as follows:

\begin{theorem1} Let $G$ be an NIP group. Then:
    \begin{enumerate}
    \item The definable subsets of $G$ that are not piecewise f-generic form a translation-invariant ideal. In fact, a definable set is piecewise f-generic if and only if it is `weakly generic' in the sense of Newelski, meaning that its union with some non-generic definable set is generic.
    \item The definable subsets of $G$ that are not piecewise strong f-generic form a translation-invariant ideal.
    \item The definable subsets of $G$ that are not piecewise strong f-generic with witnesses in $M_0$ form an S1 ideal.
    \end{enumerate}
\end{theorem1} (1) is Proposition \ref{piecewise f-genericity coincides with weak genericity}, (2) is Theorem \ref{piecewise strong f-generic sets are an ideal}, and (3) is Theorem \ref{non-wide sets are an s1 ideal}. Though not explicitly observed anywhere in the literature, (1) is basically an immediate consequence of results already in the literature, in particular of arguments from Section 3.3 of \cite{chernikov_simon}. On the other hand, (2) and (3) are very new results. The proofs are a bit involved, but involve a number of techniques that we hope will be of interest and use more generally; a key role is played by the properties of non-forking in NIP theories, and in Question \ref{ntp2 question}, which we hope may be of further interest, we ask whether our results also hold for NTP$_2$ groups.

In general, our hope is that the notions above can be a valuable tool for studying arbitrary NIP groups, just as f-genericity and strong f-genericity are valuables tools for studying definably amenable NIP groups. To this end, we give two applications of our machinery, both to the study of the `Ellis groups' of an NIP group, which we describe below. In our applications, the purely combinatorial notion of piecewise f-genericity (equivalently, by (1), weak genericity) is too weak, and we will really need the stronger model-theoretic machinery from (2) and (3). Our perspective here is that the ideals identified in (2) and (3) should together serve as the substitute for translation-invariant S1 ideals in the study of NIP groups. Due to the cumbersomeness of the terminology, throughout the paper \uline{we will use the word `wide' to abbreviate `piecewise strong f-generic with witnesses in $M_0$'}.

Let us now describe our two main applications. One highly fruitful approach to the study of definable groups has been through the lens of topological dynamics, a perspective first introduced by Newelski in \cite{newelski_1} and \cite{newelski_2}. Recall that, if $G$ is a (discrete) group, then a $G$-flow is an action of $G$ by homeomorphisms on a non-empty compact Hausdorff space $X$; we denote the flow by $(G,X)$. A basic construction in topological dynamics is the `Ellis semigroup' $E(G,X)$ of the flow $(G,X)$, which is the closure of the set $\{\pi_g:g\in G\}$ in the space of functions $X^X$, where $\pi_g:X\to X$ is the map $x\mapsto gx$ and $X^X$ is equipped with the Tychonoff topology. $E(G,X)$ is naturally a $G$-flow and a `left-topological semigroup' under composition, and on general grounds it contains a minimal left ideal $\mathcal{I}$ (equivalently, a minimal $G$-subflow). Again on general grounds, $\mathcal{I}$ contains an idempotent $u$, and the set $u\mathcal{I}$ is a group. Moreover, by a remarkable construction of Ellis, this group can be endowed with a natural topology, called the $\tau$-topology, which coarsens the subspace topology inherited from $X^X$ and with respect to which $u\mathcal{I}$ becomes a $T_1$ quasicompact `semitopological group', where `semitopological' means that the group operation is continuous in each variable. The isomorphism type of this semitopological group is independent of the choice of $u$ and $\mathcal{I}$, and in model theory we call it the `Ellis group' of the flow $(G,X)$. (We give here an obligatory warning that this is not standard terminology in topological dynamics.)

Newelski observed that, given a theory $T$, a definable group $G$ of $T$, and a model $M\models T$, there is a naturally associated flow $(G(M),S_G(M))$ of the (discrete) group of $M$-points $G(M)$ acting by translations on the (compact) space of complete types $S_G(M)$, and initiated a program to study the definable group $G$ by studying the topological dynamics of this flow. In many applications, it is more natural to study a different model-theoretic $G(M)$-flow, which is the space $S^{\mathrm{fs}}_G(\mathfrak{C},M)$ of complete types over a saturated model $\mathfrak{C}$ concentrated on $G$ and finitely satisfiable in $M$. The advantage of working with this flow is that it is canonically isomorphic as a $G(M)$-flow to its own Ellis semigroup. (The semigroup operation is given by pushing forward the `Morley product' under the group operation, so for any $M\preccurlyeq N\prec\mathfrak{C}$ we have $(pq)|_N=\tp(ab/N)$ where $b\models q|_{N}$ and $a\models p|_{(N,b)}$.) The Ellis group of this flow now gives a $T_1$ quasicompact semitopological group associated to $G$ and $M$. Throughout this paper, we will use some convenient but non-standard terminology, and call this group the `Ellis group of $G(M)$'; we will denote it $u\mathcal{I}$, where $\mathcal{I}$ is a minimal ideal of $S^{\mathrm{fs}}_G(\mathfrak{C},M)$ and $u\in\mathcal{I}$ is an idempotent.

On the other hand, there is another abstract operation that associates a compact (Hausdorff) topological group to $G$ and $M$, introduced by Pillay in \cite{pillays_conjecture}: taking the quotient $G(\mathfrak{C})/G^{00}_M(\mathfrak{C})$ and giving it the so-called `logic topology', where $G^{00}_M(\mathfrak{C})$ is the smallest subgroup of $G(\mathfrak{C})$ type-definable over $M$ and of `bounded index', ie of index smaller than the degree of saturation of $\mathfrak{C}$. The `logic topology' on $G(\mathfrak{C})/G^{00}_M(\mathfrak{C})$ is given by taking a set to be closed if and only if its preimage under the quotient map $G(\mathfrak{C})\to G(\mathfrak{C})/G^{00}_M(\mathfrak{C})$ is type-definable over some small set. One can show that this makes $G(\mathfrak{C})/G^{00}_M(\mathfrak{C})$ is a compact Hausdorff topological group, and that the isomorphism type of this group does not depend on the choice of monster model $\mathfrak{C}$, so we just denote it by $G/G^{00}_M$. `Pillay's conjecture', from \cite{pillays_conjecture}, already referred to above, was that, when $T$ is o-minimal, $G/G^{00}_M$ is a real Lie group, and when $G$ is additionally `definably compact' then the dimension of $G/G^{00}_M$ as a Lie group coincides with the `o-minimal dimension' of $G$. (The first part of the conjecture was solved soon afterwards in \cite{berarducci_otero_peterzil_pillay}, while the second part was not solved fully until \cite{hrushovski_peterzil_pillay}; we refer to the excellent survey \cite{peterzil} for a much more thorough description of the history.)

A fundamental insight of Newelski was to suspect a connection between the Ellis group of $G(M)$ and the group $G/G^{00}_M$. He showed the existence of a canonical epimorphism $u\mathcal{I}\to G/G^{00}_M$, and conjectured that, at least in good circumstances (for example in NIP theories), this map should be an isomorphism. A natural counterexample in an NIP theory was found in \cite{pillay_penazzi_gismatullin}, which showed that the Ellis group of $G=\mathrm{SL}_2$ in $M=(\mathbb{R},+,\cdot)$ is cyclic of order two, despite the fact that $G/G^{00}_M$ is trivial. (Similarly, it was later shown in \cite{pillay_penazzi_yao} that the Ellis group of $G=\mathrm{SL}_2$ in $M=(\mathbb{Q}_p,+,\cdot)$ is isomorphic to $(\widehat{\mathbb{Z}},+)$, the profinite completion of $(\mathbb{Z},+)$, despite the fact that $G/G^{00}_M$ is trivial.) Pillay thus revised Newelski's conjecture to ask whether it was true for \textit{definably amenable} NIP groups. Positive answers in a number of special cases were proven in several papers, notably \cite{chernikov_pillay_simon}, and the question was eventually positively solved in its entirety by Chernikov and Simon in \cite{chernikov_simon}; this is one of several major theorems from that paper.

\begin{fact}\label{newelski conjecture for definably amenable nip groups}
    Suppose that $T$ is NIP and that $G$ is definably amenable. For any model $M$, the Ellis group of $G(M)$ is isomorphic to $G/G^{00}_M$.
\end{fact}

For arbitrary NIP groups, not definably amenable ones, there are several ways of weakening Newelski's conjecture. A very natural one is motivated by a result of Shelah from \cite{shelah_g00}, which shows that, in an NIP theory, $G^{00}_M=G^{00}_N$ for any models $M$ and $N$. The (in my view slightly confusing) terminology for this is to say that `$G^{00}$ exists', and we then just write $G^{00}$ to mean the common $G^{00}_M$. So, by Fact \ref{newelski conjecture for definably amenable nip groups}, when $G$ is definably amenable and $T$ is NIP the Ellis groups of $G(M)$ and $G(N)$ are both isomorphic to $G/G^{00}=G/G^{00}_M=G/G^{00}_N$. So one can ask if, even without a definable amenability assumption, one gets an analogous result for arbitrary NIP groups:
\begin{question}\label{newelski question}
    Suppose that $T$ is NIP and that $M\preccurlyeq N$ are models. Are the Ellis groups of $G(M)$ and $G(N)$ isomorphic?
\end{question} To my knowledge this question has appeared independently in three places. The history is a bit complicated, and our approach here is very different than than any of the other papers in the literature that deal with it, so I defer to Section \ref{history of question} below for a description of the history and a survey of the currently known results. For now, let me just remark that, by results of Jagiella, Yao, and Zhang in \cite{yao_long}, \cite{jagiella}, and \cite{yao_zhang}, a positive answer is known when $M$ is an o-minimal expansion of $\mathbb{R}$ and when $M=\mathbb{Q}_p$.

Another way of weakening Newelski's conjecture, posed by Krupi\'nski and Pillay as Conjecture 5.3 in \cite{krupinski_pillay}, is to ask whether, if $T$ is NIP, the $\tau$-topology on the Ellis group of $G(M)$ is Hausdorff. Krupi\'nski and Pillay were motivated by an earlier theorem of their own from \cite{krupinski_pillay_early}, in which they gave a refinement of Newelski's epimorphism from the Ellis group of $G(M)$ to $G/G^{00}_M$. Letting $u\mathcal{I}$ denote the Ellis group of $G(M)$, \cite{krupinski_pillay_early} showed that in arbitrary theories there are (canonical) group epimorphisms and topological quotient maps $$u\mathcal{I}\to u\mathcal{I}/H(u\mathcal{I})\to G/G^{\infty}_M\to G/G^{00}_M,$$ where $u\mathcal{I}/H(u\mathcal{I})$ is the maximal Hausdorff quotient of $u\mathcal{I}$ and $G^{\infty}_M$ is the smallest bounded-index subgroup of $G$ \textit{invariant} over $M$. This result was an important breakthrough in understanding the relation between the Ellis group of $G(M)$ and the group $G/G^{00}_M$. Among other things it implies that any case where any one of these arrows is not an isomorphism gives a counterexample to Newelski's conjecture. The original counterexamples to Newelski's conjecture give examples where the middle arrow is not an isomorphism, and the first example where the last arrow is not an isomorphism had been given by Conversano and Pillay in \cite{conversano_pillay}, who showed that $\widetilde{\mathrm{SL}_2(\mathbb{R})}$ is such an example. Now Conjecture 5.3 from \cite{krupinski_pillay}, which in later work became known as the `revised Newelski conjecture', asked whether, in an NIP theory, at least the first arrow in this sequence is an isomorphism.

Chernikov, Gannon, and Krupi\'nski positively solved this question for countable $M$ in \cite{chernikov_gannon_krupinski}, and using tools from set theory Basso and Zucker extended this to arbitrary $M$ in \cite{basso_zucker}. In fact, both \cite{chernikov_gannon_krupinski} and \cite{basso_zucker} prove the analogous theorems, on Hausdorffness of the $\tau$-topology, for the more general class of `tame minimal flows' (metrizable ones in \cite{chernikov_gannon_krupinski} and arbitrary ones in \cite{basso_zucker}), where following \cite{glasner_tameness} a flow $(G,X)$ is said to be `tame' if, for any continuous real-valued map $f\in C(X)$, and for any sequence $g_i\in G,i\in\omega$, the sequence $(f\circ g_i:i\in\omega)$ in the Banach space $(C(X),||\cdot||_\mathrm{sup})$ is not an `$\ell^1$-sequence' in the sense of \cite{rosenthal}. An essential observation, first noted in \cite{chernikov_simon}, is that, if $T$ is NIP, then the $G(M)$-flows $S_G(M)$ and $S_G^{\mathrm{fs}}(\mathfrak{C},M)$ are tame, and the results of \cite{chernikov_gannon_krupinski} and \cite{basso_zucker} on tame minimal flows yield the results on NIP groups by virtue of that connection.

Though the `revised Newelski conjecture' has been resolved, the other weakening of Newelski's conjecture, Question \ref{newelski question}, remains open. Our first major application of Theorem 1 is to give a substantial step in the direction of a positive answer.

\begin{theorem2}
    Suppose that $T$ is NIP and that $M_0$ is a fixed model of size $|T|$. For any elementary extension $M\succcurlyeq M_0$, let $E$ be the Ellis semigroup $S^{\mathrm{fs}}_G(\mathfrak{C},M)$. Then the restriction map $E\to S_G(M_0)$ taking a global type $p$ to its restriction $p|_{M_0}$ is injective on any Ellis group $u\mathcal{I}\subseteq E$. In particular, $|u\mathcal{I}|\leqslant 2^{|T|}$, regardless of the choice of model $M$.
\end{theorem2} This is Theorem \ref{restriction map ellis group theorem}. Its proof relies heavily on the machinery of piecewise strong f-genericity developed in Theorem 1, just as the proof of Fact \ref{newelski conjecture for definably amenable nip groups} in \cite{chernikov_simon} relied heavily on the machinery of strong f-genericity, and also uses the positive solution to the `revised Newelski conjecture'. Theorem 2, despite giving substantial structural information about the Ellis group, is still not enough for us to answer Question \ref{newelski question}, and in particular we do \textit{not} obtain a map from the Ellis group of $G(M)$ to that of $G(M_0)$. Nevertheless Theorem 2 does show that the Ellis groups of $G(M)$ cannot vary too much even as $M$ varies. Note that this fails even in some of the simplest examples outside NIP, due to the epimorphism from the Ellis group of $G(M)$ to $G/G^{00}_M$. Indeed, whenever $G^{00}$ does not `exist', then on general grounds $G/G^{00}_M$ becomes arbitrarily large as $M$ becomes arbitrarily large, and so the Ellis group of $G(M)$ must then become arbitrarily large too. %Even in some of the simplest examples of non-NIP groups $G^{00}$ will not `exist'; one natural example is the so-called `extraspecial $p$-group'.

An important technical tool that we use in the proof of Theorem 2 is the existence in NIP theories of a certain `retraction map' $F_M:S^{\mathrm{inv}}(\mathfrak{C},M)\to S^{\mathrm{fs}}(\mathfrak{C},M)$ from the space of global types invariant over $M$ to the space of global types finitely satisfiable in $M$; see Section \ref{honest definitions section} for the definition. The retraction map $F_M$ is what will allow us to transfer results about global wide types, which live naturally in $S^{\mathrm{inv}}_G(\mathfrak{C},M_0)$, to the Ellis group of $G(M)$, which lives in $S^{\mathrm{fs}}_G(\mathfrak{C},M)$. This idea is influenced by techniques from all three of the papers \cite{chernikov_pillay_simon}, \cite{chernikov_simon}, and \cite{kyle_tomasz}, although it differs significantly from all of them in that we are attempting to compare Ellis groups across different models. In some sense the use is closest to the proof of Fact \ref{newelski conjecture for definably amenable nip groups} in \cite{chernikov_simon},\footnote[1]{In the proof of Theorem 5.7 in \cite{chernikov_simon} Chernikov and Simon do not actually explicitly use the language of the retraction map. Instead they work over a fixed model $M_0$ and assume that $M_0$ is equal to its Shelah expansion; in this case the retraction map $F_{M_0}:S^{\mathrm{inv}}(\mathfrak{C},M_0)\to S^{\mathrm{fs}}(\mathfrak{C},M_0)$ is quite transparent, and just takes a global $M_0$-invariant type $q$ to the unique global coheir of $q|_{M_0}$. However, their proof would work exactly the same without assuming that $M_0$ is equal to its own Shelah expansion by replacing every instance of $q|_{M_0}$ by $F_{M_0}(q)$. In this language, the proof would proceed by working with a strong f-generic type $p$ such that $F_{M_0}(p)$ is almost periodic in the Ellis semigroup $S^{\mathrm{fs}}_G(\mathfrak{C},M_0)$. In our case, we will work with a model $M\succcurlyeq M_0$ and apply $F_M$ to a global wide type; a wide type $p$ is $M_0$-invariant and hence in particular $M$-invariant, so it makes sense to apply the retraction map to it, but we caution that if $M$ is a proper elementary extension of $M_0$ then $F_M(p)\neq F_{M_0}(p)$ unless $p$ is finitely satisfiable.} but my inspiration to use the language of the retraction map, which turns out to be very convenient, comes from the recent paper \cite{kyle_tomasz} of Gannon and Rzepecki.\footnote[2]{See Remark \ref{invariant ellis group remarks} for some brief speculation on how one might use the machinery from Theorem 1 if one wanted to study the `invariant Ellis subgroups', in the sense of \cite{kyle_tomasz}, for a non-definably-amenable NIP group. As in \cite{kyle_tomasz} and the related paper \cite{pillay_invariant_ellis_group}, the appropriate notion there would be `right' piecewise f-genericity.}

Our second major application gives a very strong structural result on the Ellis group of $G(M_0)$, assuming that $T$ and $M_0$ are countable and that $T$ satisfies a strengthening of NIP that one might call `bounded VC-codensity': this condition means that, for any tuple of variables $x$, there is some $d$ depending only on $x$ such that the VC-codensity of every formula $\phi(x,y)$ is at most $d$. (See Section \ref{vc theory section} for definitions.) This is quite a natural strengthening of NIP, and by significant results in \cite{five_authors} and \cite{basu_patel} it is satisfied by o-minimal theories, the theories of the $p$-adic fields, and the theories of algebraically closed valued fields. See also \cite{guingona_laskowski}, which studies related notions and gives other examples. We prove then the following:

\begin{theorem3}Suppose that $T$ and $M_0$ are countable and that $T$ has bounded VC-codensity. Then the Ellis group of $G(M_0)$ has `finite Archimedean rank', meaning that it is profinite-by-Lie-by-profinite.
\end{theorem3} This is Theorem \ref{FAR main theorem}. Unsurprisingly, our proof draws heavily on Hrushovski's techniques from Section 7 of \cite{hrushovski}. In Theorem 7.12 there, Hrushovski shows that, in an arbitrary theory, given an NIP formula $\phi$ and a definable measure $\mu$ on $\phi$-formulas with parameters from a single Shelah strong type, then a certain natural associated compact automorphism group $\mathcal{G}_{\mu,\phi}$ has finite Archimedean rank. The definition of this group is a bit involved and we do not give it here, but by the standard technique of adding a new sort for a right torsor Hrushovski's result yields that, given a definable group $G$ and an NIP formula $\phi(x,y)$ such that $G$ is definably amenable with respect to $\phi$-formulas, ie admits a definable translation-invariant Keisler measure $\mu$ on the Boolean algebra of definable subsets of $G$ generated by the instances of $\phi$, certain quotients of $G$ by bounded-index subgroups associated to $\phi$ and $\mu$ are of finite Archimedean rank. (Any individual NIP formula has bounded VC-codensity, so in Hrushovski's case it is enough to assume just that the single formula $\phi$ is NIP. Our need to assume uniformly bounded VC-codensity comes from the fact that we are proving a theorem about a `global' object, the Ellis group, not about a `local' one; in Section \ref{local g00 section} we will give a `local' result that is closer in spirit to the pursuits in \cite{hrushovski}, and we will state the connections with the results and questions from \cite{hrushovski} there.)

Many of our key ideas in the proof of Theorem 3 come from Hrushovski's proof of Theorem 7.12 in his paper. Nevertheless, there are some important differences in approach and methods. The most notable one is that we do not make any definable amenability assumption, whereas Hrushovski's theorem assumes the presence of the invariant measure $\mu$. We avoid this assumption through one of the core ideas of Chernikov, Gannon, and Krupi\'nski from \cite{chernikov_gannon_krupinski}. In the latter paper, as mentioned above, the authors prove that, assuming NIP, the $\tau$-topology on the Ellis group is Hausdorff; in particular the Ellis group is a compact Hausdorff topological group, and so admits (finite) Haar measure. A key insight from \cite{chernikov_gannon_krupinski} is that the Haar measure on the Ellis group of a non-definably amenable NIP group can, for certain applications, replace the translation-invariant measure of a definably amenable NIP group, and the paper \cite{chernikov_gannon_krupinski} develops much useful machinery for applications of this kind, which we will rely critically on here. 

The ideas of \cite{hrushovski} and \cite{chernikov_gannon_krupinski} are a crucial starting point and inspiration for us, but we also require a few new insights in the proof of Theorem 3. One of the main ones is to find an appropriate substitute for a metric space canonically associated to the NIP formula $\phi$ and the definable measure $\mu$ that appears in Hrushovski's proof; for this purpose, we will use certain pseudometrics naturally associated to Borel subsets of a compact group. Another important new insight is to work not with the Haar measure on the Ellis group but rather with the Haar measure on its connected component; this is essentially how we eliminate the analogue of the need that appears in Hrushovski's theorem to restrict to a single Shelah strong type. In fact, we isolate these two main ideas of ours to obtain a completely general theorem about VC-sets in compact Hausdorff groups, which is Theorem \ref{general far theorem} in Section \ref{vc sets far section} and which may be of broader interest:
\begin{theorem4}Suppose $G$ is a compact Hausdorff group, and suppose there is a countable family $(B_i:i\in I)$ of Borel subsets of $G$ such that (1) every open subset of $G$ can be written as a union of some of the $B_i$, and (2) there is some $\delta$ such that, for each $i\in I$, the family of left translates $(gB_i:g\in G)$ has VC-density at most $\delta$. Then $G$ is an inverse limit of compact Lie groups of dimension at most $(4\delta)^2$.
\end{theorem4} Theorem 4 is one of the key ingredients in the proof of Theorem 3. Now, the most obvious major difference between Theorem 3 and Hrushovski's results is that we work with the Ellis group rather than with $G/G^{00}$. Despite being a more general setting, insofar as $G/G^{00}$ is a quotient of the Ellis group, once one has access to Haar measure on the Ellis group it is actually significantly more convenient to work with the Ellis group than to work with $G/G^{00}$, since the Ellis group is a set of types and so the VC-theoretic consequence of NIP are very transparent (see Lemma \ref{theta is vc set in ellis group}). This makes it easy to satisfy condition (2) in the hypotheses of Theorem 4, and it is not clear if there is an analogue in $G/G^{00}$; see right after Theorem \ref{local g/g00 main theorem} for some discussion of this. On the other hand, there are also significant difficulties that come from working with the Ellis group that do not arise in working with $G/G^{00}$, a critical one of which is that the $\tau$-topology on the Ellis group of $G(M_0)$ will in general not be second-countable even when $M_0$ is a countable model of a countable theory. One major consequence of our Theorem 2 is that, if $T$ is countable and NIP and $M_0$ is a countable model, then the $\tau$-topology on the Ellis group of $G(M_0)$ has a well-behaved and explicit countable `basis' of Borel subsets (with `basis' in quotation marks because the sets are not open); see Corollary \ref{tau-open is countable union}. This plays an essential role in our proof of Theorem 3, and is what allows us to satisfy condition (1) in the hypotheses of Theorem 4. The need to have access to Theorem 2 is part of the reason we need to make a global NIP assumption, unlike in Hrushovski's case where he can work locally; see Question \ref{local ellis group question} for more on this. %See the end of Section \ref{FAR section} for some discussion of how one might try to prove a similar result for the `local Ellis group' associated to a bi-invariant NIP formula $\phi$, as developed in \cite{devrim}, though to me it does not seem like it would be a straightforward adaptation.

In the last section of the paper, we use our techniques from the proof of Theorem 3 to prove a result for arbitrary NIP theories, not just ones of bounded VC-codensity. Given a definable group $G$, say that a formula $\phi(x,y)$ is `bi-invariant' if $\phi(x,y)$ implies $x\in G$ and any left or right translate of an instance of $\phi$ by an element of $G$ is equivalent to another instance of $\phi$. By a $\phi$-formula we mean a Boolean combination of instances of $\phi$. A folklore result is that, if $\phi(x,y)$ is bi-invariant and NIP, then there is a smallest bounded-index subgroup of $G$ type-definable by $\phi$-formulas, which is normal in $G$ by bi-invariance of $\phi$, and which we will call $G^{00}_\phi$. (For pseudofinite $G$ this object is thoroughly studied in work of Conant and Pillay in \cite{conant_pillay}. That paper develops a theory of what might be called `local fsg', which then plays an important role in the `NIP arithmetic regularity' theorem of Conant, Pillay, and Terry \cite{conant_pillay_terry}.) Using our techniques from the proof of Theorem 3 we obtain the following:
\begin{theorem5}
Suppose that $T$ is NIP and that $\phi(x,y)$ is a bi-invariant formula. Then $G/G^{00}_\phi$ has finite Archimedean rank. More precisely, if the VC-codensity of $\phi(x,y)$ is at most $\delta$, then $G/G^{00}_\phi$ is an inverse limit of compact Lie groups of dimension at most $(4\delta)^2$.
\end{theorem5}This is Theorem \ref{local g/g00 main theorem}. We give much further discussion in Section \ref{local g00 section}; in particular we give the precise connections with the results from Section 7 of \cite{hrushovski} and discuss connections with results of Pekmezci from \cite{devrim}. In particular Theorem 4 connects with Question 7.18(1) from \cite{hrushovski}, though it is different than it in a few ways, and we discuss this in detail.

Theorems 1, 2, 3, 4, and 5 are the main results of this paper, with each one corresponding to its own dedicated section. Let me now make one last remark to conclude. I originally developed much of the material in this paper as part of a project to give a structure theory for approximate groups definable in NIP structures; with a bit of additional care needed in the proofs, all of the material of Section \ref{piecewise f-generic section}, Section \ref{existence of g000 section}, and Section \ref{main consequence for ellis group section} goes through in the setting of NIP approximate groups, ie in the case where $G$ is not a definable group but rather a $\bigvee$-definable group $\langle X\rangle$ generated by some approximate group $X$ definable in an NIP structure. The beautiful paper \cite{krupinski_pillay} of Krupi\'nski and Pillay develops Ellis theory for groups of this form $\langle X\rangle$; here the Ellis groups will not be (quasi)compact but rather \textit{locally} (quasi)compact semitopological groups, and Krupi\'nski and Pillay show that the maximal Hausdorff quotients of these Ellis groups are exactly the universal `generalized locally compact models' for $X$ in the sense of \cite{hrushovski_lascar_group}. One of my major goals in my project on NIP approximate groups is to give some structural results for the Ellis groups in this setting; this is also closely related to Question 7.18(2) of \cite{hrushovski}. It was in working on this aim that I realized the machinery of Theorem 1 could also have strong and significant consequences for the Ellis groups of definable groups, and that is the origin of this paper here. The project on NIP approximate groups is still ongoing, and far from complete, so I do not include the material from it in this paper. Nevertheless throughout Section \ref{piecewise f-generic section} and Section \ref{existence section} I will include a bit of extra material, which is not necessary or relevant for the results of this paper, but which will be useful to have as a reference in the paper on NIP approximate groups; I indicate this material explicitly throughout. I will also already remark here that, again with some additional care needed in the proofs, the appropriate analogues of Theorem \ref{restriction map ellis group theorem} and Corollary \ref{tau-open is countable union} hold in the maximal Hausdorff quotient of the Ellis group of an NIP approximate group, and will play an important role in the work there.

\subsection{History of Question \ref{newelski question}}\label{history of question}
Let us discuss Question \ref{newelski question} in greater detail; the history is a bit elaborate and so we separate this discussion from the rest of the introduction. We will not actually use any of the results or work mentioned in this section, but they give some context for our Theorem 2. Question \ref{newelski question} clearly has roots in Newelski's paper \cite{newelski_2}, in which Newelski addresses the question of how, in an arbitrary theory, the Ellis groups of $G(M)$ and $G(N)$ relate to one another for $M\preccurlyeq N$. In fact, Newelski wishes to ask a more precise question about the relationship between the Ellis groups of $G(M)$ and $G(N)$, and for it to make sense he needs to put a technical assumption on the relationship between $M$ and $N$ in order to have a natural map $S^{\mathrm{fs}}_G(\mathfrak{C},N)\to S^{\mathrm{fs}}_G(\mathfrak{C},M)$. We will not state the assumption here, though in the case where $M$ is equal to its own `Shelah expansion', meaning that every type over $M$ is definable, the assumption is satisfied for any $N\succcurlyeq M$. (In particular this applies if $M$ is an o-minimal expansion of $\mathbb{R}$ or if $M=\mathbb{Q}_p$.)

To my knowledge, Question \ref{newelski question} appears independently in three places. On the one hand, in 2018, Jagiella published the paper \cite{jagiella}, Conjecture 1.6 of which is exactly Question \ref{newelski question}. (Jagiella attributes the question `generally' to \cite{newelski_2}, though in fact \cite{newelski_2} does not discuss general NIP theories at all.) On the other hand, after I finished writing the present paper, Krzysztof Krupi\'nski informed me that he and Pierre Simon had, around the same time, independently formulated this question and made some partial progress on it, advertising the question in Paris at the 2018 IHP semester on model theory. Finally, much more recently, perhaps unaware of the above history, Basso and Zucker independently restated Question \ref{newelski question} as Question 11.10 in \cite{basso_zucker}.

Let me summarize the known results on Question \ref{newelski question} that I am aware of. First, Krupi\'nski has shared with me the nature of his results with Simon, which are not public; among other things, for appropriate $M\preccurlyeq N$, they constructed under NIP a semigroup homomorphism $S^{\mathrm{fs}}_G(\mathfrak{C},M)\to S^{\mathrm{fs}}_G(\mathfrak{C},N)$ which is a section of Newelski's above-mentioned map $S^{\mathrm{fs}}_G(\mathfrak{C},N)\to S_G^{\mathrm{fs}}(\mathfrak{C},M)$. They hoped to be able to use this map to attack Question \ref{newelski question}. Nonetheless the progress there is only partial, and in particular they were not able to show boundedness of the size of the Ellis group of $G(M)$, which is the main point of our Theorem 2. The work of Krupi\'nski and Simon is the only work I am aware of that attempts to answer Question \ref{newelski question} for arbitrary NIP groups.

Jagiella's paper \cite{jagiella} is of a very different flavor, and is an instance of approach (2) from the first paragraph of the introduction. Motivated by the examples of $\mathrm{SL}_2$ discussed above, and more generally by the Iwasawa decompositions of semisimple Lie groups, Jagiella in \cite{jagiella} uses Fact \ref{newelski conjecture for definably amenable nip groups} to make progress on Question \ref{newelski question} for NIP groups $G$ that satisfy a certain decomposition as a Zappa-Szép product of an fsg subgroup $K$ and a definably extremely amenable subgroup $H$. Jagiella gives quite an explicit description of the Ellis group of a group of this kind, showing that it is isomorphic to a subgroup of $K/K^{00}$. As is the case with our Theorem 2, and to an even more pronounced degree, despite getting very strong structural information on the Ellis group in his setting, Jagiella's methods nonetheless do not manage to resolve Question \ref{newelski question} there (see the discussion around Question 4.15 in his paper).

There is also a significant body of work on Question \ref{newelski question} in the spirit of approach (1) from the first paragraph of the introduction. In the o-minimal case, Jagiella is able to combine his techniques in \cite{jagiella} with results of Yao from \cite{yao_long} and with the structure theory on groups in o-minimal structures developed by Conversano in \cite{conversano_1} and \cite{conversano_2} to give a positive answer to Question \ref{newelski question} whenever $M$ is an o-minimal expansion of $\mathbb{R}$. More recently, in \cite{yao_zhang}, Yao and Zhang have given a positive answer to Question \ref{newelski question} for $M=\mathbb{Q}_p$. To my knowledge the papers \cite{jagiella} and \cite{yao_zhang} summarize all of the theories where a positive answer to Question \ref{newelski question} is known, although as discussed in \cite{yao_zhang} it seems one would expect analogues for $M=\mathbb{C}((t))$, following the work of Kirk \cite{kirk}.

\section{Preliminaries}
As in the rest of the paper, \uline{assume throughout this section that $T$ is NIP}, expect when explicitly stated otherwise. We organize the preliminaries in order of how they appear in the paper; \ref{forking nip section} and \ref{connected components section} are the main preliminaries for Section \ref{piecewise f-generic section}, \ref{honest definitions section} and \ref{topological dynamics section} are the main preliminaries for Section \ref{existence section}, and \ref{vc theory section} and \ref{compact groups section} are the main preliminaries for Section \ref{FAR section}.

\subsection{Forking in NIP Theories}\label{forking nip section}
We recall here some properties of non-forking independence in NIP theories. Recall that a formula $\phi(x,b)$ with parameters from $\mathfrak{C}$ `divides' over a parameter set $A\subset\mathfrak{C}$ if there is some $k\in\omega$ and some sequence $(b_i:i\in\omega)$ such that $b_i\equiv_A b$ for all $i\in\omega$ and such that the family of formulas $(\phi(x,b_i):i\in\omega)$ is $k$-inconsistent. A partial type divides over $A$ if it implies a formula that divides over $A$, and it `forks' over $A$ if it implies a finite disjunction of formulas that each divide over $A$. We write $a\ind_A b$ to mean that $\tp(a/A,b)$ does not fork over $A$. We say that a global type $p(x)\in S(\mathfrak{C})$ is $M$-invariant over a small model $M\prec\mathfrak{C}$ if, for any tuples $b,b'$ from $\mathfrak{C}$ with $b\equiv_M b'$, we have $p(x)\vdash\phi(x,b)$ if and only if $p(x)\vdash\phi(x,b')$. Given two global $M$-invariant types $p(x)$ and $q(y)$, the Morley product $p(x)\otimes q(y)$, also denoted $p_x\otimes q_y$ or just $p\otimes q$ if the variables are clear, is defined to be the unique global $M$-invariant type whose restriction to any $N\succcurlyeq M$ is equal to $\tp(a,b/N)$ for some/any $b\models q|_N$ and $a\models p|_{(N,b)}$. The Morley product is an associative operation on $M$-invariant types. Given a global $M$-invariant type $p(x)$ and a linear order $I$, we similarly define $p^{\otimes I}$ to be the unique global $M$-invariant type in variables $(x_i:i\in I)$ with the property that, if $(a_i:i\in I)\models p^{\otimes I}|_N$ for $N\succcurlyeq M$, then $a_i\models p|_{(N,a_{<i})}$ for each $i\in I$. Finally, given a model $M$, a tuple $a$, and a parameter set $B$, we write $a\ind_M^u B$ to mean that $\tp(a/M,B)$ is a `coheir' of $\tp(a/M)$, ie that $\tp(a/M,B)$ is finitely satisfiable in $M$. We use $S^{\mathrm{inv}}(\mathfrak{C},M)$ and $S^{\mathrm{fs}}(\mathfrak{C},M)$ to refer to the spaces of global $M$-invariant and $M$-finitely-satisfiable types.

A basic fact that holds an any theory is that $\ind$ and $\ind^u$ satisfy `right extension', which is most convenient to state as follows:

\begin{fact}\label{right extension}
Let $M$ be a small model, $a$ a tuple, and $B$ a parameter set. For any cardinal $\kappa$, if $a\ind_M B$ then we may find a $\kappa$-saturated model $M'\supseteq (M,B)$ such that $a\ind_M M'$, and if $a\ind_M^u B$ then we may find a $\kappa$-saturated model $M'\supseteq (M,B)$ such that $a\ind_M^u M'$.
\end{fact}

Now let's turn to the facts specific in the NIP setting. Forking in NIP theories satisfies two fundamental properties; the first is from \citep{shelah_forking} and the second is from \citep{chernikov_kaplan}.
\begin{fact}\label{non-forking iff invariant}
    If $p(x)\in S(\mathfrak{C})$ is a global type and $M\prec\mathfrak{C}$ is a small model, then $p(x)$ does not fork over $M$ if and only if it is $M$-invariant.
\end{fact}

\begin{fact}\label{forking=dividing}
    Let $M\prec\mathfrak{C}$ be a small model. Then an $L(\mathfrak{C})$-formula forks over $M$ if and only if it divides over $M$. %In particular, $\ind^f_M=\ind^d_M$.
\end{fact}

Now, given a small model $M$, we say that an $M$-invariant global type $p(x)\in S(\mathfrak{C})$ is strictly invariant over $M$ if, for all parameter sets $C$, if $a\models p|_{(M,C)}$ then $C\ind_M a$. We use two facts about strict invariance in NIP theories; the first is from \cite{chernikov_kaplan} and the second can be found in Proposition 5.47 in \cite{simon_book}.

\begin{fact}\label{existence of strictly non-forking coheirs}
    Let $p(x)\in S(M)$ be a complete type over a small model $M$. Then there is $\tilde{p}(x)\in S(\mathfrak{C})$ a global coheir of $p(x)$ strictly invariant over $M$. (So, for any parameter set $B$, if $a\models \tilde{p}|_{(M,B)}$ then $a\ind_M^u B$ and $B\ind_M a$.)
\end{fact}

\begin{fact}\label{bounded weight}
Suppose that $q(y)$ is a global type strictly invariant over a small model $M$. Let $(b_\delta:\delta\in |T|^+)$ realize $q^{\otimes |T|^+}|_M$. Then, for any (finite-length) tuple $a$, there is some $\delta\in |T|^+$ with $a\ind_M b_\delta$.
\end{fact}

Another fundamental fact about forking in NIP theories is so-called `lowness'; the following is Remark 3.33 in \cite{chernikov_kaplan}:

\begin{fact}\label{lowness} Suppose that $\phi(x,y)$ is a formula. Then there is $N\in\omega$ such that, for any indiscernible sequence $(b_k:k\in\omega)$, if $\bigwedge_{k\in\omega}\phi(x,b_k)$ is inconsistent then $\bigwedge_{k\in[N]}\phi(x,b_k)$ is already inconsistent. In particular, for any model $M$, there is a partial type $\pi(y)$ with parameters in $M$ such that, for any $b\in\mathfrak{C}^y$, we have $b\models\pi$ iff $\phi(x,b)$ divides over $M$ (and hence by Fact \ref{forking=dividing} iff $\phi(x,b)$ forks over $M$).
\end{fact}

We need a model-theoretic generalization of the `$(p,q)$-theorem' of Alon-Kleitman and Matou\v{s}ek; it was conjectured in \cite{chernikov_simon_externally_2} and proved in the distal case in \cite{boxall_kestner} and in the general NIP case in \cite{kaplan}.

\begin{fact}\label{definable p,q theorem}
    Suppose that $M$ is a model and that $\phi(x,b)$ does not fork over $M$. Then there is an $L(M)$-formula $\eta(y)\in\tp(b/M)$ such that the partial type $\{\phi(x,b'):b'\models\eta\}$ is consistent (and hence does not fork over $M$).
\end{fact}

Finally, we also recall that $\ind$ in arbitrary theories has `left-transitivity' (see \citep{adler_forking}): for any small sets $A,B,C,D\subset\mathfrak{C}$, if $A\ind_{(B,C)} D$ and $B\ind_C D$ then $(A,B)\ind_C D$.

\subsection{Model-theoretic connected components}\label{connected components section} Let $M$ be a small model. We denote by $G^{00}_M(\mathfrak{C})$ the smallest subgroup of $G(\mathfrak{C})$ type-definable over $M$ and of `bounded index', ie of index smaller than the degree of saturation of $\mathfrak{C}$, and we denote by $G^{\infty}_M(\mathfrak{C})$ the smallest subgroup of $G(\mathfrak{C})$ `invariant over $M$', ie setwise invariant under automorphisms fixing $M$ pointwise, and of bounded index. One can think of $G^{\infty}_M$ (resp. $G^{00}_M$) as the smallest subset (resp. smallest closed subset) of $S_G(M)$ whose set of realizations $G^{\infty}_M(\mathfrak{C})$ (resp. $G^{00}_M(\mathfrak{C})$) is a subgroup of $G(\mathfrak{C})$ of index smaller than the degree of saturation of $\mathfrak{C}$, and one can show that this description is independent of the choice of monster model $\mathfrak{C}$. An explicit description of $G^{\infty}_M(\mathfrak{C})$ is as the subgroup of $G(\mathfrak{C})$ generated by $\{a^{-1}b:a,b\in G(\mathfrak{C}),a\equiv_M b\}$; see for example 8.1.4 in \cite{simon_book} for a reference.

The groups $G^{00}_M(\mathfrak{C})$ and $G^\infty_M(\mathfrak{C})$ are normal subgroups of $G(\mathfrak{C})$, and by the remark above the quotients $G(\mathfrak{C})/G^{00}_M(\mathfrak{C})$ and $G(\mathfrak{C})/G^{\infty}_M(\mathfrak{C})$ are independent of the choice of $\mathfrak{C}$, so we just write $G/G^{00}_M$ and $G/G^{\infty}_M$. One can endow the quotient $G/G^{00}_M$ with the `logic topology', where a subset is closed if and only if its preimage under the projection map $G(\mathfrak{C})\to G/G^{00}_M$ is type-definable over some small set. This makes $G/G^{00}_M$ into a compact Hausdorff topological group. On the other hand, the object $G/G^{\infty}_M$ is more mysterious; see \cite{krupinski_pillay_early} for an analysis of it.

In general, we have $G^{00}_M\supseteq G^{00}_N$ (respectively $G^{\infty}_M\supseteq G^{\infty}_N$) whenever $N\supseteq M$, but it may happen that $G^{00}_M=G^{00}_N$ (respectively $G^{\infty}_M=G^{\infty}_N$) for all $M,N$. In this case one says, somewhat confusingly, that $G^{00}$ (respectively $G^{\infty}$) `exists' and drops the subscript. From \citep{shelah_g00} and \citep{gismatullin} respectively, it is known that $G^{00}$ and $G^{\infty}$ always exist if $T$ is NIP. We will give a new proof of this result here, in Corollary \ref{existence of g000 theorem}.

\subsection{Honest definitions and the retraction map}\label{honest definitions section} An essential tool in NIP theories are the `honest definitions' of \cite{chernikov_simon_externally_1}. We will not use honest definitions explicitly, but rather a closely related feature of NIP theories, proved in \cite{simon_invariant_types}, which is the existence of a continuous retraction $F_M$ from the space of $M$-invariant global types to the space of global types finitely satisfiable in $M$, whose construction we recall here. Let $L_\mathbf{P}=L\cup\{\mathbf{P}\}$ be the language obtained by adding a single new unary predicate to $L$. Given a tuple of variables $x=(x_1,\dots,x_n)$, we write $\mathbf{P}(x)$ to abbreviate the formula $\mathbf{P}(x_1)\wedge\dots\wedge\mathbf{P}(x_n)$. Given models $M\preccurlyeq M'$ of $T$, we expand $M'$ into an $L_\mathbf{P}$-structure by interpreting $\mathbf{P}$ as the subset $M$ of $M'$, and we denote this structure by $(M',M)$.

Now fix models $M\prec^+M'$, where we mean by this that $M\prec M'$ and that $M'$ is $|M|^+$-saturated. Note that, by saturation, an $M$-invariant complete type over $M'$ extends uniquely to a global $M$-invariant type.

\begin{definition}\label{retraction construction}
    Fix an $|M'|^+$-saturated extension $(M',M)\prec^+(N',N)$ of the $L_\mathbf{P}$-structure $(M',M)$. For any global $M$-invariant $L$-type $p(x)$ from the original theory, the partial $L_\mathbf{P}$-type $p|_N(x)\wedge\mathbf{P}(x)$ implies a complete $L$-type over $M'$ that is finitely satisfiable in $M$. Denote the unique global $M$-invariant extension of this type by $F_M(p)$.
\end{definition}

\begin{fact}\label{retraction basic properties}
    $F_M$ is a retraction from $S^{\mathrm{inv}}(\mathfrak{C},M)$ to $S^{\mathrm{fs}}(\mathfrak{C},M)$; in other words, $F_M$ is continuous, and $F_M(p)=p$ whenever $p$ is finitely satisfiable in $M$. In general, we have $F_M(p)|_M=p|_M$, and $F_M$ commutes with $M$-definable functions, ie the pushforward of $F_M(p)$ under an $M$-definable function $f$ is equal to $F_M$ of the pushward of $p$ under $f$.
\end{fact}

\subsection{Topological dynamics}\label{topological dynamics section}
Here we recall the basic facts from topological dynamics that we will need. Standard references are \cite{auslander_book} and \cite{glasner_book}; however, a better reference for our purposes is Appendix A of \cite{tomasz_thesis}, which includes all of the facts that we need and uses the same language that we will use throughout this paper. In particular all of the facts cited here can be found there, and we use almost identical notation as from there; the only difference is that, instead of the standard notation $\mathcal{M},\mathcal{N}$ to denote minimal ideals, we use $\mathcal{I},\mathcal{J}$, since $M$ and $N$ are already notationally overloaded letters in our context.

Let $G$ be a discrete (abstract) group. A $G$-flow is a nonempty compact Hausdorff space $X$ together with an action of $G$ on $X$ by homeomorphisms. When the action is clear we denote the flow by $(G,X)$; a $G$-subflow of $X$ is then a closed non-empty subset of $X$ closed under the action of $G$. 

We call $X$ `point-transitive' if there is some $p\in X$ such that the orbit $Gp$ is dense in $X$. We call an open subset $U\subseteq X$ `generic' if finitely many translates of $U$ cover $X$, and we call $U$ `weakly generic' if there is a non-generic set $V$ such that $U\cup V$ is generic. We call a point $p\in X$ `weakly generic' if every open neighborhood is weakly generic. We also call an element $p\in X$ `almost periodic' if it is contained in a minimal subflow of $X$. Now the following is from \cite{newelski_1}:
\begin{fact}\label{almost periodic implies weak generic}
Let $X$ be a point-transitive $G$-flow. Then every almost periodic point of $X$ is weakly generic.
\end{fact}

The `Ellis semigroup' of the flow $(G,X)$, which we denote by $E(G,X)$ or by $E(X)$ when $G$ is clear, is the closure of $\{\pi_g:g\in G\}$ in the space $X^X$ of functions from $X$ to itself, where $\pi_g$ is the map $x\mapsto gx$ and $X^X$ is equipped with the Tychonoff topology. For the rest of this section denote $E(G,X)$ by $E$. $E$ is closed under composition of functions, and it is a `left topological semigroup' under composition, ie for any $f_0\in E$ the map $E\to E$ given by $f\mapsto ff_0$ is continuous. For $g\in G$, the map $f\mapsto\pi_g f$ is a homeomorphism of $E$, and so $E$ is naturally a $G$-flow by taking $g$ to act as $f\mapsto\pi_g f$.

By a left ideal of $E$ we mean a left ideal for the semigroup structure, ie a nonempty subset $\mathcal{I}\subseteq E$ such that $fg\in\mathcal{I}$ for all $g\in\mathcal{I}$ and $f\in E$. The minimal left ideals of $E$ correspond exactly to the minimal $G$-subflows of $E$, and if $\mathcal{I}$ is a minimal left ideal of $E$ then $\mathcal{I}=Ef=\overline{Gf}$ for all $f\in\mathcal{I}$. A fundamental theorem of Ellis is then the following:
\begin{fact}\label{ellis semigroup lemma}\begin{enumerate}
\item $E$ has minimal ideals.
\item Every minimal ideal $\mathcal{I}$ of $E$ contains an idempotent $u$, and the set $u\mathcal{I}$ is a group under the semigroup operation inherited from $E$, with identity element $u$.
\item Every minimal ideal $\mathcal{I}$ is the disjoint union of the $u\mathcal{I}$, as $u$ ranges over the idempotents in $\mathcal{I}$, and we have $qu=q$ for every $q\in\mathcal{I}$ and every idempotent $u\in\mathcal{I}$.
\item For any idempotents $u,v\in\mathcal{I}$, the map $u\mathcal{I}\to v\mathcal{I}$ defined by $x\mapsto vx$ is a group isomorphism.
\item The isomorphism type of the group $u\mathcal{I}$ is independent of the choice of $\mathcal{I}$ and (by the point above) by the choice of $u$.
\end{enumerate}\end{fact} We call these groups $u\mathcal{I}$ the `Ellis groups' or `ideal groups' of the flow $(G,X)$; we caution that, although this is the standard terminology used in model theory, it is not the standard terminology used in topological dynamics.

The Ellis groups can be endowed with a topology, called the $\tau$-topology, that make them into \textit{quasi}compact $T_1$ topological groups. This is defined as followed. Given $f\in E$ and $B\subseteq E$, we denote by $f\circ B$ the set of all limits of nets $(f_ib_i:i\in I)$, where $b_i\in B$ and $(f_i:i\in I)$ is a net of elements of form $f_i=\pi_{g_i},g_i\in G$ that converges to $f$. There are many basic identities of $\circ$, which can be found in Fact A.25 of \cite{tomasz_thesis}; the only one we will need to quote explicitly is the following:
\begin{fact}\label{composition for circ operation}
    For any $f_1,f_2\in E$ and $B\subseteq E$, both $f_1\circ(f_2B)$ and $f_1(f_2\circ B)$ are subsets of $(f_1f_2)\circ B$.
\end{fact} Given a minimal ideal $\mathcal{I}$ of $E$, an idempotent $u\in\mathcal{I}$, and a subset $A\subseteq u\mathcal{I}$, we define $\mathrm{cl}_\tau(A)=(u\mathcal{I})\cap (u\circ A)$; $\mathrm{cl}_\tau$ defines a closure operator on $u\mathcal{I}$, and the remarkable result of Ellis is then that taking the sets of form $\mathrm{cl}_\tau(A)$ to be the closed sets defines a quasicompact $T_1$ topology on the group $u\mathcal{I}$, called the $\tau$-topology, with respect to which multiplication is continuous in each variable; we thus say that $u\mathcal{I}$ is a quasicompact $T_1$ `semitopological' group. Moreover, the group isomorphisms from point (4) above are isomorphisms with respect to the $\tau$-topology, and point (5) above holds for the isomorphism type of the semitopological group $u\mathcal{I}$. When we talk about `the' Ellis group of the flow $(G,X)$ we mean the isomorphism type of the semitopological group $u\mathcal{I}$. An important fact (Fact A.33 in \cite{tomasz_thesis}) is that the $\tau$-topology is coarser than the subspace topology, ie every $\tau$-open set $U\subseteq u\mathcal{I}$ is of form $O\cap u\mathcal{I}$ for some open set $O\subseteq E$.

The $\tau$-topology is only $T_1$, and it may or may not be Hausdorff. If it is Hausdorff, then by `Ellis' joint continuity theorem' (Fact A.6 in \cite{tomasz_thesis}), inversion is continuous and multiplication is jointly continuous, and so the Ellis group is actually a compact Hausdorff topological group. We then can apply Lemma 3.1 of \cite{tomasz_krupinski_pillay} to get the following:

\begin{fact}\label{continuity of ux} If the $\tau$-topology is Hausdorff, then for any minimal ideal $\mathcal{I}\subseteq E$ and any idempotent $u\in \mathcal{I}$, the map $\overline{u\mathcal{I}}\to u\mathcal{I}$ given by $x\mapsto ux$ is a continuous map from the subspace topology to the $\tau$-topology, where $\overline{u\mathcal{I}}$ is the closure of $u\mathcal{I}$ computed in $E$.
\end{fact}

All of the above is machinery from abstract topological dynamics. Given a theory $T$, a definable group $G$ of $T$, and a model $M$ of $T$, there is a natural continuous action of the $M$-points $G(M)$ (thought of as a discrete group) on the compact space of types $S_G(M)$, where $gp$ is the pushforward of $p$ under the $M$-definable map $G\to G$ given by $x\mapsto gx$. So $(G(M),S_G(M))$ is a flow; note that it is point-transitive, since the orbit of the realized type $\tp(1/M)$ is dense. The idea of studying the definable group $G$ by studying the flow $(G(M),S_G(M))$ from the perspective of topological dynamics comes from \cite{newelski_1}. The Ellis semigroup of $(G(M),S_G(M))$ is a bit complicated to describe, and a detailed description can be found in \cite{newelski_1}. Instead of $(G(M),S_G(M))$, in many applications it winds up being more natural to work with a different model-theoretic $G(M)$-flow, which is the space $S_G^{\mathrm{fs}}(\mathfrak{C},M)$ of global types concentrated on $G$ and finitely satisfiable in $M$; this flow is again point-transitive, again since the orbit of the realized type $\tp(1/\mathfrak{C})$ is dense. The Ellis semigroup of the $G(M)$-flow $S_G^{\mathrm{fs}}(\mathfrak{C},M)$ is easily described, and as a $G(M)$-flow it is naturally isomorphic to $S_G^{\mathrm{fs}}(\mathfrak{C},M)$ itself. (For this reason it is the more natural flow for us to work with.) The semigroup operation is defined by taking $pq$ to be the pushforward of the Morley product $p_x\otimes q_y$ under the multiplication map; more explicitly, for any model $N\succcurlyeq M$, we have $(pq)|_N=\tp(ab/N)$ where $b\models q|_N$ and $a\models p|_{(N,b)}$.

Given a structure $M$ and a definable group $G$ in $M$, by the `Ellis semigroup of $G(M)$' and the `Ellis group of $G(M)$' we mean the Ellis semigroup and Ellis group of the $G(M)$-flow $S_G^{\mathrm{fs}}(\mathfrak{C},M)$. (\uline{As a disclaimer, this is not standard terminology.}) By the remarks above, we will throughout this paper identify the Ellis semigroup of $G(M)$ with $S_G^{\mathrm{fs}}(\mathfrak{C},M)$, and the `Ellis groups of $G(M)$' will be the groups $u\mathcal{I}$ equipped with the $\tau$-topology, for $\mathcal{I}$ a minimal ideal of $S_G^{\mathrm{fs}}(\mathfrak{C},M)$ and $u\in\mathcal{I}$ an idempotent. We caution that the Ellis group of the flow $(G(M),S_G(M))$ may in general be a proper quotient of what we are calling the Ellis group of $G(M)$; thank you to Krzysztof Krupi\'nski for pointing this out to me.

We use the following theorem, which was proved for countable $M$ in \cite{chernikov_gannon_krupinski} and then for arbitrary $M$ in \cite{basso_zucker}; see the introduction for more discussion of this.

\begin{fact}\label{tau-topology is hausdorff} If $T$ is NIP, then the $\tau$-topology on the Ellis group of $G(M)$ is Hausdorff, and so in particular the Ellis group is a compact Hausdorff topological group.
\end{fact}

We also need the following, which is Theorem 1.7 of \cite{chernikov_gannon_krupinski}:

\begin{fact}\label{borel definability}
    Suppose $T$ is NIP. Let $\phi(x,b)$ be any $L(\mathfrak{C})$-formula, and let $[\phi(x,b)]$ be the usual clopen subset of $S^{\mathrm{fs}}_G(\mathfrak{C},M)$ consisting of those types concentrated on $\phi(x,b)$. Let $\mathcal{I}\subseteq S^{\mathrm{fs}}_G(\mathfrak{C},M)$ be a minimal ideal and let $u\in\mathcal{I}$ be an idempotent. Then, $[\phi(x,b)]\cap u\mathcal{I}$ is a Borel set in the $\tau$-topology.
\end{fact}

An essential fact, first observed in \cite{chernikov_simon}, which we will not use explicitly here but which is critical in the proof of Fact \ref{tau-topology is hausdorff}, is that if $T$ is NIP then the flows $(G(M),S_G(M))$ and $(G(M),S_G^{\mathrm{fs}}(\mathfrak{C},M))$ are `tame', in the sense of Glasner \cite{glasner_tameness}.

\subsection{VC-theory}\label{vc theory section}
Suppose $\mathcal{F}$ is a family of subsets of some ambient set $X$. For a subset $Y\subseteq X$, we define the `trace' of $\mathcal{F}$ on $Y$ as $\mathcal{F}|_Y=\{F\cap Y:F\in \mathcal{F}\}$, and we say that $\mathcal{F}$ `shatters' $Y$ if $\mathcal{F}|_Y$ is the powerset of $Y$. The VC-dimension of $\mathcal{F}$ is then the supremum over all $n\in\omega\cup\{\infty\}$ such that $\mathcal{F}$ shatters a set of size $n$. A more refined measure is the VC-density of $\mathcal{F}$. First we define the `shatter function' $\pi_\mathcal{F}$, where $\pi_\mathcal{F}(n)$ is the maximum of $|\mathcal{F}|_Y|$ over all subsets $Y\subseteq X$ of size at most $n$. The VC-density of $\mathcal{F}$ is then the infimum over all $r\in\mathbb{R}\cup\{\infty\}$ such that $\pi_\mathcal{F}(n)=O(n^r)$; in other words, the infirmum over all $r$ such that there exists a constant $C$ with $\pi_\mathcal{F}(n)\leqslant Cn^r$ for all $n$. By definition, if $\mathcal{F}$ has unbounded VC-dimension then $\pi_\mathcal{F}(n)=2^n$ for all $n\in\omega$, and so the VC-density of $\mathcal{F}$ is infinite. The famous `Sauer-Shelah' theorem gives a strong converse:
\begin{fact}\label{sauer-shelah}
    If $\mathcal{F}$ has VC-dimension at most $d$, then it has VC-density at most $d$.
\end{fact} Given $\mathcal{F}$, the `dual' of $\mathcal{F}$, denoted $\mathcal{F}^*$, is the set of all subsets of the power set $P(X)$ of form $\mathcal{F}_x:=\{F\in\mathcal{F}:x\in F\}$ for some $x\in X$. We then define the VC-codimension and VC-codensity of $\mathcal{F}$ as the VC-dimension and VC-density of $\mathcal{F}^*$. If $\mathcal{F}$ has VC-dimension $d$, then the VC-dimension of $\mathcal{F}^*$ is $<2^{1+d}$.

A fundamental result on VC-dimension is the so-called VC-theorem; see Theorem 6.6 in \cite{simon_book}. We actually do not need the full strength of the VC-theorem, but just an important consequence, which is the existence of $\varepsilon$-nets. This is essentially `Proposition 2.18' in \cite{hrushovski}, but we need to clarify two differences. First, as stated in \cite{hrushovski}, Proposition 2.18 is not quite right (a standard counterexample is to look at the $\sigma$-algebra generated by intervals on $\omega_1$, and to take the measure to assign value $1$ to a set if it contains an upwards ray and measure $0$ otherwise.) The correct statement requires some measureability hypotheses on certain functions associated to the set system; the right hypotheses can be found, for example, in the discussion after Theorem 6.6 of \cite{simon_book}. In particular, any finite (or even countably infinite) family of measureable sets satisfies the hypotheses needed, and we only need this case, so we state the fact just with that assumption. Second, Proposition 2.18 in \cite{hrushovski} assumes the VC-\textit{dimension} of the family is at most $d$. However, the bounds in the VC-theorem depend only on the growth of the shatter function, so we get the stronger fact.

\begin{fact}\label{epsilon nets} For any constants $C,\delta$, there is a constant $D=D(C,\delta)$ with the following property. Suppose that $(X,\mu)$ is a probability space, and that $\mathcal{F}$ is a finite family of measureable subsets of $X$. Suppose that $\pi_\mathcal{F}(n)\leqslant Cn^\delta$ for all $n$. Then, for any $n$, if we let $N\geqslant Dn^2$, then we may find $b_1,\dots,b_N\in X$ such that, for any $F\in\mathcal{F}$, if $\mu(F)\geqslant 1/n$ then there is $i\in[N]$ with $b_i\in F$.
\end{fact}

A theory $T$ is NIP if, for any model $M\models T$, and any formula $\phi(x,y)$, the set system $\{\phi(M^x,b):b\in M^y\}$ has finite VC-dimension. By the VC-dimension and the VC-density of the formula $\phi(x,y)$ we mean that of the associated set system. In particular, the VC-codensity of $\phi$ is precisely the infimum over all $r$ such that $|S_\phi(B)|=O(|B|^r)$, where $B\subseteq M^y$ ranges over finite sets and $S_\phi(B)$ is the set of complete $\phi$-types over $B$.

Let us say that a theory has `bounded VC-codensity' (this is, to my knowledge, not standard terminology) if, for every tuple of variables $x$, there is a fixed $\delta$ depending only on $x$ such that every formula $\phi(x,y)$ has VC-codensity at most $\delta$. There are many examples, coming for example from the papers \cite{five_authors}, \cite{guingona_laskowski}, and \cite{basu_patel}.

\subsection{Compact groups}\label{compact groups section}
We will need the basic structure theory of compact Hausdorff groups; everything we need can be found in Chapters 1 and 2 of \cite{hofmann_morris}. Recall that, if $G$ is a compact Hausdorff group, then $G$ admits a unique normalized bi-invariant Haar measure, ie a regular Borel measure $\eta$ on $G$ such that $\eta(G)=1$ and such that $\eta(gB)=\eta(Bg)=\eta(B)$ for every Borel set $B\subseteq G$ and $g\in G$. A Hilbert module of $G$ is a Hilbert space $H$ and a homomorphism $\lambda$ from $G$ to the group $U(H)$ of unitary transformations of $H$ such that the map $G\times H\to H$ given by $(g,h)\mapsto \lambda(g)h$ is continuous. $G$ always admits a faithful Hilbert representation: let $H=L^2(G,\eta)$, and let $\lambda$ be the regular representation, so that $(\lambda(g)(f))(x)=f(g^{-1}x)$ for all $f\in L^2(G,\eta)$ and $x\in G$. The main basic result we will use is the following part of the Peter-Weyl theorem; see Corollary 2.25 in \cite{hofmann_morris}:
\begin{fact}\label{peter-weyl 1}
    Let $H$ be a Hilbert module of $G$. Then there is an orthogonal family $(H_i:i\in I)$ of finite-dimensional $G$-submodules of $H$ such that $H$ is the Hilbert space orthogonal sum of the $H_i$, ie $\bigoplus_{i\in I}H_i$ is dense in $H$.
\end{fact} A fundamental consequence, for example Corollary 2.35 in \cite{hofmann_morris}, is the following:
\begin{fact}\label{peter-weyl 2}
    $G$ is an inverse limit of compact Lie groups. In particular, $G$ is Lie if and only if it is `NSS', ie if and only if there is an open neighborhood of the identity in $G$ that does not contain any non-trivial subgroup of $G$.
\end{fact}

In Section 7 of \cite{hrushovski}, Hrushovski defines a notion of `finite Archimedean rank' and develops its basic theory. In particular, he gives several equivalent characterizations, some of which we cite below (see points (3) and (4) of Proposition 7.6 in \cite{hrushovski}):
\begin{fact}\label{finite archimedean rank definition}
    For $G$ a compact Hausdorff group, the following are equivalent:
    \begin{enumerate}
        \item The connected component $G^0$ is profinite-by-Lie, ie $G^0$ contains a closed normal profinite subgroup $N$ such that $G^0/N$ is Lie. In particular this means that $G$ is profinite-by-Lie-by-profinite.
        \item $G$ is profinite-by-Lie-by-finite, ie $G$ contains a closed normal finite-index profinite-by-Lie subgroup.
        \item There is some $d$ such that $G$ is an inverse limit of compact Lie groups of dimension at most $d$.
        \item There is some $d$ such that $G^0$ is an inverse limit of compact Lie groups of dimension at most $d$.
    \end{enumerate} We say that a group satisfying these conditions has `finite Archimedean rank'.
\end{fact}For us the most convenient characterization will be (4). Let us also record a few specific easy observations:

\begin{lemma}\label{lie quotient of far group}
        Suppose that $G$ is a compact Hausdorff group and that $G$ is an inverse limit of Lie groups of dimension at most $d$. Then, for any Lie group $L$, if there is a continuous group epimorphism $\pi:G\to L$ then $L$ has dimension at most $d$.
\end{lemma}
\begin{proof}
Lie groups are NSS, so let $V\subseteq L$ be an open neighborhood of the identity containing no nontrivial subgroup of $L$ and let $U=\pi^{-1}(V)$. Thus every subgroup of $G$ contained in $U$ is contained in $\ker(\pi)$; call this ($\ast$). Since $G$ is an inverse limit of Lie groups of dimension at most $d$, let $(K_i:i\in I)$ be a downwards-directed family of closed normal subgroups of $G$ such that each $G/K_i$ is Lie of dimension at most $d$ and such that $\bigcap_{i\in I}K_i=\{1\}$. In particular, $\bigcap_{i\in I}K_i\subseteq U$. So the $G\setminus K_i$ together with $U$ form an open covering of $G$, and hence by compactness there is $s\subseteq I$ finite with $G=U\cup\bigcup_{i\in s}(G\setminus K_i)$. In other words $\bigcap_{i\in s}K_i\subseteq U$. By downwards-directedness there is some $j\in I$ with $K_j\subseteq\bigcap_{i\in s}K_i$. But now $K_j\subseteq U$, so by ($\ast$) $K_j\subseteq \ker(\pi)$. So $\pi$ induces a continuous group epimorphism $G/K_j\to L$. Since $G/K_j$ and $L$ are compact Hausdorff, this implies that $L$ is isomorphic as a topological group to a quotient of $G/K_j$ by a closed normal subgroup, and since $G/K_j$ is Lie of dimension at most $d$ thus also $L$ has dimension at most $d$.
\end{proof}

% \begin{lemma}\label{lie quotient of far group}
%     Suppose $G$ is a compact Hausdorff group and that $G$ is an inverse limit of Lie groups of dimension at most $d$. Then, for any closed normal subgroup $N$ of $G$ such that $G/N$ is a Lie group, $G/N$ has dimension at most $d$.
% \end{lemma}
% \begin{proof}
% Lie groups are NSS, so let $V\subseteq G/N$ be an open neighborhood of the identity containing no nontrivial subgroup of $G/N$ and let $U$ be the preimage of $V$ under the quotient map $G\to G/N$. Thus every subgroup of $G$ contained in $U$ is contained in $N$; call this ($\ast$). Since $G$ is an inverse limit of Lie groups of dimension at most $d$, let $(K_i:i\in I)$ be a downwards-directed family of closed normal subgroups of $G$ such that each $G/K_i$ is Lie of dimension at most $d$ and such that $\bigcap_{i\in I}K_i=\{1\}$. In particular, $\bigcap_{i\in I}K_i\subseteq U$. So the $G\setminus K_i$ together with $U$ form an open covering of $G$, and hence by compactness there is $s\subseteq I$ finite with $G=U\cup\bigcup_{i\in s}(G\setminus K_i)$. In other words $\bigcap_{i\in s}K_i\subseteq U$. By downwards-directedness there is some $j\in I$ with $K_j\subseteq\bigcap_{i\in s}K_i$. But now $K_j\subseteq U$, so by ($\ast$) $K_j\subseteq N$. Hence $G/N$ is the quotient of $G/K_j$ by the closed normal subgroup $N/K_j$, and since $G/K_j$ is Lie of dimension at most $d$ also $G/N$ has dimension at most $d$.
% \end{proof}

\begin{lemma}\label{connected component of quotient}Suppose $\pi:S\to T$ is a continuous epimorphism of compact Hausdorff groups. Then $\pi(S^0)=T^0$.
\end{lemma}
\begin{proof} Being the continuous image of a connected set, $\pi(S^0)$ is connected, so $\pi(S^0)\subseteq T^0$. For the other inclusion, note since $S,T$ are compact Hausdorff, $\pi$ is a closed map, so $\pi(S^0)$ is a closed normal subgroup of $T$, and $\pi$ induces a continuous group epimorphism $S/S^0\to T/\pi(S^0)$. But $S/S^{0}$ and $T/\pi(S^0)$ are compact Hausdorff, so $T/\pi(S^0)$ is isomorphic as a topological group to a quotient of $S/S^0$ by a closed normal subgroup. Since $S/S^0$ is profinite, thus $T/\pi(S^0)$ is profinite too, and thus in particular $T/\pi(S^0)$ is totally disconnected. But now, being the continuous image of a connected set, $T^0/\pi(S^0)\subseteq T/\pi(S^0)$ is connected, forcing $T^0/\pi(S^0)$ to be trivial and hence $T^0=\pi(S^0)$.
\end{proof}

The following is of course implicit in Section 7 of \cite{hrushovski} but it is not written explicitly anywhere, so we include it here.

\begin{lemma}\label{archimedean rank of connected component}
    Suppose $T$ is a compact Hausdorff group and that $T^{0}$ is an inverse limit of compact Lie groups of dimension at most $r$. Then $T$ is an inverse limit of compact Lie groups of dimension at most $r$.
\end{lemma}
\begin{proof}
By Peter-Weyl $T$ is an inverse limit of compact Lie groups, so it suffices to show that, if $N$ is any closed normal subgroup of $T$ such that $T/N$ is a Lie group, then $T/N$ has dimension at most $r$. So fix such an $N$. By Lemma \ref{connected component of quotient}, we have $T^0N/N=(T/N)^0$, so $T/N$ and its closed subgroup $T^0N/N$ have the same dimension as Lie groups.

On the other hand we have a natural continuous group isomorphism $T^0/T^0\cap N\to T^0N/N$; since all the groups in question are compact Hausdorff, this map (being a continuous bijection) is a homeomorphism, so that $T^0/T^0\cap N$ and $T^0N/N$ are isomorphic as topological groups. But now $T^0/T^0\cap N$ is a Lie quotient of $T^0$, so by Lemma \ref{lie quotient of far group}, it has dimension at most $r$, so we are done.
\end{proof}

%\newpage
\section{Piecewise f-genericity}\label{piecewise f-generic section}
Throughout this section, \uline{let $T$ be an NIP theory and $\mathfrak{C}$ a monster model, and let $G$ a $\varnothing$-definable group}. Also, \uline{assume throughout that $\dcl(\varnothing)$ is a model, which we denote throughout by $M_0$}. We recall the following definitions from \cite{hrushovski_pillay} and \cite{chernikov_simon}, which are indispensable in the study of definably amenable NIP groups:

\begin{definition}\label{f-generic definition}
A definable subset $\phi(x,b)$ of $G$ is (left) `f-generic' if there is some small model $M$ such that no left translate of $\phi(x,b)$ forks over $M$, and (left) `strong f-generic' if no left translate of $\phi(x,b)$ forks over $M_0$.
\end{definition} (In \cite{hrushovski_pillay}, what we are calling `strong f-genericity' was just called `f-genericity', and the notion was renamed in \cite{chernikov_simon}; we follow the terminological conventions of \cite{chernikov_simon}.) As usual, we say that a partial type is (strong) f-generic if it does not imply any non-(strong)-f-generic formula. By Fact \ref{forking=dividing}, as observed in \cite{chernikov_simon}, f-genericity has a purely combinatorial characterization: a definable subset $D\subseteq G$ is not f-generic if and only if it `$G$-divides', meaning that for some $k\in\omega$ there is an infinite $k$-inconsistent sequence of left translates of $D$. On the other hand, strong f-genericity is really a model-theoretic notion, without a combinatorial characterization.

We summarize some fundamental results, none of which we need, but which provide context for the rest of our paper. These facts come from \cite{hrushovski_pillay}, \cite{chernikov_simon}, and \cite{stonestrom}.
\begin{fact}\label{summary of basic f-generic facts}
    The following are equivalent:
    \begin{enumerate}
    \item $G$ is definably amenable.
    \item The non-(strong)-f-generic definable subsets of $G$ are an ideal, ie a union of two non-(strong)-f-generic definable sets is still not (strong) f-generic.
    \item $G$ admits a global (strong) f-generic type.
    \end{enumerate} Moreover, if $G$ is definably amenable, then a global type is strong f-generic if and only if it is f-generic and non-forking over $M_0$.
\end{fact}

Our fundamental insight in this paper is to identify appropriate ideals of definable sets in an \textit{arbitrary} NIP group, which we can use to substitute the use of (strong) f-generic sets in definably amenable NIP groups. We identify three natural ideals, which we give in the definition below; my motivation for the terminology came from the notion of a `piecewise syndetic' set in ergodic theory.
\begin{notation}
    Given a definable subset $\phi(x,b)$ of $G$ and a finite subset $s\subset G(\mathfrak{C})$, we use $s\phi(x,b)$ to denote the formula $\bigvee_{g\in s}g\phi(x,b)$.
\end{notation}
\begin{definition}\label{piecewise f-generic definition}
    A definable subset $\phi(x,b)$ of $G$ is `piecewise f-generic' if there is some finite $s\subset G(\mathfrak{C})$ such that $s\phi(x,b)$ is f-generic, `piecewise strong f-generic' if there is some finite $s\subset G(\mathfrak{C})$ such that $s\phi(x,b)$ is strong f-generic, and `piecewise strong f-generic with witnesses in $M_0$' if there is some finite $s\subset G(M_0)$ such that $s\phi(x,b)$ is strong f-generic.
\end{definition}

We summarize the basic properties of these three notions and give some examples in Section \ref{basic facts and examples section} below. The main result of Section \ref{piecewise f-generic section} is to show that (the negations of) the three notions defined above all give rise to ideals in the Boolean algebra of definable subsets of $G$; Sections \ref{weak genericity section}, \ref{piecewise strong f-genericity section}, and \ref{wideness section} respectively prove each of these.

The first and easiest of these cases is to show that the non-piecewise-f-generic sets form an ideal; in fact, we will point out that piecewise f-genericity coincides with Newelski's notion of `weak genericity'. Although not observed explicitly anywhere in the literature, this result is not really new, and it follows by arguments already in Section 3.3 of \cite{chernikov_simon}.

The notion of piecewise f-genericity will however be too weak to obtain our main applications in Section \ref{existence section} and Section \ref{FAR section}, and we will really need the notions of piecewise strong f-genericity and piecewise strong f-genericity with witnesses in $M_0$; unlike in the case of piecewise f-genericity, the techniques to prove that these latter notions give rise to ideals are new and we hope of independent interest. In fact, for our applications we will really only need the strongest notion, that of piecewise strong f-genericity with witnesses in $M_0$, and for technical reasons it is cleaner for us to prove that that notion gives rise to an ideal before showing that piecewise strong f-genericity does. Nonetheless, we still record the result on piecewise strong f-genericity, because it gives a pleasant duality; namely, as remarked in the introduction, in arbitrary NIP groups one cannot hope for a translation-invariant S1 ideal. However, the piecewise strong f-generic sets will give rise to a translation-invariant ideal, and the piecewise strong f-generic sets with witnesses in $M_0$ will give rise to an S1 ideal, and we think of these two notions as somehow being the two best alternatives in the absence of a canonical translation-invariant S1 ideal.

\subsection{Basic facts and examples}\label{basic facts and examples section}
Here we will summarize some basic examples and easy properties of the notions from Definition \ref{piecewise f-generic definition}.

First, by Fact \ref{summary of basic f-generic facts}, we have the following:

\begin{remark}\label{notions collapse in definably amenable groups}
    In a definably amenable NIP group, a definable subset is piecewise f-generic if and only if it is f-generic, and it is piecewise strong f-generic if and only if it is strong f-generic, hence if and only if it is piecewise strong f-generic with witnesses in $M_0$.
\end{remark}

Let us also observe the following:

\begin{remark}\label{wide implies non-forking}
    If $\phi(x,b)$ forks over $M_0$, then so does $g\phi(x,b)$ for any $g\in G(M_0)$. Thus if $\phi(x,b)$ is piecewise strong f-generic with witnesses in $M_0$ then in particular $\phi(x,b)$ does not fork over $M_0$.
\end{remark}
The following is easy to see, but requires a little bit more setup, so we postpone its proof to Lemma \ref{piecewise f-generic coincides with piecewise strong f-generic for M0 definable sets}.
\begin{remark}
    If $D\subseteq G$ is $M_0$-definable, then $D$ is piecewise f-generic if and only if it is piecewise strong f-generic if and only if it is piecewise strong f-generic with witnesses in $M_0$.
\end{remark}
We have two other basic combinatorial examples; recall that a definable subset is (left) generic if some union of finitely many left translates of it covers $G$, and right generic if some union of finitely many right translates of it covers $G$.
\begin{remark}\label{genericity implications}
    Let $\phi(x,b)$ be a definable subset of $G$. If $\phi(x,b)$ is generic, then it is piecewise strong f-generic. If $\phi(x,b)$ is right generic, then it is f-generic and hence piecewise f-generic.
\end{remark}
\begin{proof}
    The first part is obvious, since $G$ is a strong f-generic subset of itself. The second part is folklore and can essentially be found in \cite{gismatullin}, among other places; nonetheless we give the details. Let $g_1,\dots,g_n$ be such that $G(x)\vdash \bigcup_{i\in[n]}\phi(x,b)g_i$, and let $M$ be any model containing $(b,g_1,\dots,g_n)$; it is enough to show that no translate of $\phi(x,b)$ forks over $M$. By Fact \ref{forking=dividing}, it is enough to show that no translate of $\phi(x,b)$ divides over $M$. So let $(a_i:i\in\omega)$ be an $M$-indiscernible sequence. There is now some $i\in[n]$ with $a_0^{-1}\models\phi(x,b)g_i$, and by $M$-indiscernible we have $a_k^{-1}\models \phi(x,b)g_i$ for all $k\in\omega$, so now $g_i^{-1}\in\bigcap_{k\in\omega}a_k\phi(x,b)$. So in particular $\bigcap_{k\in\omega}a_k\phi(x,b)$ is consistent, and since $(a_k:k\in\omega)$ was arbitrary it follows that no translate of $\phi(x,b)$ divides over $M$.
\end{proof}

Another easy observation is the following:
\begin{remark}\label{translation invariance of piecewise (strong) f-genericity}
    A left translate of a piecewise f-generic (resp. piecewise strong f-generic) set is still piecewise f-generic (resp. piecewise strong f-generic). A right translate of a piecewise f-generic set is still piecewise f-generic.
\end{remark}
\begin{proof}
    Suppose $s\phi(x,b)$ is (strong) f-generic for some finite $s\subset G(\mathfrak{C})$. Then, for any $a$, if we let $t=sa^{-1}=\{ga^{-1}:g\in s\}$, we have $t(a\phi(x,b))=s\phi(x,b)$, so that $a\phi(x,b)$ is also piecewise (strong) f-generic.

    On the other hand, suppose $s\phi(x,b)$ is f-generic for some finite $s\subset G(\mathfrak{C})$. By the `$G$-dividing' characterization of f-genericity, it is easy to see that any right translate $s\phi(x,b)a$ is also f-generic for any $a\in G(\mathfrak{C})$, so that $\phi(x,b)a$ is indeed piecewise f-generic.
\end{proof} 
Now let us give some examples distinguishing the notions. All of the basic kinds of behavior can already be found in definable groups in the fields $\mathbb{R}$ and $\mathbb{Q}_p$.

\begin{example} Let $M_0=(\mathbb{R},+,\cdot)$ and let $R$ be a saturated model.
    \begin{enumerate}
    \item If $G$ is definably amenable and $\phi(x,b)$ is f-generic but not strong f-generic, then by Remark \ref{notions collapse in definably amenable groups} $\phi(x,b)$ is not piecewise strong f-generic either. The simplest example of this is from Example 3.10 in \cite{chernikov_simon}: let $\phi(x_1,x_2;b_1,b_2)$ be the definable subset of $(R^2,+)$ given by $x_1\neq 0\wedge (b_1<x_2/x_1<b_2)$. Then $\phi(\bar{x},\bar{b})$ is f-generic but not piecewise strong f-generic.
    \item If $G$ is not definably amenable, there may be subsets of $G$ that are generic but not f-generic. The easiest example of this is in $G=\mathrm{SL}_2$ as a definable group in $M_0$. Indeed, consider the definable subset $D=\left\{\begin{bmatrix} a & b \\ c & d \end{bmatrix}:|a|\geqslant|c|\right\}$ of $G$. On the one hand, letting $g_n=\begin{bmatrix} 1 & 0 \\ 3n & 1 \end{bmatrix}$  for each $n\in\omega$, then we have that $g_iD\cap g_jD=\varnothing$ for each $i\neq j$, so that $D$ $G$-divides, and thus we may find some $g\in\mathrm{SL}_2(R)$ such that $gD$ forks over $M_0$. On the other hand, letting $h=\begin{bmatrix} 0 & -1 \\ 1 & 0 \end{bmatrix}$, we have $G=D\cup hD$, so that $D$ is generic. Now $gD$ is still generic, so it is piecewise strong f-generic by Remark \ref{genericity implications}, but by Remark \ref{wide implies non-forking} it is not piecewise strong f-generic with witnesses in $M_0$.
    \item Even if $G$ is definably amenable one may find f-generic sets that are not right f-generic. The easiest example is given in Example 6.4 from \cite{chernikov_simon}. Let $G$ be the semidirect product $(R^2,+)\rtimes S^1$, where the circle group $S^1$ acts by rotations. If $0<a$ is a positive infinitessimal, then the set $X_a=C_a\times S^1$ from Example 6.4 in \cite{chernikov_simon} is right f-generic but not f-generic, so $D=(X_a)^{-1}$ is f-generic but not right f-generic. Since $D$ is $M_0$-definable and f-generic, it is strong f-generic. However, it is not right f-generic, so we may find $g\in G(R)$ such that $Dg$ forks over $M_0$; then $Dg$ is not strong f-generic, so by Remark \ref{notions collapse in definably amenable groups} it is not piecewise strong f-generic. In particular the second part of Remark \ref{translation invariance of piecewise (strong) f-genericity} does not hold for piecewise strong f-genericity.
    \end{enumerate} Similar examples may be found in $\mathbb{Q}_p$.
\end{example}

\subsection{Weak genericity}\label{weak genericity section}
This section should be seen as an easy `warmup' to the substantially more involved results of the subsequent two sections. In particular, the result of this section basically follows immediately from arguments already present in Section 3.3 of \cite{chernikov_simon}. (See also Section 3 of our paper \cite{pillay_stonestrom} with Anand Pillay, where we adapted the arguments from Section 3.3 of \cite{chernikov_simon} to the setting of the automorphism group acting on the space of types.)

Recall from \cite{newelski_1} that a definable subset $\phi(x,b)$ of $G$ is said to be `weakly generic' if there is a non-generic definable subset $\psi(x,c)$ of $G$ such that $\phi(x,b)\vee\psi(x,c)$ is generic. In the language from Section \ref{topological dynamics section}, this says precisely that the clopen subset of $S_G(\mathfrak{C})$ of types concentrated on $\phi(x,b)$ is a weakly generic set in the flow $(G(\mathfrak{C}),S_G(\mathfrak{C}))$. We will observe that the weakly generic sets correspond precisely to the piecewise f-generic sets.

In one direction, we have the following, which does not require an NIP assumption; it is exactly the content of the proof of Lemma 1.8 in \cite{newelski_petrykowski} and Proposition 3.30 in \cite{chernikov_simon}.
\begin{fact}\label{weak generic implies piecewise f-generic}
    If $\phi(x,b)$ is a weakly generic definable subset of $G$, then there is some finite $s\subset G(\mathfrak{C})$ such that $s\phi(x,b)$ does not $G$-divide.
\end{fact}

The other direction uses NIP essentially; it is exactly Proposition 3.33 in \cite{chernikov_simon}. Proposition 3.33 in \cite{chernikov_simon} assumes that $G$ is definably amenable, but this assumption is not used anywhere in the proof, and so the result holds without that assumption.

\begin{fact}\label{finitely many almost periodic types}
    If $\phi(x,b)$ is an f-generic definable subset of $G$, then there are finitely many almost periodic types $p_1,\dots,p_n$ of the flow $(G(\mathfrak{C}),S_G(\mathfrak{C}))$ such that, for any $g\in G(\mathfrak{C})$, there is some $i\in[n]$ with $p_i\vdash g\phi(x,b)$.
\end{fact}

So altogether we get the following:

\begin{proposition}\label{piecewise f-genericity coincides with weak genericity}
    A definable subset of $G$ is piecewise f-generic if and only if it is weakly generic. In particular, the non-piecewise-f-generic definable subsets of $G$ form an ideal.
\end{proposition}
\begin{proof}
The backwards direction is exactly Fact \ref{weak generic implies piecewise f-generic}. For the forwards direction, suppose $\phi(x,b)$ is piecewise f-generic, and let $s\subset G(\mathfrak{C})$ be a finite set such that $s\phi(x,b)$ is f-generic. By Fact \ref{finitely many almost periodic types}, there is an almost periodic type $p$ of the flow $(G(\mathfrak{C}),S_G(\mathfrak{C}))$ such that $p\vdash s\phi(x,b)$. So $p\vdash g\phi(x,b)$ for some $g\in G(\mathfrak{C})$. By Fact \ref{almost periodic implies weak generic}, $g\phi(x,b)$ is weakly generic, and since the ideal of weakly generic sets is translation-invariant $\phi(x,b)$ is weakly generic too.
\end{proof}

%Let us make a brief remark about the automorphism group setting, as in \cite{khp_1} and \cite{khp_2}. In this brief discussion assume that $T$ is arbitrary, not necessarily NIP. Say that a formula $\phi(x,b)$ is `piecewise non-dividing (over $\varnothing$)' if there are finitely many conjugates $b_1,\dots,b_n$ of $b$ such that $\bigvee_{i\in[n]}\phi(x,b_i)$ does not divide (over $\varnothing$); this is the precise analogue (for the flow $(\mathrm{Aut}(\mathfrak{C}),S_x(\mathfrak{C})$) of what one might call `piecewise non-$G$-dividing' (for the flow $(G(\mathfrak{C}),S_G(\mathfrak{C}))$. Now, just as piecewise \cite{pillay_stonestrom}

\subsection{Wideness}\label{wideness section}
Now we move on to our main result. For technical reasons, it is convenient to first prove that the definable sets that are not piecewise strong f-generic with witnesses in $M_0$ form an ideal, before proving the analogous result about piecewise strong f-genericity, despite the fact that the former notion is stronger than the latter. We will prove the latter result in Section \ref{piecewise strong f-genericity section}. To make our terminology less cumbersome, we will use a shorthand term, which we will use throughout the rest of the paper:

\begin{notation}\label{wideness definition}
    Say that a definable subset of $G$ is `wide' if and only if it is piecewise strong f-generic with witnesses in $M_0$. Thus $\phi(x,b)$ is wide iff there is a finite subset $s\subset G(M_0)$ such that $s\phi(x,b)$ is strong f-generic, ie such that no translate of $s\phi(x,b)$ forks over $M_0$.
\end{notation}

Our starting point is the following key observation.
\begin{lemma}\label{key observation}
    Let $\psi_i(x,z):i\in[n]$ be formulas that imply $x\in G$, and let $(a_i,c_i:i\in[n])$ be a sequence such that $a_{<i}\ind_{M_0} (a_{\geqslant i},c_i)$ and $a_{>i}\ind_{M_0}^u (a_i,c_i)$ for each $i$. Suppose that $a_ig\psi_i(x,c_i)$ forks over $M_0$ for all $g\in G(M_0)$. Then, letting $b=a_1\ldots a_n$, $bg\psi_i(x,c_i)$ forks over $M_0$ for all $g\in G(M_0)$ and $i\in[n]$.
\end{lemma}
\begin{proof}
    Suppose for contradiction that, for some $i\in[n]$ and $g_0\in G(M_0)$, $bg_0\psi_i(x,c_i)$ does not fork over $M_0$. Write $b_1=a_1\dots a_{i-1}$ and $b_2=a_{i+1}\dots a_n$, so that $b=b_1a_ib_2$; thus $b_1a_ib_2g_0\psi_i(x,c_i)$ does not fork over $M_0$. Also, by the assumption of the lemma, we have $b_1\ind_{M_0} (a_i,b_2,c_i)$ and $b_2\ind_{M_0}^u(a_i,c_i)$.

    I first claim that $a_ib_2g_0\psi_i(x,c_i)$ does not fork over $M_0$. Indeed, suppose otherwise; then by Fact \ref{forking=dividing} there is an $M_0$-indiscernible sequence $(a_{ik},b_{2k},c_{ik}:k\in\omega)$ with $(a_{i0},b_{20},c_{i0})=(a_i,b_2,c_i)$ and such that $\bigwedge_{k\in\omega}a_{ik}b_{2k}g_0\psi_i(x,c_{ik})=\varnothing$. Since $b_1\ind_{M_0} (a_i,b_2,c_i)$, we may assume after an automorphism that $(a_{ki},b_{2k},c_{ik}:k\in\omega)$ is $(M_0,b_1)$-indiscernible, or in other words that $(b_1,a_{ki},b_{2k},c_{ik}:k\in\omega)$ is $M_0$-indiscernible. But $\bigwedge_{k\in\omega}a_{ik}b_{2k}g_0\psi_i(x,c_{ik})=\varnothing$, so also $$\bigwedge_{k\in\omega}b_1a_{ik}b_{2k}g_0\psi_i(x,c_{ik})=b_1\bigwedge_{k\in\omega}a_{ik}b_{2k}g_0\psi_i(x,c_{ik})=\varnothing;$$ since $(b_1,a_{ki},b_{2k},c_{ik}:k\in\omega)$ is $M_0$-indiscernible it follows that $b_1a_{i0}b_{20}g_0\psi_i(x,c_{i0})=b_1a_ib_2g_0\psi(x,c_i)$ divides over $M_0$, a contradiction.

    So indeed $a_ib_2g_0\psi_i(x,c_i)$ does not fork over $M_0$. By Fact \ref{lowness}, there is some $\eta(x,y,z)\in\tp(a_i,b_2,c_i/M_0)$ such that $a'b'g_0\psi_i(x,c')$ does not fork over $M_0$ for all $(a',b',c')\models\eta$. Now $b_2\models\eta(a_i,y,c_i)$, and we have $b_2\ind_{M_0}^u(a_i,c_i)$, so there is some $g_1\in G(M_0)$ such that $g_1\models\eta(a_i,y,c_i)$. But then $a_ig_1g_0\psi(x,c_i)$ does not fork over $M_0$. But $g_1g_0\in G(M_0)$, contradicting the assumption that $a_ig\psi(x,c_i)$ forks over $M_0$ for all $g\in G(M_0)$.
\end{proof}

Now we may begin working with wide and non-wide sets. The first significant consequence of Lemma \ref{key observation} is the following:
\begin{lemma}\label{union of conjugates wide implies wide}
    Let $\psi(x,c)$ be a definable subset of $G$, and suppose that, for every $(c_i:i\in\omega)$ a Morley sequence in a global coheir of $\tp(c/M_0)$ strictly invariant over $M_0$, there is some $m\in\omega$ such that $\bigvee_{k\in[m]}\psi(x,c_k)$ is wide. Then $\psi(x,c)$ is wide.
\end{lemma}
\begin{proof} Assume for contradiction that $\psi(x,c)$ is not wide. Thus, for any finite $s\subset G(M_0)$, $s\psi(x,c)$ is not strong f-generic, whence there is some $a$ such that $as\psi(x,c)$ forks over $M_0$, ie such that $ag\psi(x,c)$ forks over $M_0$ for all $g\in s$. For a given $g\in G(M_0)$, by Fact \ref{lowness} there is a partial type $\pi_g(u,z)$ with parameters in $M_0$ such that $(a,c')\models\pi_g(u,z)$ if and only if $ag\psi(x,c')$ forks over $M_0$, and now $\bigwedge_{g\in s} \pi_g(u,c)$ is satisfiable for each finite subset $s\subset G(M_0)$. So by compactness we may find $a\models\bigwedge_{g\in G(M_0)}\pi_g(u,c)$, so that $ag\psi(x,c)$ forks over $M_0$ for all $g\in G(M_0)$.

By Fact \ref{existence of strictly non-forking coheirs}, let $(a_i,c_i:i\in\omega)$ be a Morley sequence in a global coheir of $\tp(a,c/M_0)$ strictly invariant over $M_0$. By the assumption, there is some $m\in\omega$ such that $\bigvee_{k\in[m]}\psi(x,c_k)$ is wide. Let $s\subset G(M_0)$ be such that $s\bigvee_{k\in[m]}\psi(x,c_k)$ is strong f-generic. Let $b=a_1\ldots a_m$. Since $s\bigvee_{k\in[m]}\psi(x,c_k)$ is strong f-generic, $bs\bigvee_{k\in[m]}\psi(x,c_k)$ does not fork over $M_0$, so in particular (since the forking sets form an ideal) there is some $g_0\in s$ and $k\in[m]$ such that $bg_0\psi(x,c_k)$ does not fork over $M_0$. This contradicts Lemma \ref{key observation}.
\end{proof}

We wish to show that the family of non-wide sets is S1, in the sense of Hrushovski, meaning that is $\psi(x,c)$ is wide and $c,c'$ start an indiscernible sequence then $\psi(x,c)\wedge\psi(x,c')$ is still wide. By Lemma \ref{union of conjugates wide implies wide}, it is enough to show that an appropriate union of conjugates of $\psi(x,c)\wedge\psi(x,c')$ is wide, which we do now.

\begin{lemma}\label{union of conjunction is wide}
    Suppose that $\psi(x,c)$ is wide and that $c'\equiv_{M_0}c$. Then, for any $(c_k,c'_k:k\in\omega)$ a Morley sequence in a global extension of $\tp(c,c'/M_0)$ strictly invariant over $M_0$, there is $m\in\omega$ such that $\bigvee_{k\in[m]}[\psi(x,c_k)\wedge\psi(x,c'_{k})]$ is wide.
\end{lemma}
\begin{proof} Let $s\subset G(M_0)$ be a finite set such that $s\psi(x,c)$ is strong f-generic. We claim that, for some $m\in\omega$, $s\bigvee_{k\in[m]}[\psi(x,c_k)\wedge\psi(x,c'_k)]=\bigvee_{k\in[m]}s[\psi(x,c_k)\wedge\psi(x,c_k')]$ is strong f-generic, which will show the desired result. Thus suppose otherwise for a contradiction. Stretch the sequence $(c_k,c'_k:k\in\omega)$ into a longer indiscernible sequence $(c_\delta,c'_\delta:\delta\in|T|^+)$. By the assumption, for any finite $t\subset |T|^+$, $\bigvee_{\delta\in t}s[\psi(x,c_\delta)\wedge\psi(x,c'_\delta)]$ is not strong f-generic, so there is some $a$ such that $a\bigvee_{\delta\in t}s[\psi(x,c_\delta)\wedge\psi(x,c'_\delta)]$ forks over $M_0$, ie such that $as[\psi(x,c_\delta)\wedge\psi(x,c'_\delta)]$ forks over $M_0$ for each $\delta\in t$. By Fact \ref{lowness}, there is a partial type $\pi(u,z_1,z_2)$ such that $(a,c_1,c_2)\models\pi(u,z_1,z_2)$ iff $as[\psi(x,c_1)\wedge\psi(x,c_2)]$ forks over $M_0$, and now $\bigwedge_{\delta\in t}\pi(u,c_\delta,c'_\delta)$ is satisfiable for each finite $t\subset |T|^+$, so now by compactness there is some $a\models\bigwedge_{\delta\in|T|^+}\pi(u,c_\delta,c'_\delta)$, so that $as[\psi(x,c_\delta)\wedge\psi(x,c_\delta')]$ forks over $M_0$ for all $\delta\in |T|^+$.

On the other hand, by Fact \ref{bounded weight}, there is some $\delta\in |T|^+$ with $a\ind_{M_0}(c_\delta,c'_\delta)$. By Fact \ref{non-forking iff invariant}, it follows that $c_\delta\equiv_{(M_0,a)}c'_\delta$. On the other hand, since $s\psi(x,c)$ is strong f-generic, also $s\psi(x,c_\delta)$ is strong f-generic, so $as\psi(x,c_\delta)$ does not fork over $M_0$ and thus (since the forking sets form an ideal) there is some $g\in s$ such that $ag\psi(x,c_\delta)$ does not fork over $M_0$. Again by Fact \ref{non-forking iff invariant}, $ag\psi(x,c_\delta)$ is thus contained in a global $M_0$-invariant type, but we have $c_\delta\equiv_{(M_0,a)}c'_\delta$ and hence $(a,g,c_\delta)\equiv_{M_0} (a,g,c'_\delta)$, so any such type also contains $ag\psi(x,c'_\delta)$. This means that $ag\psi(x,c_\delta)\wedge ag\psi(x,c_\delta')=ag[\psi(x,c_\delta)\wedge\psi(x,c_\delta')]$ does not fork over $M_0$, contradicting that $as[\psi(x,c_\delta)\wedge\psi(x,c_\delta')]$ forks over $M_0$.
\end{proof} Now, as promised, we can deduce the S1 property.

\begin{theorem}\label{non-wide sets are s1}
    Suppose that $\psi(x,c)$ is wide, and that $(c_i:i\in[n])$ are such that $c_i\equiv_{M_0}c$ for each $i\in[n]$. Then $\bigwedge_{i\in[n]}\psi(x,c_i)$ is wide.
\end{theorem}
\begin{proof}
    By induction on $n$ it is enough to show that, if $c'\equiv_{M_0}c$, then $\psi(x,c)\wedge\psi(x,c')$ is wide. This follows from Lemma \ref{union of conjunction is wide} and Lemma \ref{union of conjugates wide implies wide}.
\end{proof}

Now, finally, we can prove the main result.
\begin{theorem}\label{non-wide sets are an ideal}
    Suppose $\phi(x,b)$ and $\psi(x,c)$ are not wide. Then neither is $\phi(x,b)\vee\psi(x,c)$.
\end{theorem}
\begin{proof}
    Suppose for contradiction that $\phi(x,b)\vee\psi(x,c)$ is wide. Since $\phi(x,b)$ and $\psi(x,c)$ are not wide, we may argue as in the first paragraph of the proof of Lemma \ref{union of conjugates wide implies wide} to find $u,v\in G(\mathfrak{C})$ such that $ug\phi(x,b)$ and $vg\psi(x,c)$ both fork over $M_0$ for all $g\in G(M_0)$. We will use the key observation \ref{key observation} to `combine' $u$ and $v$ in a certain way to find $w$ such that $wg\phi(x,b)$ and $wg\psi(x,c)$ both fork over $M_0$ for all $g\in G(M_0)$, which will contradict that $\phi(x,b)\vee\psi(x,c)$ is wide.

    Let $(u_\delta,v_\delta,b_\delta,c_\delta:\delta\in|T|^+)$ be a Morley sequence in a global coheir of $\tp(u,v,b,c/M_0)$ strictly invariant over $M_0$. I claim we may find $w$ such that, for every $\delta$ of form $\gamma+2i$ for $\gamma$ a limit ordinal and $i\in\omega$, $wg\psi(x,c_\delta)$ forks over $M_0$ for all $g\in G(M_0)$, and such that, for every $\delta$ of form $\gamma+2i+1$ for $\gamma$ a limit ordinal and $i\in\omega$, $wg\phi(x,b_\delta)$ forks over $M_0$ for all $g\in G(M_0)$; call this desired condition ($\ast$). Arguing by Fact \ref{lowness} and compactness as in the first paragraphs of Lemma \ref{union of conjugates wide implies wide} and Lemma \ref{union of conjunction is wide}, it is enough to show that, for each $n\in\omega$, there is $w_n$ such that $w_ng\phi(x,b_{2i-1})$ and $w_ng\psi(x,c_{2i})$ each fork over $M_0$ for all $i\in[n]$ and $g\in G(M_0)$.

    We do this by `interleaving' the $u_i$ and $v_i$ and using Lemma \ref{key observation}. More precisely, fix $n\in\omega$, and let $w_n=u_1v_2u_3v_4\ldots u_{2n-1}v_{2n}$. Now the sequence $$(u_1,b_1),(v_2,c_2),(u_3,b_3),\dots,(u_{2n-1},b_{2n-1}),(v_{2n},c_{2n})$$ satisfies the hypotheses of Lemma \ref{key observation} (with the even-indexed formulas being $\psi(x,z)$ and the odd-index formulas being $\phi(x,y)$). So by Lemma \ref{key observation} each $w_ng\phi(x,b_{2i-1})$ forks over $M_0$ for all $g\in G(M_0)$ and each $w_ng\psi(x,c_{2i})$ forks over $M_0$ for all $g\in G(M_0)$, exactly as needed.

    So we indeed get $w$ satisfying condition ($\ast$). Now we apply Fact \ref{bounded weight} to the sequence of quadruples $(b_{\gamma+2i},c_{\gamma+2i},b_{\gamma+2i+1},c_{\gamma+2i+1})$ to get some limit ordinal $\gamma$ and some $i\in\omega$ with $w\ind_{M_0}(b_{\gamma+2i},c_{\gamma+2i},b_{\gamma+2i+1},c_{\gamma+2i+1})$. (More precisely, given an ordinal $\delta=\gamma+i$, where $\gamma$ is a limit ordinal and $i\in\omega$, we define $d_\delta=(b_{\gamma+2i},c_{\gamma+2i},b_{\gamma+2i+1},c_{\gamma+2i+1})$, and then $(d_\delta:\delta\in |T|^+)$ is still a strict Morley sequence over $M_0$ to which we can apply Fact \ref{bounded weight}.) For notational convenience, write $\delta=\gamma+2i$, so that $w\ind_{M_0}(b_\delta,c_\delta,b_{\delta+1},c_{\delta+1})$; call this ($\ast\ast$).

    Recall that we assumed for contradiction that $\phi(x,b)\vee\psi(x,c)$ is wide. So $\phi(x,b_\delta)\vee\psi(x,c_\delta)$ is also wide. Let $s\subset G(M_0)$ be a finite set such that $s[\phi(x,b_\delta)\vee\psi(x,c_\delta)]$ is strong f-generic. Thus $ws[\phi(x,b_\delta)\vee\psi(x,c_\delta)]$ does not fork over $M_0$, and so (since the forkings sets form an ideal) there is some $g\in s$ such that $wg[\phi(x,b_\delta)\vee\psi(x,c_\delta)]$ does not fork over $M_0$. Thus (again since the forking sets form an ideal) one of $wg\phi(x,b_\delta)$ and $wg\psi(x,c_\delta)$ does not fork over $M_0$. Since $\delta=\gamma+2i$, by condition ($\ast$) $wg\psi(x,c_\delta)$ forks over $M_0$, so we must have that $wg\phi(x,b_\delta)$ does not fork over $M_0$. By ($\ast\ast$), we have $w\ind_{M_0}(b_\delta,b_{\delta+1})$, so by Fact \ref{non-forking iff invariant} $b_{\delta+1}\equiv_{(M_0,w)}b_\delta$, and thus $wg\phi(x,b_{\delta+1})$ also does not fork over $M_0$. But $\delta+1=\gamma+2i+1$, so this contradicts condition ($\ast$).
\end{proof}

Now combining Theorems \ref{non-wide sets are s1} and \ref{non-wide sets are an ideal} gives the following, which can be seen as one of the three main results of the paper and will be the essential technical tool in our applications. 

\begin{theorem}\label{non-wide sets are an s1 ideal}
    The non-wide definable subsets of $G$ are an S1 ideal.
\end{theorem}

\subsection{Piecewise strong f-genericity}\label{piecewise strong f-genericity section}
We point out here that the non-piecewise-strong-f-generic sets are also an ideal; the proof is an easy consequence of Theorem \ref{non-wide sets are an ideal}.

\begin{theorem}\label{piecewise strong f-generic sets are an ideal}Suppose $\phi(x,b)$ and $\psi(x,c)$ are not piecewise strong f-generic. Then neither is $\phi(x,b)\vee\psi(x,c)$.
\end{theorem}
\begin{proof}
    Suppose for contradiction that $\phi(x,b)\vee\psi(x,c)$ is piecewise strong f-generic. Thus there is a finite set $t\subset G(\mathfrak{C})$ such that $t[\phi(x,b)\vee\psi(x,c)]=t\phi(x,b)\vee t\psi(x,c)$ is strong f-generic. In particular, it is wide, and so by Theorem \ref{non-wide sets are an ideal} one of $t\phi(x,b)$ and $t\psi(x,c)$ is wide. In the first case, there is a finite set $s\subset G(M_0)$ such that $st\phi(x,b)$ is strong f-generic, in which case $\phi(x,b)$ is piecewise strong f-generic, and, in the second case, there is a finite set $s\subset G(M_0)$ such that $st\psi(x,c)$ is strong f-generic, in which case $\psi(x,c)$ is piecewise strong f-generic. Either way we get a contradiction.
\end{proof}

Theorem \ref{non-wide sets are an s1 ideal} is the driving force behind the main applications in this paper, and for the rest of this paper wideness will be the useful notion, not piecewise strong f-genericity. However, for other applications, the weaker notion of piecewise strong f-genericity is important; the reason is that the ideal of wide sets is not translation invariant, while by Remark \ref{translation invariance of piecewise (strong) f-genericity} the ideal of non-piecewise strong f-generic sets is translation invariant. As explained in the introduction, there are in general no translation-invariant S1 ideals, so we cannot hope to get both simultaneously. On the other hand, by Theorem \ref{non-wide sets are an s1 ideal} and respectively by Remark \ref{translation invariance of piecewise (strong) f-genericity} and Theorem \ref{piecewise strong f-generic sets are an ideal}, we do have two ideals: the S1 ideal of piecewise strong f-generic sets with witnesses in $M_0$, and the translation-invariant ideal of piecewise strong f-generic sets. These serve complementary roles, and in fact in our upcoming work on NIP approximate groups the latter notion will play an important role, in part because of Lemma \ref{stabilizer subgroup} below.

Let us now discuss a question that may be of interest for further investigation.

\begin{question}\label{ntp2 question}
    Do Theorem \ref{non-wide sets are an s1 ideal} and its consequence Theorem \ref{piecewise strong f-generic sets are an ideal} hold for definable groups in NTP$_2$ theories? In particular, do wide or at least piecewise strong f-generic types always exist in NTP$_2$ groups?
\end{question} The motivation for this question comes from two sources. First are the results of Pillay in \cite{pillay_simple}; as mentioned in the introduction, \cite{pillay_simple} shows that, for a group definable in a simple theory, f-genericity and strong f-genericity coincide and (their negation) gives rise to an ideal. So Question \ref{ntp2 question} has a strong positive answer in simple theories, and our results here show that it has a positive answer in NIP theories. One might thus hope to obtain a positive answer in NTP$_2$. The second motivation comes from the results of Montenegro, Onshuus, and Simon in \cite{montenegro_onshuus_simon}, which show that NTP$_2$ groups admitting a strong f-generic type have many good properties in common with definably amenable NIP groups, and apply this in particular to obtain applications to groups with strong f-generics definable in bounded PRC fields. So it seems realistic to hope that a positive answer to Question \ref{ntp2 question} might be able to lead to new and interesting theorems for definable groups in NTP$_2$.

Note that in definable groups in arbitrary theories one can always find weakly generic types. By Fact \ref{weak generic implies piecewise f-generic}, weakly generic sets are always what one might call `piecewise non-$G$-dividing'. If $T$ is NTP$_2$, then forking and dividing coincide over models, so non-$G$-dividing implies f-genericity, and hence in an NTP$_2$ group weakly generic sets are piecewise f-generic and thus one can always find global piecewise f-generic types. So the interesting question is whether piecewise \textit{strong} f-generic types exist. Our proof of Theorem \ref{non-wide sets are an s1 ideal} uses two essential things about NIP which do not hold in NTP$_2$, which are Fact \ref{non-forking iff invariant} and Fact \ref{lowness}, and it is not clear to me how one might try to substitute them.

\subsection{Type-definability}
For future reference, we record here the basic observation that all of the ideals we have defined are type-definable, in the sense that, for any formula $\phi(x,y)$, the set of $b$ such that $\phi(x,b)$ is small is type-definable over $M_0$. It follows basically immediately from Fact \ref{lowness}.

\begin{lemma}\label{type-definability}
    For any formula $\phi(x,y)$ implying $x\in G$ and any of the three properties defined in Definition \ref{piecewise f-generic definition}, the set of $b$ such that $\phi(x,b)$ does not satisfy the property is type-definable over $M_0$. 
\end{lemma}
\begin{proof}First we show it for piecewise strong f-genericity and wideness. By Fact \ref{lowness}, there is a partial type $\pi(u,y)$ such that $(a,b)\models\pi(u,y)$ if and only if $a\phi(x,b)$ forks over $M_0$. So now $\phi(x,b)$ is not wide if and only if $b\models \exists u\bigwedge_{g\in G(M_0)}\pi(ug,y)$, and it is not piecewise strong f-generic if and only if $b\models\bigwedge_{n\in\omega}\forall v_1\dots\forall v_n\exists u\bigwedge_{i\in[n]}\pi(uv_i,y)$. By compactness, these are both type-definable conditions. (For example, for the latter, it's equivalent to the conjunction of all formulas $\forall v_1\dots\forall v_n\exists u\bigwedge_{i\in[n]}\theta(uv_i,y)$ as $n$ ranges over the elements of $\omega$ and $\theta(u,y)$ ranges over the formulas of $\pi(u,y)$.)

For piecewise f-genericity, first we note that, for any formula $\psi(x,z)$, the set of $c$ such that $\psi(x,c)$ is not f-generic is type-definable. Indeed, by Fact \ref{lowness}, there is some $N_\psi$ such that, for any $(M,c)$-indiscernible sequence $(a_i:i\in\omega)$, if $\bigwedge_{i\in\omega}a_i\psi(x,c)$ is inconsistent then $\bigwedge_{i\in[N]}a_i\psi(x,c)$ is already inconsistent. Thus if $\pi_\psi(z)$ is the partial type containing the formula $$\exists u_1\dots \exists u_n\bigwedge_{s\subseteq[n],|s|=N}\neg\exists x\bigwedge_{i\in s}u_i\psi(x,y)$$ for every $n\in\omega$, then by Ramsey and compactness $c\models\pi$ if and only if $\psi(x,c)$ is not f-generic. For each $n$, let $\psi_n(x,u_1,\dots,u_n,y)$ be the formula $\bigvee_{i\in[n]}u_i\phi(x,y)$. Then $\phi(x,b)$ is not piecewise f-generic if and only if $b\models\bigwedge_{n\in\omega}\forall u_1\dots\forall u_n\pi_{\psi_n}(u_1,\dots,u_n,y)$.
\end{proof} Artem Chernikov has pointed out to me that Lemma \ref{type-definability} can be combined with his group chunk theorems in Section 4 of \cite{chernikov} to canonically recover an NIP group from associated type-definable data, an observation he had already made for definably amenable NIP groups using the type-definable ideal of non-f-generic sets.

We also have the following, again an easy consequence of Fact \ref{lowness}:
\begin{lemma}\label{piecewise f-generic coincides with piecewise strong f-generic for M0 definable sets}
    If $D$ is $M_0$-definable, then $D$ is piecewise f-generic if and only if it is piecewise strong f-generic if and only if it is piecewise strong f-generic with witnesses in $M_0$.
\end{lemma}
\begin{proof}It is enough to show that, if $D$ is piecewise f-generic, then it is piecewise strong f-generic with witnesses in $M_0$. So suppose $D$ is piecewise f-generic. Then there is a finite subset $t=\{g_1,\dots,g_n\}\subset G(\mathfrak{C})$ such that $tD$ is piecewise f-generic.  By the second paragraph of the proof of the lemma above, there is a partial type $\pi(u_1,\dots,u_n)$ such that $\bar{h}\models\pi$ if and only if $h_1D\vee\dots\vee h_nD$ is not f-generic. So $\bar{g}$ does not realize $\pi$, and hence there is $\theta(\bar{u})\in \pi$ such that $\bar{g}\models\neg\theta$. Since $\theta$ is $M_0$-definable, there is some tuple $\bar{g}'$ from $G(M_0)$ such that $\bar{g}'\models\neg\theta$, and now $g_1'D\vee\dots\vee g_n'D$ is f-generic. Since it is an $M_0$-definable set, it is also strong f-generic, so we are done.
\end{proof}

By Lemma \ref{translation invariance of piecewise (strong) f-genericity}, and Proposition \ref{piecewise f-genericity coincides with weak genericity} and Theorem \ref{piecewise strong f-generic sets are an ideal}, the non-piecewise (strong) f-generic sets are a translation-invariant ideal of definable subsets of $G$. By a standard argument, it follows that, if $\phi(x,b)$ is piecewise (strong) f-generic, then the set of $g$ such that $g\phi(x,b)\triangle\phi(x,b)$ is not piecewise (strong) f-generic is a subgroup of $G$, and by Lemma \ref{type-definability} this subgroup is type-definable. So we get:
\begin{lemma}\label{stabilizer subgroup}
    If $\phi(x,b)$ is piecewise (strong) f-generic, then the set $$\mathrm{st}(\phi(x,b))=\{g:g\phi(x,b)\triangle\phi(x,b)\text{ is not piecewise (strong) f-generic}\}$$ is a type-definable subgroup of $G$.
\end{lemma}

With a bit of care, everything in Section \ref{piecewise f-generic section}, including Lemma \ref{stabilizer subgroup}, goes through in the setting of $\bigvee$-definable groups $G=\langle X\rangle$ generated by a definable approximate group $X$. In my upcoming project on NIP approximate groups, the corresponding groups from Lemma \ref{stabilizer subgroup} will play an important role. I caution, however, that they need not have bounded index; indeed if $X$ is `non-laminar' in the sense of \cite{hrushovski_lascar_group}, then they cannot have bounded index. The easiest example is to take $X=\{-1,0,1\}\times\mathrm{SL}_2(\mathbb{R})$ conceived of as a subset of the definable group $\widetilde{\mathrm{SL}_2(\mathbb{R})}$ as treated in \cite{conversano_pillay} or \cite{krupinski_pillay}. We will discuss this in much greater detail in the upcoming paper. 

\section{`Existence' results}\label{existence section}
We \uline{continue assuming that $\dcl(\varnothing)=M_0$} is a model, and we continue following the terminology of Notation \ref{wideness definition}, so we use `wide' to abbreviate `piecewise strong f-generic with witnesses in $M_0$'. In this section we will give the first major application of Theorem \ref{non-wide sets are an ideal}: namely, that the Ellis group of $G(M)$ has size at most $2^{|T|}$ for any model $M$ (Theorem \ref{restriction map ellis group theorem}). As a `warmup', we will first give a new proof of Gismatullin's theorem on the `existence' of $G^{\infty}$ (Corollary \ref{existence of g000 theorem}). These two results are closed linked; indeed, as mentioned in our introduction, Krupi\'nski and Pillay observed in \cite{krupinski_pillay_early} that the Ellis group of $G(M)$ admits a canonical surjective epimorphism onto $G/G^{\infty}_M$; in particular, if the Ellis group has bounded size independent of the choice of $M$, then so too must $G/G^{\infty}_M$, and it then follows on formal grounds that $G^{\infty}$ `exists'. Our proofs are also linked, insofar as the starting point for both is the observation in Lemma \ref{key observation 2} below. However, Corollary \ref{existence of g000 theorem} is basically an immediate consequence of Lemma \ref{key observation 2}, whereas Theorem \ref{restriction map ellis group theorem} requires quite a bit more work. In particular, to apply Lemma \ref{key observation 2} to the Ellis group, we will need a way of translating our machinery of wideness to the setting of the Ellis group; inspired by the recent paper \cite{kyle_tomasz} of Gannon and Rzepecki, the translation will come via the retraction map $F_M$ from Section \ref{honest definitions section}. Combining Lemma \ref{key observation 2} with Fact \ref{definable p,q theorem} and suitably applying the retraction map, we will obtain a single main consequence for the Ellis semigroup of $G(M)$, Theorem \ref{downstairs lemma dynamical language}, from which Theorem \ref{restriction map ellis group theorem} will then follow on general topological dynamics grounds. 

Here is the key observation that will allow us to apply the machinery of Section \ref{piecewise f-generic section}.

\begin{lemma}\label{key observation 2}
    Suppose $\psi(x,c)$ is wide. For any $b\equiv_{M_0}b'$, there is some $g\in G(M_0)$ such that $bg\psi(x,c)\wedge b'g\psi(x,c)$ does not fork over $M_0$.
\end{lemma}
\begin{proof}
    Let $c'$ such that $(b,c)\equiv_{M_0}(b',c')$. By Theorem \ref{non-wide sets are s1}, $\psi(x,c)\wedge\psi(x,c')$ is wide, so let $s\subset G(M_0)$ be a finite set such that $s[\psi(x,c)\wedge\psi(x,c')]$ is strong f-generic. Thus in particular $b's[\psi(x,c)\wedge\psi(x,c')]$ does not fork over $M_0$, and so (since the forking sets form an ideal) there is $g\in s$ such that $b'g[\psi(x,c)\wedge\psi(x,c')]$ does not fork over $M_0$. By Fact \ref{non-forking iff invariant}, there is a global $M_0$-invariant type $p$ containing $b'g[\psi(x,c)\wedge\psi(x,c')]$, ie containing $b'g\psi(x,c)$ and $b'g\psi(x,c')$. We have $(b',c')\equiv_{M_0}(b,c)$, and $p\vdash b'g\psi(x,c')$, so by $M_0$-invariance also $p\vdash bg\psi(x,c)$. But now $p\vdash bg\psi(x,c)\wedge b'g\psi(x,c)$, so in particular that formula does not fork over $M_0$, and we are done.
\end{proof}

\subsection{`Existence' of $G^{\infty}$}\label{existence of g000 section}
Recall from Section \ref{connected components section} that, for a small model $M$, $G^{\infty}_M$ is the group generated by the set $X_M:=\{a^{-1}b:a,b\in G,a\equiv_M b\}$. It is a theorem of Gismatullin from \cite{gismatullin} that $G^\infty_M=G^\infty_N$ for any $M,N$. The proof from \cite{gismatullin} builds on ideas from an earlier theorem from \cite{shelah_g00}, that $G^{00}_M=G^{00}_N$ for any $M,N$, which in turn builds on ideas from the `Baldwin-Saxl' theorem. We give here a new and rather different proof, using the machinery of wide types we developed in Section 3. Our proof gives actually a slightly better result, insofar as we can show $X_N\subseteq X_M^2$ for any $M,N$, whereas the proof from \cite{gismatullin} gets no such bound.

\begin{theorem}
    For any $M$, we have $X_{M_0}\subseteq X_M^G$, where by $X_M^G$ we mean the set of conjugates $\{a^g=g^{-1}ag:a\in X_M,g\in G\}$.
\end{theorem}
\begin{proof}
 By Theorem \ref{non-wide sets are an ideal}, let $p$ be a global wide type. Let $b_0^{-1}b_1\in X_{M_0}$ be arbitrary, with $b_0\equiv_{M_0}b_1$. By Lemma \ref{key observation 2}, for any $\psi(x,c)\in p|_M$, there is some $g\in G(M_0)$ with $b_0g\psi(x,c)\wedge b_1g\psi(x,c)$ non-forking over $M_0$; in particular $b_0g\psi(x,c)\wedge b_1g\psi(x,c)$ is consistent. By compactness, we may thus find $f\in G(\mathfrak{C})$ (in fact $f\ind_{M_0}^u (M,b_0,b_1)$) such that $b_0f\psi(x,c)\wedge b_1f\psi(x,c)$ is consistent for every $\psi(x,c)\in p|_M$, ie such that $b_0f(p|_M)\wedge b_1f(p|_M)$ is consistent. Let $a$ be a realization. Then $f^{-1}b_0^{-1}a\models p|_M$ and $f^{-1}b_1^{-1}a\models p|_M$, so that $$X_M\ni (f^{-1}b_0^{-1}a)(f^{-1}b_1^{-1}a)^{-1}=f^{-1}b_0^{-1}b_1 f=(b_0^{-1}b_1)^f.$$ So $b_0^{-1}b_1\in X_M^{f^{-1}}$, and we are done.
\end{proof} By a standard trick (see eg Lemma 8.6 in \cite{simon_book}), we thus have:
\begin{corollary}\label{existence of g000 theorem} For any $M$ we have $X_{M_0}\subseteq X_M^2$, and hence $G^{\infty}_{M_0}= G^{\infty}_M$.
\end{corollary}

This recovers Gismatullin's theorem. Let us note a slightly stronger consequence, which we don't need here, but which will play an important role on our upcoming paper on NIP approximate groups.

\begin{lemma}\label{setminus conjugate not wide}
    For any formula $\psi(x,c)$ and any $a\in X_{M_0}$, the formula $\psi(x,c)\setminus a^G\psi(x,c)$ is not wide, where by $a^G$ we mean the conjugacy class of $a$.
\end{lemma}
\begin{proof}
    Suppose otherwise. Write $a=b_1^{-1}b_0$ for some $b_0\equiv_{M_0}b_1$. By Lemma \ref{key observation 2}, there is $g\in G(M_0)$ such that $$b_0g[\psi(x,c)\setminus a^G\psi(x,c)]\wedge b_1g[\psi(x,c)\setminus a^G\psi(x,c)]$$ is non-forking over $M_0$, and hence in particular consistent. Translating by $(b_1g)^{-1}$ gives in particular that $(b_1g)^{-1}b_0g\psi(x,c)\setminus a^G\psi(x,c)$ is consistent. But $(b_1g)^{-1}b_0g=a^g$, a contradiction.
\end{proof}

\begin{corollary}
    For any $a\in G(\mathfrak{C})$, we have $a\in X_{M_0}^G$ if and only if, for every formula $\psi(x,c)$, the formula $\psi(x,c)\setminus a^G\psi(x,c)$ is not wide.
\end{corollary}
\begin{proof}
    The forward direction is Lemma \ref{setminus conjugate not wide}. For the reverse direction, suppose $a$ is such that $\psi(x,c)\setminus a^G\psi(x,c)$ is never wide. By Theorem \ref{non-wide sets are an ideal}, let $p$ be a global wide type. Now for every $\psi(x,c)\in p$, we have $a^G\psi(x,c)\in p$. Thus, if $M$ is a small model containing $a$, and $b\models p|_M$, then $b\models a^G\psi(x,c)$ for all $\psi(x,c)\in p|_M$; by compactness there is thus $g\in G$ and $u\models p|_M$ such that $b=a^gu$. But now $a^g=bu^{-1}\in p|_Mp|_M^{-1}\subseteq X_{M_0}$, so indeed $a\in X_{M_0}^{g^{-1}}$.
\end{proof}

\begin{corollary}\label{twisted product}
    For any $a,b\in X_{M_0}$, there are $g,h\in G$ such that $a^gb^h\in X_{M_0}$.
\end{corollary}
\begin{proof} By Theorem \ref{non-wide sets are an ideal}, let $p$ be a global wide type. By Lemma \ref{setminus conjugate not wide}, for every $\psi(x,c)\in p$, also $a^{-G}\psi(x,c)\in p$ and $b^G\psi(x,c)\in p$, where by $a^{-G}$ we mean $(a^{-1})^G$. Let $M$ be a small model containing $a,b$; now if $e\models p|_M$, then $e\models x\in a^{-G}\psi(x,c)\wedge x\in b^G\psi(x,c)$ for every $\psi(x,c)\in p$, so by compactness we may find $u,v\models p|_M$ and $g,h\in G$ such that $e=a^{-g}u=b^hv$. But now $a^gb^h=uv^{-1}\in X_M\subseteq X_{M_0}$.
\end{proof}

Here is the relevant result whose analogue for $\bigvee$-definable groups will be important in our upcoming work on NIP approximate groups. It is an immediate consequence of Corollary \ref{twisted product} (just by working in the abelianization $G/[G,G]$):

\begin{theorem}
The set $X_{M_0}[G,G]$ is a subgroup of $G$, where $[G,G]$ is the derived subgroup of $G$. Since it contains $X_{M_0}$, it has bounded index in $G$.
\end{theorem}

\subsection{Technical lemma on the retraction map}\label{technical lemma on retraction map section}
Now we'll begin working towards the main theorem of this section. We first need a rather technical lemma on the retraction map from Section \ref{honest definitions section}; it is closely related to Fact 4.7 in \cite{kyle_tomasz}, since the $s$ in our lemma statement somewhat `resembles' a definable type with respect to the formula $\psi(y,v,z)$, by virtue of the formula $\eta(y,z)$. (In our use of this lemma, we will be getting the formula $\eta(y,z)$ from Fact \ref{definable p,q theorem}.) The fact that $s$ is not actually a definable type is remedied by the additional hypothesis that $q$ is finitely satisfiable in $M$, which is not needed in Fact 4.7 in \cite{kyle_tomasz}.

\begin{lemma}\label{retraction map lemma}Let $M\prec^+M'$ and work in the $L_\mathbf{P}$-structure $(M',M)$. Take an $|M'|^+$-saturated extension $(M',M)\prec^+(N',N)$. Let $\chi(y,v,e)$ be an $L(M')$-formula, and suppose that $\psi(y,v,c)$ is an $L(N)$-formula such that $$(N',N)\models\forall (y,v)\in\mathbf{P}[\psi(y,v,c)\to\chi(y,v,e)].$$ Now let $q(y),s(v)$ be $M$-invariant global types, and suppose that $q(y)$ is finitely satisfiable in $M$. Suppose that $\eta(y,z)$ is a formula such that (1) $q(y)\vdash\eta(y,c)$, and (2) $s(v)\vdash\psi(b,v,c')$ for all $(b,c')\models \eta$. (So in particular $s_v\otimes q_y\vdash\psi(y,v,c)$ and hence $F_M(s_v\otimes q_y)\vdash\chi(y,v,e)$.) Then $q_y\otimes F_M(s)_v\vdash\chi(y,v,e)$.
\end{lemma}
\begin{proof}
    Suppose for a contradiction that $q_y\otimes F_M(s)_v\vdash\neg\chi(y,v,e)$. Let $a\models s|_N(v)\wedge\mathbf{P}(v)$, so that by definition of $F_M$ we have $a\models F_M(s)|_{M'}$. (We think of $a$ as living in $\mathbf{P}(\mathfrak{C}')=\mathfrak{C}$, where $(\mathfrak{C}',\mathfrak{C}){\succ} (N',N)$ is a monster model.) Since $q$ is finitely satisfiable in $M=\mathbf{P}(M')$, the conjunction $q|_{(N',a)}(y)\wedge\mathbf{P}(y)$ is finitely satisfiable in $(\mathfrak{C}',\mathfrak{C})$; let $b\in\mathbf{P}(\mathfrak{C}')=\mathfrak{C}$ be a realization. In particular $b\models q|_{(M',a)}$, so $(b,a)\models [q\otimes F_M(s)]|_{M'}$, and thus $(b,a)\models\neg\chi(y,v,e)$. On the other hand, $(b,a)\models\mathbf{P}(y,v)$. By the assumption on $\psi(y,v,c)$, thus $(b,a)\models\neg\psi(y,v,c)$. So $b\models\neg\psi(y,a,c)$. Also, by the assumption (1), we have $b\models\eta(y,c)$. Since $q$ is finitely satisfiable in $M$, and since $b\models q|_{(M,a,c)}$, there is thus some $m\in M^y$ with $m\models \neg\psi(y,a,c)\wedge\eta(y,c)$. So now $(m,c)\models\eta$ and $a\models\neg\psi(m,v,c)$. Since $a\models s|_N\supseteq s|_{(M,c)}$, this contradicts assumption (2).
\end{proof}

\subsection{Main consequence for the Ellis group}\label{main consequence for ellis group section}
We give here the relevant consequence of Lemma \ref{key observation 2}, whose conclusions will allow us to apply Lemma \ref{retraction map lemma}.
\begin{lemma}\label{upstairs lemma} Let $M\succcurlyeq M_0$ be a model, and let $q_1,q_2\in S_G(\mathfrak{C})$ be global $M$-invariant types with $q_1|_{M_0}=q_2|_{M_0}$. Let $\psi(x,c)$ be a wide formula. Then there is an element $g\in G(M_0)$, an $L(M_0)$-formula $\eta(y,z)$ such that $\eta(y,c)$ lies in each $q_i|_{M_0},i=1,2$, and a global $M_0$-invariant type $s(v)\in S_G(\mathfrak{C})$ such that $s(v)\vdash\psi(gbv,c')$ for all $(b,c')\models\eta$. In particular, $s(v)\otimes q_i(y)\vdash\psi(gyv,c)$ for each $i=1,2$.
\end{lemma}
\begin{proof}Let $N$ be a model containing $(M,c)$, and let $b_1\models q_1|_N$ and $b_2\models q_2|_{N}$. By Lemma \ref{key observation 2}, there is some $g\in G(M_0)$ such that $b_1g\psi(x,c)\wedge b_2g\psi(x,c)$ does not fork over $M_0$. By Fact \ref{definable p,q theorem}, there is thus an $L(M_0)$-formula $\eta_0(y,y',z)\in\tp(b_1,b_2,c/M_0)$ such that $$\{bg\psi(x,c')\wedge b'g\psi(x,c'):(b,b',c')\models\eta\}$$ is consistent and non-forking over $M_0$. In particular, letting $\eta(y,z)\equiv\exists y_0\eta_0(y,y_0,z)\vee\exists y_0\eta_0(y_0,y,z)$, then $b_1\models \eta(y,c)$ and $b_2\models \eta(y,c)$, and the partial type $\{bg\psi(x,c'):(b,c')\models\eta(y,z)\}$ is consistent and non-forking over $M_0$. Thus by Fact \ref{non-forking iff invariant} we may find a global $M_0$-invariant type $s(v)$ such that $s(v)\vdash bg\psi(v,c')$ for all $(b,c')\models\eta$. Since $b_i\models q_i|_N$, in particular $b_i\models q_i|_{(M,c)}$, and so since $b_i\models\eta(y,c)$ thus $\eta(y,c)\in q_i$, and we are done.
\end{proof}

\begin{remark}\label{invariant ellis group remarks} Inspired by Gannon and Rzepecki's recent paper \cite{kyle_tomasz}, we will be using the retraction map from Section \ref{honest definitions section} to transfer Lemma \ref{upstairs lemma}, which is about global types invariant over $M$, into a statement about the Ellis semigroup of global types finitely satisfiable in $M$. If one wished to adopt the perspective of \cite{kyle_tomasz}, one could instead work directly in the `invariant Ellis semigroup' $S^{\mathrm{inv}}_G(\mathfrak{C},M)$. Here, just as in \cite{kyle_tomasz} and in Pillay's early related paper \cite{pillay_invariant_ellis_group}, the natural choice of type to work with would be `right' wide types, where we say a global type $p$ is `right' wide if $\psi(x,c)^{-1}$ is wide for all $\psi(x,c)\in p$. A right wide type is still $M_0$-invariant, and so naturally lives in $S^{\mathrm{inv}}_G(\mathfrak{C},M_0)$. By arguing as in the proof of Lemma \ref{upstairs lemma} we can get the following consequence. Suppose $p$ is a right wide type and that $M\succcurlyeq M_0$ is any model. Let $E$ be the `invariant' Ellis semigroup $S^{\mathrm{inv}}_G(\mathfrak{C},M)$. We naturally have $S^{\mathrm{inv}}_G(\mathfrak{C},M_0)\subseteq E$, and so $p\in E$. Then the proof of Lemma \ref{upstairs lemma} gives the following: for any $q_1,q_2\in E$ with $q_1|_{M_0}=q_2|_{M_0}$, and for any $\psi(x,c)\in p$, there are $g\in G(M_0)$ and $s\in S^{\mathrm{inv}}_G(\mathfrak{C},M_0)\subseteq E$ such that $gq_1s\vdash \psi(x,c)$ and $gq_2s\vdash\psi(x,c)$, where we compute $gq_1s$ and $gq_2s$ in the semigroup $E$. This could be useful if, following the themes of \cite{kyle_tomasz}, one wished to prove results about the `invariant Ellis group' in NIP groups that are not necessarily definably amenable. An obvious candidate question is whether the isomorphism type of the invariant Ellis group (which is only an abstract group, not a topological one, since there is no analogue of the $\tau$-topology) is independent of the choice of $M$. We do not pursue this question here.
\end{remark}

Combined with Lemma \ref{retraction map lemma} we can transfer Lemma \ref{upstairs lemma} to the Ellis semigroup $S^{\mathrm{fs}}_G(\mathfrak{C},M)$. Note that, if $p$ is a global wide type, then $p\in S_G^{\mathrm{inv}}(\mathfrak{C},M_0)\subseteq S_G^{\mathrm{inv}}(\mathfrak{C},M)$, so we may apply the retraction map $F_M$ from Fact \ref{retraction construction} to get an $M$-finitely satisfiable type $F_M(p)\in S_G^{\mathrm{fs}}(\mathfrak{C},M)$.

\begin{lemma}\label{downstairs lemma}
    Let $M\succcurlyeq M_0$ be a model, and let $q_1,q_2\in S_G^{\mathrm{fs}}(\mathfrak{C},M)$ be two global types finitely satisfiable in $M$ and such that $q_1|_{M_0}=q_2|_{M_0}$. Let $p\in S_G(\mathfrak{C})$ be a global wide type. For any $\chi(x,e)\in F_M(p)$, there are $g\in G(M_0)$ and $s\in S_G^{\mathrm{inv}}(\mathfrak{C},M_0)$ such that $g q_i F_M(s)\vdash\chi(x,e)$ for each $i=1,2$, where we are computing $gq_iF_M(s)$ in the Ellis semigroup $S^{\mathrm{fs}}_G(\mathfrak{C},M)$.
\end{lemma}
\begin{proof}
    This is an immediate consequence of Lemma \ref{retraction map lemma}, Lemma \ref{upstairs lemma}, and the definition of $F_M$. We spell out the details due to the regrettable technicality of these lemmas. Let $M'\supseteq (M,e)$ be an $|M|^+$-saturated model, and take an $|M'|^+$-saturated elementary extension $(M',M)\prec^+(N',N)$ of the pair $(M',M)$. By the definition of $F_M$ (see Fact \ref{retraction construction}) we have $p|_N(x)\wedge\mathbf{P}(x)\vdash F_M(p)|_{M'}(x)$, so $p|_N(x)\wedge \mathbf{P}(x)\vdash\chi(x,e)$, and thus by compactness there is a formula $\psi(x,c)\in p|_N$ such that $(N',N)\models\forall x\in\mathbf{P}(\psi(x,c)\to\chi(x,e))$; call this ($\ast$).

    Let $g\in G(M_0)$, $\eta(y,z)\in L(M_0)$, and $s(v)\in S^{\mathrm{inv}}_G(\mathfrak{C},M_0)$ be given by Lemma \ref{upstairs lemma} applied to the wide formula $\psi(x,c)$ and the $M$-invariant types $q_1,q_2$. So now (1) $\eta(y,c)\in q_i$ for each $i=1,2$, and (2) $s(v)\vdash\psi(gbv,c')$ for all $(b,c')\models\eta$. Since $g\in G(M_0)$, we have $g\models \mathbf{P}$, so $gba\models\mathbf{P}$ for all $(b,a)\models\mathbf{P}(y,v)$, and thus by ($\ast$) we have $$(N',N)\models\forall(v,y)\in\mathbf{P}[\psi(gyv,c)\to\chi(gyv,e)].$$ For each $i=1,2$, by Lemma \ref{retraction map lemma} applied to the $L(M')$-formula $\chi(gyv,e)$ and the $L(N)$-formula $\psi(gyv,c)$ we now get $
    (q_i)_y\otimes F_M(s)_v\vdash\chi(gyv,e)$ for each $i=1,2$. By definition of the semigroup operation on $S^{\mathrm{fs}}_G(\mathfrak{C},M)$, this means exactly that $q_iF_M(s)\vdash\chi(gx,e)$ for each $i=1,2$, so indeed $gq_iF_M(s)\vdash \chi(x,e)$ for each $i=1,2$. (Note that the expression $gq_iF_M(s)$ makes sense since $g\in G(M_0)\subseteq G(M)$, and we think of $S^{\mathrm{fs}}_G(\mathfrak{C},M)$ as a $G(M)$-flow.)
\end{proof}

The statement of Lemma \ref{downstairs lemma} is a bit technical. Let us translate it into a statement purely about the Ellis semigroup $S^{\mathrm{fs}}_G(\mathfrak{C},M)$, which will be easier to work with.

\begin{theorem}\label{downstairs lemma dynamical language} Let $M\succcurlyeq M_0$ be any model, and let $E=S^{\mathrm{fs}}_G(\mathfrak{C},M)$ be the Ellis semigroup of $G(M)$. Then there is an element $p_0\in E$ such that the following holds. Suppose $q_1,q_2\in E$ with $q_1|_{M_0}=q_2|_{M_0}$. Then, for any open neighborhood $U$ of $p_0$, there are $g_U\in G(M_0)$ and $t_U\in E$ such that $g_Uq_it_U\in U$ for each $i=1,2$.
\end{theorem}
\begin{proof}By Theorem \ref{non-wide sets are an ideal}, let $p\in S^{\mathrm{inv}}_G(\mathfrak{C},M_0)$ be a global wide type. We take $p_0=F_M(p)$. Then any open neighborhood $U$ of $p_0$ contains a clopen neighborhood $[\chi(x,e)]$ of $p_0$, and then we take $g_U$ and $t_U$ to be the $g$ and $F_M(s)$ given by Lemma \ref{downstairs lemma}.
\end{proof}

\subsection{Boundedness of the Ellis group}\label{boundedness of the ellis group section} Theorem \ref{downstairs lemma dynamical language} distills all of the consequences of our machinery of wideness that we will need to prove the main result, and this section can be read in a self-contained way using Theorem \ref{downstairs lemma dynamical language} as a `black box'. For this entire section, \uline{fix a model $M\succcurlyeq M_0$, and let $E=S^{\mathrm{fs}}_G(\mathfrak{C},M)$ be the Ellis semigroup of $G(M)$}. Also, \uline{fix a minimal left ideal $\mathcal{I}$ of $E$}. Let us first note that Theorem \ref{downstairs lemma dynamical language} immediately specializes to $\mathcal{I}$.

\begin{lemma}\label{downstairs lemma for ideal}
    Then there is $p\in \mathcal{I}$ such that the following holds. Suppose $q_1,q_2\in E$ with $q_1|_{M_0}=q_2|_{M_0}$. Then, for any open neighborhood $U$ of $p$, there are $g_U\in G(M_0)$ and $t_U\in \mathcal{I}$ such that $g_Uq_i t_U\in U$ for $i=1,2$.
\end{lemma}
\begin{proof}
    Let $p_0\in E$ be given by Theorem \ref{downstairs lemma dynamical language}. Let $r\in\mathcal{I}$ be arbitrary. Now $p_0r\in\mathcal{I}$, and we claim that taking $p=p_0r$ works. This is an immediate consequence of the continuity of the map $E\to E$, $x\mapsto xr$. Indeed, let $U$ be an open neighborhood of $p$. By continuity of the map $x\mapsto xr$, there is an open neighborhood $V$ of $p_0$ such that $Vr\subseteq U$. By Theorem \ref{downstairs lemma dynamical language}, there are $g\in G(M_0)$ and $t_0\in E$ such that $gq_it_0\in V$ for each $i=1,2$. Now we just take $g_U=g$ and $t_U=t_0r$.
\end{proof}

From hereonout, \uline{fix $p\in\mathcal{I}$ given by Lemma \ref{downstairs lemma for ideal}}. We will now show that, for any idempotent $u\in\mathcal{I}$, the restriction map $u\mathcal{I}\to S_G(M_0)$ taking an element $q\in u\mathcal{I}$ to the type $q|_{M_0}$ is injective; this will give the main result. Let us recall some standard notation, eg from \cite{auslander_book}:

\begin{notation}Given an element $a\in\mathcal{I}$, by Fact \ref{ellis semigroup lemma} there is a unique idempotent $w\in\mathcal{I}$ with $a\in w\mathcal{I}$. By $a^{-1}$ we mean the inverse of $a$ in the group $w\mathcal{I}$. So $a^{-1}\in w\mathcal{I}$ and $a^{-1}a=w=aa^{-1}$.
\end{notation}

We need one lemma about the $\tau$-topology, which is Lemma 12 in Chapter 14 of \cite{auslander_book}; the language there there is somewhat different than ours, so we include the proof here for the reader's convenience.

\begin{fact}\label{auslander lemma}Let $(P,<)$ be a downwards-directed poset. Let $u\in\mathcal{I}$ be an idempotent. Suppose that $(a_i:i\in P)$ is a net of elements in $u\mathcal{I}$ and that $(f_i:i\in P)$ is a net of elements in $G(M)$, where $G(M)$ is thought of as a subset of $E$, and suppose that $(f_i:i\in P)$ converges to some element $f\in E$ and that $(f_ia_i:i\in P)$ converges to some element $b\in \mathcal{I}$. Then some subnet of $(a_i:i\in P)$ converges to $u(fu)^{-1}b$ in the $\tau$-topology.
\end{fact}
\begin{proof}We follow the proof of Lemma 12 in Chapter 14 of \cite{auslander_book}. For each $k\in P$, let $A_k=\{a_i:i\leqslant k\}$; it is enough to show that $u(fu)^{-1}b\in\mathrm{cl}_\tau(A_k)$ for every $k\in P$. So fix such $k\in P$. Note that $f_ia_i\in f_i A_k$ for every $i\leqslant k$, and since $(f_i:i\leqslant k)$ converges to $f$ and $(f_ia_i:i\leqslant k)$ converges to $b$ thus by definition of the operator $\circ$ we have $b\in f\circ A_k$. Since $A_k\subseteq u\mathcal{I}$, we have $A_k=uA_k$, and thus by Fact \ref{composition for circ operation} $f\circ A_k=f\circ (uA_k)\subseteq (fu)\circ A_k$. So now $b\in (fu)\circ A_k$, so $u(fu)^{-1}b\in u(fu)^{-1}[(fu)\circ A_k]$, and so again using Fact \ref{composition for circ operation}, letting $w\in\mathcal{I}$ denote the idempotent of $\mathcal{I}$ with $fu\in w\mathcal{I}$, we have $$u(fu)^{-1}b\in u(fu)^{-1}[(fu)\circ A_k]\subseteq (u(fu)^{-1}(fu))\circ A_k=(uw)\circ A_k=u\circ A_k.$$ On the other hand, $u(fu)^{-1}b\in u\mathcal{I}$, so $u(fu)^{-1}b$ lies in $u\mathcal{I}\cap (u\circ A_k)$, which is by definition $\mathrm{cl}_\tau(A_k)$.
\end{proof}

So now Lemma \ref{downstairs lemma for ideal} has the following consequence:

\begin{lemma}\label{convergent nets in different ellis groups}
    Let $u,u'\in\mathcal{I}$ be any idempotents, and let $q\in u\mathcal{I}$ and $q'\in u'\mathcal{I}$. Suppose that $q|_{M_0}=q'|_{M_0}$. Then there is $f\in E$, in fact $f\in S^\mathrm{fs}_G(\mathfrak{C},M_0)$, and a net $(t_i:i\in P)$ of elements of $\mathcal{I}$, such that $(qt_i:i\in P)$ converges to $u(fu)^{-1}p$ in the $\tau$-topology on $u\mathcal{I}$ and $(q't_i:i\in P)$ converges to $u'(fu')^{-1}p$ in the $\tau$-topology on $u'\mathcal{I}$.
\end{lemma}
\begin{proof}
    By Lemma \ref{downstairs lemma for ideal}, for every open set $U\subseteq E$ with $p\in U$, there are $g_U\in G(M_0)$ and $t_U\in\mathcal{I}$ such that $g_Uqt_U\in U$ and $g_Uq't_U\in U$. Let $P$ be the downwards-directed poset of open neighborhoods of $p$ in $E$. Then clearly the nets $(g_Uqt_U:U\in P)$ and $(g_Uq't_U:U\in P)$ each converge to $p$. So if we define $a_U:=q t_U$ and $a'_U:=q't_U$ then the nets $(g_Ua_U:U\in P)$ and $(g_Ua'_U:U\in P)$ each converge to $p$. Since $q\in u\mathcal{I}$ we have $a_U\in u\mathcal{I}$, and since $q'\in u'\mathcal{I}$ we have $a'_U\in u'\mathcal{I}$.
    
    Consider $G(M_0)$ as a subset of $S^{\mathrm{fs}}_G(\mathfrak{C},M_0)\subseteq E$; then by compactness there is a subnet $(g_{U'}:U'\in P')$ of $(g_U:U\in P)$ that converges to some $f\in S^{\mathrm{fs}}_G(\mathfrak{C},M_0)$. Now by Fact \ref{auslander lemma} we get a subnet $(a_{U''}:U''\in P'')$ of $(a_{U'}:U'\in P')$ such that $(a_{U''}:U''\in P'')$ converges to $u(fu)^{-1}p$ in the $\tau$-topology, and again by Fact \ref{auslander lemma} we get a subnet $(a'_{U'''}:U'''\in P''')$ of $(a'_{U''}:U''\in P'')$ such that $(a'_{U'''}:U'''\in P''')$ converges to $u'(fu')^{-1}p$ in the $\tau$-topology. So the element $f$ and the net $(t_{U'''}:U'''\in P''')$ have the desired properties.
\end{proof}

Now we can get the main theorem. Recall the crucial Fact \ref{tau-topology is hausdorff}, that the $\tau$-topology on $u\mathcal{I}$ is Hausdorff for every idempotent $u\in\mathcal{I}$.

\begin{theorem}\label{restriction map ellis group theorem}
    Let $u\in\mathcal{I}$ be any idempotent. Then the restriction map $u\mathcal{I}\to S_G(M_0)$ taking an element $q\in u\mathcal{I}$ to the type $q|_{M_0}$ is injective. In particular, $|u\mathcal{I}|\leqslant |S_G(M_0)|\leqslant 2^{|T|}$.
\end{theorem}
\begin{proof}
    Suppose we have $q,q'\in u\mathcal{I}$ with $q|_{M_0}=q'|_{M_0}$. By Lemma \ref{convergent nets in different ellis groups}, there is $f\in E$ and a net $(t_i:i\in P)$ of elements of $\mathcal{I}$ such that $(qt_i:i\in P)$ and $(q't_i:i\in P)$ each converge to $u(fu)^{-1}p$ in the $\tau$-topology. We have $ut_i\in u\mathcal{I}$, so by compactness of the $\tau$-topology there is a subnet $(ut_{i'}:i'\in P')$ of $(ut_i:i\in P)$ and an element $t\in u\mathcal{I}$ such that $(ut_{i'}:i'\in P')$ converges to $t$ in the $\tau$-topology. But now $q=qu$, so $qt_i=(qu)t_i=q(ut_i)$, and similarly $q't_i=q'(ut_i)$. Since $q,q'$ and the $ut_i$ all lie in $u\mathcal{I}$, thus $(qt_{i'}:i'\in P')$ and $(q't_{i'}:i'\in P')$ converge respectively to $qt$ and $q't$ in the $\tau$-topology. (This does not use Fact \ref{tau-topology is hausdorff} yet, just the fact that $u\mathcal{I}$ is a semitopological group.) But now, by Fact \ref{tau-topology is hausdorff}, limits in the $\tau$-topology in $u\mathcal{I}$ are unique, and so we must have $qt=u(fu)^{-1}p$ and $q't=u(fu)^{-1}p$. So $qt=q't$, and now since $u\mathcal{I}\ni q,q',t$ is a group we have $q=q'$.
\end{proof}

\section{VC-sets and finite Archimedean rank}\label{vc sets far section}
This section can be seen as an interlude of the paper; we will prove in it some general facts about VC-sets in compact Hausdorff sets, which we will then apply to the Ellis group in Section \ref{FAR section} and to $G/G^{00}_\phi$ in Section \ref{local g00 section}. The reader may also wish to compare with Simon's paper \cite{simon_vc_sets}, which studies VC-sets in locally compact Hausdorff group, though with very different goals. The results themselves are completely self-contained and do not require any knowledge of model theory or Ellis groups; I hope they will perhaps be of independent interest. Let us remark at the outset that our results here are greatly inspired by the techniques of Hrushovski in \cite{hrushovski}. Nevertheless there are some novelties to our approach. The main one is to use the pseudometrics associated to Borel subsets of compact Hausdorff groups to substitute certain metric spaces that Hrushovski works with in \cite{hrushovski}, which is how we adapt his arguments to the setting of arbitrary compact groups. After this insight is made the material of the two subsections Section \ref{far orbit section} and Section \ref{vc sets and packing dimension section} become direct adaptations of arguments of Hrushovski's. On the other hand, our main result, Theorem \ref{general far theorem}, is quite different than anything that appears in \cite{hrushovski}. The idea in the proof there to work directly with the Haar measure on the connected component, though it might seem obvious to the reader in hindsight, was actually quite an important one, and in our applications in Section \ref{FAR section} and Section \ref{local g00 section} it is exactly what allows us to avoid the technical need from Hrushovski's results in Section 7 of \cite{hrushovski} to restrict to parameters from a single complete Shelah strong type.

Let us recall the following definition, following the terminology of \cite{hrushovski}:

\begin{definition}
    Let $(X,d)$ be a (pseudo)metric space of finite diameter. We say that $(X,d)$ has `packing dimension at most $\delta$' if there are at most $O(n^\delta)$ pairwise disjoint $d$-balls in $X$ of radius $1/n$.
\end{definition} Two basic observations are the following; the first is Lemma 7.9 in \cite{hrushovski}, and the second follows immediately from the definition.

\begin{fact}\label{lipschitz map packing dimension}
    If $f:Y\to X$ is a surjective Lipschitz map between metric spaces, then the packing dimension of $X$ is at most the packing dimension of $Y$.
\end{fact}

\begin{remark}\label{sqrt metric}
    Suppose $(X,d)$ is a pseudometric space of packing dimension at most $\delta$. Then the pseudometric space $(X,\sqrt{d})$ has packing dimension at most $2\delta$.
\end{remark}
% \begin{proof}
% Let $k$ be such that there are at most $kn^\delta$ disjoint $d$-balls of radius $1/n$. Note, suppose $a_1,\dots,a_s$ have pw disjoint $\sqrt{d}$-balls of radius $1/n$. We have $\sqrt{d}(a_i,b)\leqslant 1/n$ iff $d(a_i,b)\leqslant 1/n^2$. So the $a_i$ have pw disjoint $d$-balls around then. Hence $s\leqslant k(n^2)^\delta=kn^{2\delta}$.
% \end{proof}

Throughout this section, we will work with the pseudometrics discussed in Appendix A. Thus:

\begin{definition} For $G$ a compact Hausdorff group with normalized Haar measure $\eta$, and $B\subseteq G$ a Borel subset, we define $d_B:G\times G\to [0,1]$ by $d_B(g,h)=\eta(gB\triangle hB)$, and we define $\ker(d_B):=\{g:d_B(g,1)=0\}$.
\end{definition}So, by Lemma \ref{pseudometric and closed subgroup lemma}, $d_B$ is a left-invariant pseudometric on $G$, and $\ker(d_B)$ is a closed subgroup of $G$.

\subsection{Lemmas on compact groups}\label{far orbit section}
We begin by isolating some useful arguments from \cite{hrushovski} that apply to any connected compact group. Recall the basic facts from the representation theory of compact groups recalled in Section \ref{compact groups section}. The following is Lemma 7.10 from \cite{hrushovski}:

\begin{fact}\label{dimension bound for lie groups} Let $G$ be a connected compact Lie group and let $H$ be a Hilbert module of $G$. Let $X\subseteq H$ be an orbit of $G$ and suppose that $G$ acts faithfully on $X$. Let $\delta$ be the packing dimension of $X$ as a metric subspace of $H$. Then $\dim(G)\leqslant\delta^2$.
\end{fact}

We observe here that Fact \ref{dimension bound for lie groups} implies an analogous result for arbitrary connected compact Hausdorff groups. This is an easy consequence of Peter-Weyl, and in fact the proof of Fact \ref{dimension bound for lie groups} in \cite{hrushovski} basically already shows it.

\begin{lemma}\label{dimension bound for arbitrary compact groups} Let $G$ be a connected compact Hausdorff group and let $H$ be a Hilbert module of $G$. Let $X\subseteq H$ be an orbit of $G$ and suppose that $G$ acts faithfully on $X$. Let $\delta$ be the packing dimension of $X$ as a metric subspace of $H$. Then $G$ is an inverse limit of compact Lie groups of dimension at most $\delta^2$.
\end{lemma}
\begin{proof}By Peter-Weyl, there is an orthogonal family $(H_i:i\in I)$ of finite-dimensional $G$-invariant subspaces of $H$ such that $\bigoplus_{i\in I}H_i$ is dense in $H$. For each finite $s\subseteq I$, let $H_s=\bigoplus_{i\in s}H_i$, and let $\pi_s:H\to H_s$ denote the orthogonal projection onto $H_s$. Since $H_s$ is $G$-invariant, $\pi_s$ commutes with the action of $G$ (see eg Lemma 2.18 in \cite{hofmann_morris}), and so $\pi_s(X)$ is an orbit of $G$ in $H_s$. Let $K_s$ be the kernel of the action of $G$ on $\pi_s(X)$, ie let $K_s$ be the set of $g\in G$ such that $g$ acts as the identity on $\pi_s(X)$; now $K_s$ is a closed normal subgroup of $G$ and $G/K_s$ acts faithfully on $\pi_s(X)$, and $\pi_s(X)$ is an orbit of $G/K_s$. Let $H'$ be the subspace of $H_s$ generated by $\pi_s(X)$; since $\pi_s(X)$ is an orbit of $G$, $H'$ is still a $G$-invariant subspace of $H$, and $K_s$ acts as the identity on $H'$, so $H'$ is a $G/K_s$-module.

Since $G$ is connected, so is $G/K_s$. By Fact \ref{lipschitz map packing dimension}, the packing dimension of $\pi_s(X)$ is at most that of $X$, and hence at most $\delta$. On the other hand, each $G/K_s$ is a compact Lie group. (Indeed, if $N_s$ is the kernel of the action of $G$ on $H_s$, then $K_s\supseteq N_s$, so $G/K_s$ is a quotient of $G/N_s$, and now $G/N_s$ acts faithfully on the finite-dimensional vector space $H_s$, hence is a compact Lie group, whence $G/K_s$ is a compact Lie group too.) So $G/K_s$ is a connected compact Lie group acting faithfully on its orbit $\pi_s(X)\subseteq H'$, and hence by Fact \ref{dimension bound for lie groups} we have $\dim(G/K_s)\leqslant\delta^2$.

So the $G/K_s$ are all compact Lie groups of dimension at most $\delta^2$. On the other hand, note that $\bigcap_sK_s=\{1\}$; indeed, suppose $g\in\bigcap_sK_s$. Since $G$ acts faithfully on $X$, to show $g=1$ it is enough to show that $g$ acts as the identity on $X$. So fix an arbitrary $x\in X$. By density of $\bigoplus_{i\in I}H_i$, there are $s_i,i\in\omega$ and $x_i\in \pi_{s_i}(X)$ such that $\lim_i x_i= x$. Since $g\in K_{s_i}$, $gx_i=x_i$ for each $i$, and so by continuity $gx=x$. So indeed $g=1$. Also, we have $K_s\cap K_t\supseteq K_{s\cup t}$ for each $s,t$. So $G$ is the inverse limit of the $G/K_s$ and the result follows.
\end{proof}

Now we isolate one of the main abstract arguments from the proof of Theorem 7.12 in \cite{hrushovski}, but adapted to our more general setting.

\begin{lemma}\label{far bound assuming action on orbit is faithful}
    Let $G$ be a connected compact Hausdorff group with normalized Haar measure $\eta$. Let $B\subseteq G$ be a Borel subset, and suppose that the action of $G$ on the space of left cosets of $\ker(d_B)$ is faithful, and that the pseudometric space $(G,d_B)$ has packing dimension at most $\delta$. Then $G$ is an inverse limit of compact Lie groups of dimension at most $(2\delta)^2$.
\end{lemma}
\begin{proof}
By Lemma \ref{pseudometric and closed subgroup lemma}, $\ker(d_B)$ is a closed subgroup of $G$. Let $H=L^2(G,\eta)$ and let $\lambda:G\to U(H)$ be the regular representation. Define the map $i:G\to H$ by taking $i(g)$ to be (the $L^2$-equivalence class of) the indicator function of $gB$. Note that, for any $g,h\in G$, $|i(g)-i(h)|$ is identically $1$ on $gB\triangle hB$ and identically $0$ outside it, so we have 
$$||i(g)-i(h)||_2^2=\int_{G}(i(g)-i(h))^2d\eta=\eta(gB\triangle hB)=d_B(g,h),\ \ \text{(}\ast\text{)}$$ so that $i$ is an isometry of (pseudo)metric spaces $(G,\sqrt{d_B})\to H$. In particular, the packing dimension of the metric subspace $i(G)\subseteq H$ is the same as that of $(G,\sqrt{d_B})$, and hence by Remark \ref{sqrt metric} at most $2\delta$.

Note that $\lambda(g)(i(h))=i(gh)$ for all $g,h\in G$, so that $i(G)$ is a single orbit of $G$. By assumption, the action of $G$ on the space of left cosets $G/\ker(d_B)$ is faithful, or in other words the intersection of all conjugates of $\ker(d_B)$ in $G$ is trivial. It follows that the action of $G$ on $i(G)$ is faithful; indeed, suppose $\lambda(g)(i(x))=i(x)$ for all $x\in G$, ie that $i(gx)=i(x)$ for all $x\in G$. By ($\ast$), we hence have $d_B(gx,x)=0$ for all $x\in G$, and thus by left-invariance $d_B(x^{-1}gx,1)=0$ for all $x\in G$. In other words, $x^{-1}gx\in\ker(d_B)$ for all $x\in G$, so $g$ lies in every conjugate of $\ker(d_B)$, and hence $g=1$.

So now $i(G)\subseteq H$ is an orbit of $G$ of packing dimension at most $2\delta$ and on which $G$ acts faithfully. The claim follows from Lemma \ref{dimension bound for arbitrary compact groups}.
\end{proof}

\subsection{VC-sets and packing dimensions}\label{vc sets and packing dimension section}
Following the terminology of \cite{simon_vc_sets}, let us say that a subset of a group is a `VC-set' if the family of left translates of the set has finite VC-dimension. If $G$ is a compact group and $B$ is a Borel subset, and $B$ is VC-set, we are going to apply the facts from Section \ref{vc theory section} to obtain a bound on the packing dimension of the associated pseudometric space $(G,d_B)$. The argument is from Proposition 3.27 in \cite{hrushovski}, which attributes it to Lovász and Szegedy.

\begin{lemma}\label{packing dimension lemma}
    Let $G$ be a compact Hausdorff group with normalized Haar measure $\eta$, and let $B\subseteq G$ be a Borel subset. Suppose that the family of left translates of $B$ has VC-density at most $\delta_0$. Then the pseudometric space $(G,d_B)$ has packing dimension at most $2\delta_0$.
\end{lemma}
\begin{proof} Let $\mathcal{F}=\{gB:g\in G\}$, so by assumption $\mathcal{F}$ has VC-density at most $\delta_0$. Let $K$ be such that $\pi_\mathcal{F}(n)\leqslant Kn^{\delta_0}$ for all $n$. Define a new set system $\mathcal{F}':=\{gB\setminus hB:g,h\in G\}$. It is easy to see that the VC-density of $\mathcal{F}'$ is at most $2\delta_0$. Indeed, more precisely, given a subset $S\subseteq G$, there is a natural surjection from $\mathcal{F}|_S\times\mathcal{F}|_S$ to $\mathcal{F}'|_S$, taking $(gB\cap S,hB\cap S)$ to $(gB\cap S)\setminus (hB\cap S)=(gB\setminus hB)\cap S$, so $|\mathcal{F}'|_S|\leqslant|\mathcal{F}|_S|^2$, and hence $\pi_{\mathcal{F}'}(n)\leqslant\pi_\mathcal{F}(n)^2\leqslant K^2n^{2\delta_0}$ for all $n$. In other words, if we let $C=K^2$ and $\delta=2\delta_0$, then $\pi_{\mathcal{F}'}(n)\leqslant Cn^{\delta}$ for all $n$. Let $D$ be given by Fact \ref{epsilon nets} for $C$ and $\delta$.

Now, our aim is to show that $(G,d_B)$ has packing dimension at most $2\delta_0$. So suppose that $a_1,\dots,a_s\in G$ are surrounded by pairwise disjoint $d_B$-balls of radius $1/n$; we need to show $s=O(n^{2\delta_0})$. For each $i\neq j$ we have $d_B(a_i,a_j)\geqslant 1/n$, ie $\eta(a_iB\triangle a_jB)\geqslant 1/n$, and so one of $\eta(a_iB\setminus a_jB)\geqslant 1/2n$ and $\eta(a_jB\setminus a_iB)\geqslant 1/2n$ holds. Call this ($\ast$).

Let $\mathcal{F}'_0=\{a_iB\setminus a_jB:i,j\in[s],i\neq j\}$. Now $\mathcal{F}'_0$ is a subset of the family $\mathcal{F}'$, so $\pi_{\mathcal{F}'_0}(n)\leqslant\pi_{\mathcal{F}'}(n)\leqslant Cn^\delta$ for all $n$, and $\mathcal{F}'_0$ is also finite, so applying Fact \ref{epsilon nets} to $(G,\eta,\mathcal{F}'_0)$ we may find $b_1,\dots, b_N\in G$, with $N=D(2n)^2=4Dn^2$, such that, for all $i\neq j$, if $\eta(a_iB\setminus a_jB)\geqslant 1/2n$ then there is some $b_k\in a_iB\setminus a_jB$. But the $a_iB$ comes from the set system $\mathcal{F}$, so it follows that $s\leqslant\pi_{\mathcal{F}}(N)\leqslant KN^{\delta_0}=4^{\delta_0}KD^{\delta_0}n^{2\delta_0}$, and the desired result follows.
\end{proof}

\subsection{Main result}\label{general borel lemma section}
Now we can give the main consequences.

\begin{lemma}\label{general local far result}
    Suppose that $G=G^0$ is a connected compact Hausdorff group, and that $B\subseteq G$ is a Borel subset such that the family of translates of $B$ has VC-density at most $\delta$. Let $N$ be the kernel of the action of $G$ on the space of left cosets of $\ker(d_B)$, ie the intersection of all conjugates of $\ker(d_B)$. Then $G/N$ is an inverse limit of compact Lie groups of dimension at most $(4\delta)^2$.
\end{lemma}
\begin{proof} Note that $N$ is a closed normal subgroup of $G$; indeed it is the intersection of all the conjugates of the closed subgroup $\ker(d_B)$. By Lemma \ref{haar measure descends}, there is a Borel subset $C\subseteq G/N$ such that $d_C(gN,hN)=d_B(g,h)$ for all $g,h\in G$. In particular, the packing dimension of the pseudometric space $(G/N,d_C)$ is the same as that of $(G,d_B)$, and hence by Lemma \ref{packing dimension lemma} is at most $2\delta$. Also, again since $d_C(gN,hN)=d_B(g,h)$ for all $g,h$, we have $\ker(d_C)=\ker(d_B)/N$. So $G/N$ acts faithfully on the space of left cosets of $\ker(d_C)$. Since $G/N$ is connected, we are done by Lemma \ref{far bound assuming action on orbit is faithful}.
\end{proof}

\begin{theorem}\label{general far theorem}Suppose $G$ is a compact Hausdorff group, and suppose there is a countable family $(B_i:i\in I)$ of Borel subsets of $G$ such that (1) every open subset of $G$ can be written as a union of some of the $B_i$, and (2) there is some $\delta$ such that, for each $i\in I$, the family of left translates $(gB_i:g\in G)$ has VC-density at most $\delta$. Then $G$ is an inverse limit of compact Lie groups of dimension at most $(4\delta)^2$.
\end{theorem}
\begin{proof}
    By Lemma \ref{archimedean rank of connected component}, it is enough to show that $G^0$ is an inverse limit of compact Lie groups of dimension at most $(4\delta)^2$. Let $\eta_0$ be the normalized Haar measure on $G^0$. For each $i\in I$, let $B^0_i=B_i\cap G^0$. Clearly $B_i^0$ is a Borel subset of $G^0$ and the family $(gB^0_i:g\in G^0)$ has VC-density at most $\delta$. Thus let $d_i$ be the pseudometric $d_{B^0_i}$ on $G^0$, and let $N_i$ be the intersection of all the conjugates of $\ker(d_{B^0_i})$ by elements of $G^0$. Then $N_i$ is a closed normal subgroup of $G^0$ and by Lemma \ref{general local far result} $G^0/N_i$ is an inverse limit of compact Lie groups of dimension at most $(4\delta)^2$.

    Now, we wish to show that $G^0$ is an inverse limit of compact Lie groups of dimension at most $(4\delta)^2$. By Peter-Weyl, $G^0$ is an inverse limit of compact Lie groups, and so by Lemma \ref{lie quotient of far group} and the above paragraph it is enough to show that, if $N$ is any closed normal subgroup of $G^0$ such that $G^0/N$ is a Lie group, then there is some $i\in I$ with $N_i\leqslant N$.

    So fix such an $N$. Lie groups are NSS, so let $O\subseteq G^0/N$ be an open neighborhood of the identity containing no nontrivial subgroup of $G^0/N$. Let $V\subseteq G^0/N$ be an open neighborhood of the identity such that $VV^{-1}\subseteq O$. Let $U\subseteq G^0$ be the preimage of $V$ under the quotient map $G^0\to G^0/N$. Then any subgroup of $G^0$ contained in $UU^{-1}$ is contained in $N$; call this ($\ast$).

    Since $U\subseteq G^0$ is open in $G^0$, let $U'\subseteq G$ be an open subset of $G$ with $U=G^0\cap U'$. Now by assumption (1) there is some $J\subseteq I$ with $U'=\bigcup_{i\in J}B_i$. In particular, $U=\bigcup_{i\in J}G^0\cap B_i=\bigcup_{i\in J}B_i^0$. Since $U$ is non-empty open, $\eta_0(U)>0$, and so since $J$ is countable by countable additivity there is thus some $i\in J$ with $\eta_0(B_i^0)>0$.

    It then follows that $\ker(d_i)\subseteq B_i^0(B_i^0)^{-1}$; indeed suppose $g\in\ker(d_i)$. By definition then $\eta_0(gB_i^0\triangle B_i^0)=0$. Since $\eta_0(B_i^0)>0$ it follows in particular that $gB_i^0\cap B_i^0\neq\varnothing$, so that indeed $g\in B_i^0(B_i^0)^{-1}$. But now $\ker(d_i)\subseteq UU^{-1}$, so by ($\ast$) $\ker(d_i)\leqslant N$, and since $N_i\leqslant\ker(d_i)$ we are done.
\end{proof}

\section{The finer structure of the Ellis group}\label{FAR section}
In this section we will prove a strong structural theorem about the Ellis group; under a strengthening of NIP of `bounded VC-codensity', we will be able to prove that, if $T$ and $M_0$ are countable, then the Ellis group of $G(M_0)$ is of `finite Archimedean rank' in the sense of Section \ref{compact groups section}, ie it is profinite-by-Lie-by-profinite; this is Theorem \ref{FAR main theorem}. Our proof has three ingredients; the first is the structural results about the Ellis group proved by Chernikov, Gannon, and Krupi\'nski in \cite{chernikov_gannon_krupinski}, the second is Theorem \ref{general far theorem} above (which as discussed in the beginning of Section \ref{vc sets far section} is greatly inspired by Hrushovski's techniques from \cite{hrushovski}), and the third is Theorem \ref{restriction map ellis group theorem}.

\uline{Throughout this section, continue to assume that $\dcl(\varnothing)=M_0$ is a model and that $T$ is NIP. Additionally, let $E$ be the Ellis semigroup $S_G^{\mathrm{fs}}(\mathfrak{C},M_0)$, let $\mathcal{I}$ be a minimal left ideal of $E$, and let $u\in\mathcal{I}$ be an idempotent. Finally, let $\mathcal{G}=u\mathcal{I}$ equipped with the $\tau$-topology.} So by Fact \ref{tau-topology is hausdorff} $\mathcal{G}$ is a compact Hausdorff group. To ease some notational strain, we give the following:

\begin{notation}
    Given an $L(\mathfrak{C})$-formula $\psi(x,c)$, let $[\psi(x,c)]$ denote the clopen subset of $E$ of types concentrated on $\psi(x,c)$, and let $\langle\psi(x,c)\rangle=[\psi(x,c)]\cap u\mathcal{I}$.
\end{notation} So Fact \ref{borel definability}, proved in \cite{chernikov_gannon_krupinski}, says that each $\langle\psi(x,c)\rangle$ is a Borel subset of $u\mathcal{I}$ in the $\tau$-topology. A core insight of \cite{chernikov_gannon_krupinski} is that, by virtue of this, the Haar measure on the Ellis group of an arbitrary NIP group can in many applications substitute translation-invariant measures of a definably amenable NIP group. We are greatly influenced by this idea.

Our first step will be to use Theorem \ref{restriction map ellis group theorem} to show that the $\langle\theta\rangle$ for $\theta(x)$ an $L(M_0)$-formula behave like a `basis' for the $\tau$-topology, in the sense that every $\tau$-open subset of $u\mathcal{I}$ is a union of sets of that form. (They are only Borel, not open, so they are not actually a basis.) This is Corollary \ref{tau-open is countable union}.

\subsection{`Second countability'}
Let $\overline{u\mathcal{I}}$ be the closure of $u\mathcal{I}$ in $E$, and let $f:\overline{u\mathcal{I}}\to u\mathcal{I}$ be the map $x\mapsto ux$. By Fact \ref{continuity of ux}, $f$ is a continuous map from the Stone topology inherited from $E$ to the $\tau$-topology.
\begin{lemma}\label{preimage is countable union}
    For every $\tau$-open set $U\subseteq u\mathcal{I}$, $f^{-1}(U)$ is a union of sets of form $[\theta(x)]\cap\overline{u\mathcal{I}}$ for $\theta(x)$ an $L(M_0)$-formula.
\end{lemma}
\begin{proof}
    By Theorem \ref{restriction map ellis group theorem}, the restriction map $S^{\mathrm{fs}}_G(\mathfrak{C},M_0)\to S_G(M_0)$ is injective on $u\mathcal{I}$. It follows that, for any $p,p'\in\overline{u\mathcal{I}}$, if $p|_{M_0}=p'|_{M_0}$, then $f(p)=f(p')$. Indeed, by definition of the semigroup operation on $E=S^{\mathrm{fs}}_G(\mathfrak{C},M_0)$, the assumption $p|_{M_0}=p'|_{M_0}$ gives that $(qp)|_{M_0}=(qp')|_{M_0}$ for every $q\in E$, so in particular $(up)|_{M_0}=(up')|_{M_0}$. But $up,up'\in u\mathcal{I}$, so by injectivity of the restriction map $up=up'$, ie $f(p)=f(p')$, as claimed. Thus, for any $p\in f^{-1}(U)$, if $p'\in \overline{u\mathcal{I}}$ and $p'|_{M_0}=p|_{M_0}$ then $p'\in f^{-1}(U)$. Call this observation ($\ast$).

    We wish to show that $f^{-1}(U)$ is a union of sets of form $[\theta(x)]\cap\overline{u\mathcal{I}}$, with $\theta(x)\in L(M_0)$. So fix $p\in f^{-1}(U)$; we need to find $\theta(x)\in L(M_0)$ such that $p\vdash\theta(x)$ and $[\theta(x)]\cap\overline{u\mathcal{I}}\subseteq f^{-1}(U)$. By observation ($\ast$), for any $q\in\overline{u\mathcal{I}}\setminus f^{-1}(U)$ we have $q|_{M_0}\neq p|_{M_0}$, so there is $\theta_q(x)\in L(M_0)$ with $p\vdash\theta_q(x)$ and $q\vdash\neg\theta_q(x)$. Now the $([\neg\theta_q(x)]:q\in\overline{u\mathcal{I}}\setminus f^{-1}(U))$ are a Stone-open cover of $\overline{u\mathcal{I}}\setminus f^{-1}(U)$.

    On the other hand, by Fact \ref{continuity of ux}, $f^{-1}(U)$ is a Stone-open subset of $\overline{u\mathcal{I}}$, so $\overline{u\mathcal{I}}\setminus f^{-1}(U)$ is a Stone-closed subset of $E$, and is hence in particular Stone-compact. Thus there are finitely many $q_1,\dots,q_n\in\overline{u\mathcal{I}}\setminus f^{-1}(U)$ with $\overline{u\mathcal{I}}\setminus f^{-1}(U)\subseteq\bigcup_{i\in[n]}[\neg\theta_{q_i}(x)]$. Let $\theta(x)\equiv\bigwedge_{i\in[n]}\theta_{q_i}(x)$. Now $p(x)\vdash\theta(x)$. On the other hand, if $q\in\overline{u\mathcal{I}}\setminus f^{-1}(U)$, there there is $i\in[n]$ with $q\vdash\neg\theta_{q_i}(x)$, and hence $q\vdash\neg\theta(x)$. So in other words $\overline{u\mathcal{I}}\cap [\theta(x)]\subseteq f^{-1}(U)$, as needed.
\end{proof}
\begin{corollary}\label{tau-open is countable union}
    Every $\tau$-open set $U\subseteq u\mathcal{I}$ is a union of sets of form $\langle\theta(x)\rangle$ for $\theta(x)$ an $L(M_0)$-formula.
\end{corollary}
\begin{proof}
    By Lemma \ref{preimage is countable union}, let $\theta_i(x)$ be $L(M_0)$-formulas with $f^{-1}(U)=\bigcup_i([\theta_i(x)]\cap\overline{u\mathcal{I}})$. We claim that $U=\bigcup_{i}\langle\theta_i(x)\rangle$, ie that $U=\bigcup_{i}([\theta_i(x)]\cap u\mathcal{I})$. To see the inclusion $\subseteq$, suppose $p\in U$. Since $p\in u\mathcal{I}$, $p=up=f(p)$, so $p\in f^{-1}(U)$, and hence there is $i$ with $p\in [\theta_i(x)]\cap\overline{u\mathcal{I}}$, so in particular $p\in[\theta_i(x)]$ and hence $p\in[\theta_i(x)]\cap u\mathcal{I}$. To see the inclusion $\supseteq$, suppose $p\in[\theta_i(x)]\cap u\mathcal{I}$. So in particular $p\in [\theta_i(x)]\cap \overline{u\mathcal{I}}$, and so $p\in f^{-1}(U)$, ie $f(p)\in U$. But $p\in u\mathcal{I}$, so again $f(p)=up=p$, as needed.
\end{proof}

Note that, if $T$ and $M_0$ are countable, then there are only countably many $L(M_0)$-formulas, and so the previous two lemmas tell us that the sets $\langle\theta\rangle$ behave like a countable basis for the $\tau$-topology, in the sense that every $\tau$-open subset of $\mathcal{G}$ is a union of sets of that form. (Again, as remarked above, note that they are not actually a countable basis, since they are only Borel and not open.) This is what will allow us to apply Theorem \ref{general far theorem}.

\subsection{Main result}\label{far main result section}
To apply Theorem \ref{general far theorem}, we need to show that the subsets $\langle\theta\rangle$ of $\mathcal{G}$ are VC-sets. Thus is an immediate consequence of NIP. Indeed:

\begin{lemma}\label{theta is vc set in ellis group}
    Let $\theta(x)$ be an $L(M_0)$-formula, and let $\mathcal{F}$ be the family of left translates $\{p\langle\theta\rangle: p\in \mathcal{G}\}$. If the formula $\theta(xy)$ has VC-codensity at most $\delta$, then $\mathcal{F}$ has VC-density at most $\delta$. In fact, if $C,\delta$ are such that the number of $\phi$-types over a subset of $G(\mathfrak{C})$ of size $n$ is at most $Cn^\delta$ for all $n$, then $\pi_\mathcal{F}(n)\leqslant Cn^\delta$ for all $n$.
\end{lemma}
\begin{proof}
    Recall that $\pi_\mathcal{F}(n)$ is the maximal possible size of $\mathcal{F}|_B:=\{F\cap B:F\in\mathcal{F}\}$ as $B$ ranges over the subsets of $\mathcal{G}$ of size at most $n$. So fix some $B\subseteq\mathcal{G}$ of size at most $n$; write $B=\{q_1,\dots,q_m\}$, where $m\leqslant n$. Recalling that $\mathcal{G}\subseteq E= S^{\mathrm{fs}}_G(\mathfrak{C},M_0)$, for each $q_i\in B$ let $b_i\models q_i|_{M_0}$. We will show that distinct sets in $\mathcal{F}|_B$ give rise to distinct $\theta(xy)$-types over $\{b_1,\dots,b_m\}$, which will give the result.

    Let $p_k\in\mathcal{G},k=1,\dots,s$ be such that the $p_k\langle\theta\rangle\cap B$ are pairwise distinct. For each $k\in [s]$, let $p_k^{-1}$ denote the inverse of $p_k$ computed in $\mathcal{G}$. So each $p_k^{-1}$ is still an $M_0$-invariant global type. For each $k\in[s]$, let $a_k$ realize $p_k^{-1}|_{(M_0,b_1,\dots,b_m)}$. By definition of the semigroup operation on $E$, we have $a_kb_i\models (p_k^{-1}q_i)|_{M_0}$ for each $i\in[m]$ and $k\in[s]$. In particular, $a_kb_i\models\theta(x)$ iff $p_k^{-1}q_i\vdash\theta(x)$ iff $p_k^{-1}q_i\in\langle\theta\rangle$ iff $q_i\in p_k\langle\theta\rangle$. Since the $p_k\langle\theta\rangle\cap B$ are pairwise distinct, it follows that the $a_k$ have pairwise distinct $\theta(xy)$-types over $\{b_1,\dots,b_m\}$, so that indeed $s\leqslant Cm^\delta\leqslant Cn^\delta$.
\end{proof}

Our main result now follows.

\begin{theorem}\label{FAR main theorem}
    Assume that $T$ and $M_0$ are countable. Suppose there is a fixed $\delta>0$ such that, for any $M_0$-definable subset $\theta(x)$ of $G$, the VC-codensities of the formula $\theta(xy)$ is at most $\delta$. Then the Ellis group $\mathcal{G}$ of $G(M_0)$ has finite Archimedean rank. More precisely, $\mathcal{G}$ is an inverse limit of compact Lie groups of dimension at most $(4\delta)^2$.
\end{theorem}
\begin{proof}
    This is an immediate consequence of all of our ingredients. Consider the subsets $\langle\theta\rangle\subseteq\mathcal{G}$, for $\theta(x)$ an $L(M_0)$-formula. By Fact \ref{borel definability}, each $\langle\theta\rangle$ is Borel. By Lemma \ref{theta is vc set in ellis group}, for any $\theta(x)$, the family of left translates of $\langle\theta\rangle$ has VC-density at most $\delta$. By Corollary \ref{tau-open is countable union}, every open subset of $\mathcal{G}$ is a union of sets of this form. Since $T$ and $M_0$ are countable, there are only countably many $\theta(x)$, and so we are done by Theorem \ref{general far theorem}.
\end{proof}

\begin{remark} By the respective results of \cite{five_authors} and \cite{basu_patel}, the VC-codensity hypothesis of Theorem \ref{FAR main theorem} is satisfied for o-minimal theories, the theories of the $p$-adic fields, and the theories of algebraically closed valued fields.
\end{remark}

% Recall that by the `Ellis group of $G(M_0)$', ie by $\mathcal{G}$, we mean the Ellis group of the tame flow $(G(M_0),S^{\mathrm{fs}}_G(\mathfrak{C},M_0))$. The point of Theorem \ref{FAR main theorem} is that, when $T$ and $M_0$ are countable and $T$ has bounded VC-codensity, Theorem \ref{general far theorem} naturally applies to this Ellis group. However we are using the specific details of the situation very heavily for this. As I mention in the AI disclosure, I am quite ignorant of topological dynamics, but motivated by \cite{basso_zucker} it seems to me it might be interesting to know if Theorem \ref{general far theorem} might apply naturally to the Ellis groups of some other tame flows that satisfy some natural topological dynamical restrictions. So I leave the following as a very naive and vague question:

% \begin{question}Are there other natural classes of tame flows to whose Ellis groups Theorem \ref{general far theorem} naturally applies?
% \end{question}

\section{`Local $G/G^{00}$'}\label{local g00 section}
We will now give some consequences of our techniques to `local $G/G^{00}$' in NIP theories, inspired by Question 7.18(1) in \cite{hrushovski}. First we need to give some relevant definitions; the best reference in the literature for our purposes is \cite{conant_pillay}, in which Conant and Pillay develop a theory of `local fsg' for pseudofinite groups. Unlike in the rest of the paper, \uline{in this section, do \textbf{not} assume $T$ is NIP}. We continue working in a monster model $\mathfrak{C}$ of $T$ and we continue assuming $G$ is a $\varnothing$-definable group; as usual we identify $G$ with $G(\mathfrak{C})$.

Let $\phi(x,y)$ be a formula that implies $x\in G$. We say $\phi(x,y)$ is `left-invariant' if every left translate of an instance of $\phi$ is equivalent to another instance of $\phi$; more precisely if, for every $b\in\mathfrak{C}^y$ and $g\in G(\mathfrak{C})$, there is $b'\in\mathfrak{C}^y$ with $\models\forall x[\phi(gx,b)\leftrightarrow\phi(x,b')]$. We define `right-invariant' and `bi-invariant' analogously. Given any formula $\phi(x)$ defining a subset of $G$, we can define the `left-stratification' $\phi^l(x,u)\equiv (u\in G)\wedge \phi(ux)$ and the `right-stratification' $\phi^r(x,u)\equiv (u\in G)\wedge\phi(xu)$, which are then respectively left-invariant and right-invariant. By slight abuse of notation we will often write the left and right stratifications of $\phi(x)$ as just $\phi(yx)$ and $\phi(xy)$, respectively.

Given a formula $\phi(x,y)$, by a `$\phi$-formula' we mean a Boolean combination of instances of $\phi(x,y)$, and as usual we let $S_\phi(\mathfrak{C})$ denote the collection of all complete global $\phi$-types, equipped with the Stone topology. If $\phi(x,y)$ is left-invariant and $H$ is a bounded-index subgroup of $G$ type-definable by $\phi$-formulas, then each left coset of $H$ is also type-definable by $\phi$-formulas, and every type in $S_\phi(\mathfrak{C})$ is concentrated on a (unique) left coset of $H$. So we get a canonical map from $S_\phi(\mathfrak{C})$ to the space $G/H$ of left cosets of $H$ in $G$. Moreover, since the relation of lying in the same left coset of $H$ is a bounded type-definable equivalence relation on $G$, we can equip the space $G/H$ with the (compact, Hausdorff) logic topology, where a subset of $G/H$ is closed if and only if its preimage under the canonical map $G\to G/H$ is type-definable. When $H$ is normal, so that $G/H$ is actually a group, then by Lemma 2.7 in \cite{pillays_conjecture} the logic topology makes $G/H$ into a topological group. Now, by Section 4 of \cite{conant_pillay} we have:

\begin{fact}\label{local logic topology}
    Let $\phi(x,y)$ be a left-invariant formula, and let $H$ be a bounded-index subgroup of $G$ type-definable by $\phi$-formulas. Equip the space of left cosets $G/H$ with the logic topology. Then the preimage of any closed subset of $G/H$ under the canonical map $G\to G/H$ is type-definable by $\phi$-formulas, and so the canonical map $S_\phi(\mathfrak{C})\to G/H$ is continuous.
\end{fact}

We additionally have the following, which is folklore, and for which we do not give details.

\begin{fact}\label{local g00 exists}
    Let $\phi(x,y)$ be a bi-invariant formula, and assume additionally that $\phi(x,y)$ is NIP. Then there is a smallest bounded-index subgroup of $G$ type-definable by $\phi$-formulas, which we denote $G^{00}_\phi$. $G^{00}_\phi$ is a normal subgroup of $G$, and the quotient $G/G^{00}_\phi$ equipped with the logic topology is a compact Hausdorff topological group. Moreover, $G^{00}_\phi$ is type-definable by a countable intersection of $\phi$-formulas.
\end{fact}
\begin{proof}
    That there is a smallest such subgroup is proved analogously as the `existence' of $G^{00}$ in NIP theories, proved in \cite{shelah_g00}, though see Proposition 6.1 in \cite{hrushovski_peterzil_pillay} for a much clearer writeup of that result. That $G^{00}_\phi$ is normal in $G$ follows from bi-invariance of $\phi$. As mentioned above, that $G/G^{00}_\phi$ is a topological group in the logic topology follows from Lemma 2.7 in \cite{pillays_conjecture}.

    To see that $G^{00}_\phi$ is definable by a countable intersection of $\phi$-formulas, note that $G^{00}_\phi$ as computed in $\mathfrak{C}$ is the same as it is computed in any reduct of $\mathfrak{C}$ to a subglanguage of $L$ in which $(G,\cdot)$ and $\phi(x,y)$ are definable, so without loss we may assume that $L$ is countable. Now, $G^{00}_\phi$ is type-definable, and it is clearly invariant under all automorphisms of $\mathfrak{C}$, so by a standard argument (eg Exercise 6.1.10 in \cite{tent_ziegler}) $G^{00}_\phi$ is type-definable over $\varnothing$. Thus $G^{00}_\phi$ can be expressed as an intersection of formulas without parameters, of which there are only countably many, and $G^{00}_\phi$ can be expressed as an intersection of $\phi$-formulas. By a standard compactness argument it follows that $G^{00}_\phi$ can be expressed as a countable intersection of $\phi$-formulas.
\end{proof} Now inspired by Section 7 of \cite{hrushovski} one can ask the following:

\begin{question}\label{local g/g00 question}
    Suppose $\phi(x,y)$ is bi-invariant and NIP. Then does $G/G^{00}_\phi$ have finite Archimedean rank, depending only on the VC-codensity of $\phi$?
\end{question} This question is related to Theorem 7.15 and to Question 7.18(1) in \cite{hrushovski}, though there are some important differences. The situation is a bit subtle and we will take a minute to explain the differences.

First let us describe the setting. In Section 7.14 of \cite{hrushovski}, Hrushovski works with an NIP formula $\phi(x,y)$ such that $G$ is definably amenable with respect to $\phi$-formulas; in other words, he assumes the existence of a definable translation-invariant Keisler measure $\mu$ on the Boolean algebra of definable subsets of $G$ generated by the left translates of instances of $\phi$. In fact, it is important to note that, by Example 3.7 in \cite{conant_pillay}, the left-stratification of even a right-invariant NIP formula need not be NIP, and because of this behavior it seems to me that in Theorem 7.15 of \cite{hrushovski} it is perhaps necessary, or at least very convenient, to assume that $\phi(x,y)$ is bi-invariant. So let us assume that $\phi(x,y)$ is bi-invariant, although this is not an assumption made in \cite{hrushovski}. Hrushovski defines an object he calls $G^{00}_\phi$, but which to avoid conflict of notation (and since it depends on $\mu$) I will call $G^{00}_{\phi,\mu}$, which is the set of all $g$ such that $\mu(\phi(gx,b)\triangle\phi(x,b))=0$ for all $b$. By definability of $\mu$, $G^{00}_{\phi,\mu}$ is type-definable, and, by left-invariance of $\phi$, $G^{00}_{\phi,\mu}$ is normal. One can check that $G^{00}_{\phi,\mu}$ has bounded-index.

Hrushovski's aim is to prove that $G/G^{00}_{\phi,\mu}$, equipped with the logic topology, has finite Archimedean rank, but he is not able to achieve this goal, and instead only gets a partial result (Theorem 7.15 in \cite{hrushovski}), which shows that, for any complete Shelah strong type $q$, the quotient $G/G^{00}_{\phi,\mu,q}$ has finite Archimedean rank, where $G^{00}_{\phi,\mu,q}$ is defined as $G^{00}_{\phi,\mu}$ but only looking at formulas $\phi(x,b)$ for $b\models q$. While working on the material of Section \ref{FAR section}, I was initially getting stuck needing to make a similar restriction as Hrushovski's restriction to a single Shelah strong type, and I eventually found the way around this by working with the Haar measure on the connected component of the Ellis group rather than with the Haar measure on the Ellis group itself. As such I began writing this present section with the hope of using my results or techniques to answer Hrushovski's question, ie to prove that $G/G^{00}_{\phi,\mu}$ always has finite Archimedean rank. However in working on this I have realized that my methods do not really (or at least do not obviously) seem to lead to a path to doing this. The main issue is that the group that I am calling $G^{00}_{\phi}$ (in line with the terminology from \cite{conant_pillay}) may be strictly larger than the group that Hrushovski calls $G^{00}_\phi$, ie the group $G^{00}_{\phi,\mu}$; indeed it is not difficult to see that $G^{00}_{\phi,\mu}\subseteq G^{00}_{\phi}$, but there is no reason that we should have $G^{00}_{\phi}\subseteq G^{00}_{\phi,\mu}$, ie no reason that $G^{00}_{\phi,\mu}$ need be type-definable by $\phi$-formulas. 

If $\mu$ is generically stable, then, again recalling the assumption that $\phi(x,y)$ is bi-invariant, $G^{00}_{\phi,\mu}$ does indeed coincide with $G^{00}_\phi$, by virtue of Conant and Pillay's results in \cite{conant_pillay} (see Remark 1.3 in \cite{conant_pillay}).\footnote[1]{On the other hand, even in the stable setting this can fail for almost trivial reasons if $\phi(x,y)$ is only left-invariant and not bi-invariant, even though the definition of $G^{00}_{\phi,\mu}$ still makes sense here. For instance, if $H$ is a $\varnothing$-definable finite-index subgroup of $G$ and $\phi(x,y)$ is the formula $x\in yH$, then the $\phi$-definable subsets of $G$ are just the finite unions of cosets of $H$, and so $G^{00}_\phi=H$, but if $\mu$ is the lift of the normalized counting measure on $G/H$, then $G^{00}_{\phi,\mu}$ will be the intersection of all conjugates of $H$. So if $H$ is not normal then already the two disagree. Thanks to Gabe Conant for pointing this out to me.} However, in general, even with the assumption that $\phi(x,y)$ is bi-invariant, the two do not coincide. I am very grateful to Devrim Pekmezci for pointing this out to me. Pekmezci is working towards proving a local analogue of Fact \ref{newelski conjecture for definably amenable nip groups}, and among numerous positive results he also has an example showing that the naive analogue fails; his counterexample also provides an example of a bi-invariant NIP formula $\phi(x,y)$ and a measure $\mu$ where $G^{00}_{\phi,\mu}\subsetneq G^{00}_\phi$. We will not describe the details here and instead refer to his upcoming papers.

So Question \ref{local g/g00 question} is in that sense weaker than Hrushovski's question of whether the $G/G^{00}_{\phi,\mu}$ are of finite Archimedean rank. However, Question \ref{local g/g00 question} is also more general than Hrushovski's question, since it does not make any definable amenability assumption, whereas Hrushovski's question requires a definable amenability assumption even to be formulated, since the definition of $G^{00}_{\phi,\mu}$ depends on $\mu$. 

Those are the principal differences between Hrushovski's question and our Question \ref{local g/g00 question}. From now on we will just focus on our question and not on Hrushovski's. We will show that, under a global NIP assumption, we can obtain a positive answer to Question \ref{local g/g00 question} as a consequence of Lemma \ref{general local far result} along with Fact \ref{local logic topology}. Let us work towards proving this. 

We first need a slight refinement of Fact \ref{local logic topology}, which is standard. (It is a small variation of the proof that Kim-Pillay types coincide with types over models.) \begin{lemma}\label{local logic topology refinement}Let $\phi(x,y)$ be a left-invariant formula and let $H$ be a bounded-index subgroup of $G$ type-definable by $\phi$-formulas over some parameter set $B$. Then there is $B'\supseteq B$ with $|B'|=|B|+\aleph_0$ and such that (1) every complete $\phi$-type over $B'$ determines a left coset of $H$, and (2) the induced map $S_\phi(B')\to G/H$ to the space of left cosets is continuous.
\end{lemma}
\begin{proof}
    Let $\Sigma(x)$ be a collection of $\phi$-formulas with parameters from $B$ closed under conjunction and type-defining $H$.

    Note that, by Fact \ref{local logic topology}, the logic topology on $G/H$ remains the same in any reduct to a sublanguage in which $(G,\cdot)$ and $\phi$ are still definable; thus we may assume without loss that $L$ is countable. Now by Lowenheim-Skolem there is $M\prec\mathfrak{C}$ with $B\subseteq M$ and $|M|=|B|+\aleph_0$. We will show that any complete $\phi$-type over $M$ determines a coset of $H$.

    %Let $\mathfrak{C}^*$ be the reduct of $\mathfrak{C}$ to a countable sublanguage $L^*\subseteq L$ such that $(G,\cdot)$ and $\phi(x,y)$ are still definable in $L^*$. From hereonout we will always work in the reduct. By Lowenheim-Skolem, we may find an $L^*$-elementary substructure $M^*\prec \mathfrak{C}^*$ such that $A\subseteq M^*$ and $|M^*|=|A|+\aleph_0$. We will show that any complete $\phi$-type over $M^*$ determines a coset of $H$.

    Suppose otherwise that $g,h$ have the same $\phi$-type over $M$ but that they lie in different cosets of $H$. Let $p\in S_G(\mathfrak{C})$ be any complete global type finitely satisfiable in $M$ and extending the (shared) $\phi$-type of $g$ and $h$ over $M$. Let $(a_i:i\in\kappa)$ be a long Morley sequence of $p$ over $(M,g,h)$. We will show that the $a_i$ lie in distinct cosets of $H$, contradicting that $H$ has bounded-index.

    Suppose instead that $a_k\in a_iH$ for some $i<k$. Since $g,h$ lie in distinct cosets of $H$, at least one of $a_k\notin gH$ and $a_k\notin hH$ holds; suppose without loss that $a_k\notin gH$. Now we have $a_k\models\Sigma(a_i^{-1}x)$ and $a_k\not\models\Sigma(g^{-1}x)$. On the other hand, since $H$ is closed under inverses, $\Sigma(x)\vdash\Sigma(x^{-1})$, so also $a_k\models\Sigma(x^{-1}a_i)$ and $a_k\not\models\Sigma(x^{-1}g)$. So there is a formula $\psi(x,c)\in\Sigma$ such that $a_k\models\psi(x^{-1}a_i,c)\wedge\neg\psi(x^{-1}g,c)$. Since $p$ is finitely satisfiable in $M$, we have $a_k\ind^u_{M}(g,a_i)$, and thus there is some $m\in G(M)$ with $m\models\psi(x^{-1}a_i,c)\wedge\neg\psi(x^{-1}g,c)$.  So $a_i\models\psi(m^{-1}x,c)$ and $g\models\neg\psi(m^{-1}x,c)$. But now $\psi(x,c)$ is a $\phi$-formula with parameters from $B$, hence in particular with parameters from $M$, and $\psi(m^{-1}x,c)=m\psi(x,c)$ is a left translate of $\psi(x,c)$ by an element of $G(M)$, hence (since $\phi$ is left-invariant) itself a $\phi$-formula with parameters from $M$. But now $a_i$ and $g$ have different $\phi$-types over $M$, contradicting that $p$ extends the $\phi$-type of $g$ over $M$.

    To see continuity, recall that, by Fact \ref{local logic topology}, the natural map $S_\phi(\mathfrak{C})\to G/H$ is continuous, and it is the composition of the restriction map $S_\phi(\mathfrak{C})\to S_\phi(M)$ with the natural map $S_\phi(M)\to G/H$. But the restriction map $S_\phi(\mathfrak{C})\to S_\phi(M)$ is a continuous map between compact Hausdorff spaces, and is hence a closed map, so the claim follows.
\end{proof}

The most naive approach to Question \ref{local g/g00 question} would be to try to apply Theorem \ref{general far theorem} directly to $G/G^{00}_\phi$. Indeed suppose $\phi(x,y)$ is bi-invariant and NIP. By Lemma \ref{local logic topology refinement}, we have a countable parameter set $B$ and a natural continuous surjection $\pi:S_\phi(B)\to G/G^{00}_\phi$. Now, for any $\phi$-formula $\theta(x)$ with parameters in $B$, the pushforward of the clopen subset $[\theta(x)]\subseteq S_\phi(B)$ under $\pi$ gives a closed subset $C_\theta\subseteq G/G^{00}_\phi$, and any open subset of $G/G^{00}_\phi$ will be a union of sets of form $C_\theta$. Since $B$ is countable, there are only countably many $C_\theta$, and these are the obvious candidate sets to try to apply Theorem \ref{general far theorem} to. However, it is not at all clear if the family of left translates of $C_\theta$ has bounded VC-density.\begin{question}\label{vc-density question}
    Suppose $\phi(x,y)$ is a bi-invariant NIP formula. Given a $\phi$-formula $\theta(x)$, let $C_\theta$ be the closed subset of $G/G^{00}_\phi$ described above. Does $(gC_\theta:g\in G/G^{00}_\phi)$ have bounded VC-density, depending only on the VC-codensity of $\phi(x,y)$?
\end{question} Question \ref{vc-density question} is related in spirit to an open question of Chernikov, Pillay, and Simon from \cite{chernikov_pillay_simon}, of whether naming $G^{00}$ by a predicate in an NIP theory preserves NIP. More recently, Chernikov gave a `local' strengthening of this question in Question 3.40 of \cite{chernikov}, asking whether $G^{00}_\phi$ is externally definable, and this is extremely closely related to Question \ref{vc-density question}. It is not at all clear to me how to address any of those questions. Instead we will argue via Ellis groups, where the VC-theoretic behavior is clear, by Lemma \ref{theta is vc set in ellis group}.

First we need an easy observation about VC-codensity; compare with Proposition 3.12(a) in \cite{conant_pillay} and with Lemma 2.8 in \cite{five_authors}.

\begin{lemma}\label{vc codensity lemma}
    Suppose $\phi(x,y)$ is a right-invariant formula of VC-codensity at most $\delta$. Then, for any $\phi$-formula $\theta(x)$ implying $x\in G$, the right-stratification $\theta(xy)$ has VC-codensity at most $\delta$.
\end{lemma}
\begin{proof} Let $k$ be such that, for any finite $B\subset \mathfrak{C}^y$, we have $|S_\phi(B)|\leqslant k|B|^{\delta}$. Since $\theta(x)$ is a $\phi$-formula, let $\phi(x,b_1),\dots,\phi(x,b_s)$ be a fixed collection of instances of $\phi$ such that $\theta(x)$ is a Boolean combination of them. For any $g\in G(\mathfrak{C})$, since $\phi$ is right-invariant, there are $b_1^g,\dots,b_s^g$ such that each $\phi(xg,b_i)$ is equivalent to $\phi(x,b_i^g)$. In particular, then $\theta(xg)$ is equivalent to a Boolean combination of the $\phi(x,b_1^g),\dots,\phi(x,b_s^g)$. Thus, for any $C\subseteq G(\mathfrak{C})$, a $\theta(xy)$-type over $C$ is determined by the $\phi$-type over $\{b_i^g:i\in [s],g\in C\}$ that it induces. This latter set has size at most $s|C|$, and so the number of $\theta(xy)$-types over $C$ is at most $k(s|C|)^{\delta}=ks^{\delta}|C|^{\delta}$, so we are done.
\end{proof}

Now, from hereonout, \uline{assume $\phi(x,y)$ is a bi-invariant NIP formula}. So by Fact \ref{local g00 exists} $G^{00}_\phi$ `exists' and is type-definable by countably many $\phi$-formulas. By Lemma \ref{local logic topology refinement}, \uline{fix $B$ a countable parameter set such that (1) every coset of $G^{00}_\phi$ is type-definable by $\phi$-formulas with parameters in $B$, and (2) the induced map $S_\phi(B)\to G/G^{00}_\phi$ is continuous}. Finally, by adding constant symbols, \uline{assume that $\dcl(\varnothing)$ is a model $M_0$ containing $B$}. (Note that, even though $B$ is countable, $T$ may not be countable, so $M_0$ may not be countable.) Finally, \uline{let $E$ be the Ellis semigroup $S^{\mathrm{fs}}_G(\mathfrak{C},M_0)$, let $\mathcal{I}$ be a minimal ideal of $E$, let $u\in\mathcal{I}$ be an idempotent, and let $\mathcal{G}=u\mathcal{I}$ equipped with the $\tau$-topology}.

We are not making a global NIP assumption, so $\mathcal{G}$ need not be Hausdorff. However, by Proposition 3.1 and Theorem 0.1(1) in \cite{krupinski_pillay_early}, we do have:

\begin{fact}\label{tau topology to logic topology preliminary}
    The natural map $\mathcal{G}\to G/G^{00}_{M_0}$ (obtained by restricting the natural map $S_G(\mathfrak{C})\to G/G^{00}_{M_0}$ to the subset $\mathcal{G}\subseteq E\subseteq S_G(\mathfrak{C})$) is a continuous group epimorphism, where $\mathcal{G}$ has the $\tau$-topology and $G/G^{00}_{M_0}$ has the logic topology.
\end{fact}

On the other hand, $G^{00}_\phi$ is bounded-index and type-definable over $M_0$, so $G^{00}_{M_0}\subseteq G^{00}_\phi$ and $G^{00}_\phi/G^{00}_{M_0}$ is a closed normal subgroup of $G/G^{00}_{M_0}$, and the induced quotient map $G/G^{00}_{M_0}\to G/G^{00}_\phi$ is a continuous group epimorphism between the logic topologies. So altogether, by composing the two maps:

\begin{fact}\label{tau topology to logic topology}
    The natural map $\pi:\mathcal{G}\to G/G^{00}_\phi$ is a continuous group epimorphism, where $\mathcal{G}$ has the $\tau$-topology and $G/G^{00}_\phi$ has the logic topology.
\end{fact}

On the other hand, we can factor $\pi$ as the composition $\mathcal{G}\to S_\phi(B)\to G/G^{00}_\phi$, which gives us the following. Recall the notational convention from Section \ref{FAR section} that, for a formula $\theta(x)$, we use $[\theta]$ to denote the clopen subset of $E=S^{\mathrm{fs}}_G(\mathfrak{C},M_0)$ corresponding to $\theta(x)$, and use $\langle\theta\rangle$ to denote the subset $[\theta]\cap\mathcal{G}$ of $\mathcal{G}$.

\begin{lemma}\label{preimage of logic-open is countable union}
    Let $\pi:\mathcal{G}\to G/G^{00}_\phi$ be the natural map. For any open subset $U\subseteq G/G^{00}_\phi$, there are $\phi$-formulas $\theta_i(x)$ with parameters in $B$ such that $\pi^{-1}(U)=\bigcup_{i}\langle\theta_i(x)\rangle$. Since $B$ is countable, this is a countable union.
\end{lemma}
\begin{proof}
    Recalling that $\mathcal{G}\subseteq S_G^{\mathrm{fs}}(\mathfrak{C},M_0)$ is a set of complete global types, let $\rho:\mathcal{G}\to S_\phi(B)$ be the natural restriction map. Given a $\phi$-formula $\theta(x)$ with parameters in $B$, by abuse of notation let $[\theta(x)]$ denote the associated clopen subset of both $S^{\mathrm{fs}}(\mathfrak{C},M_0)$ and of $S_\phi(B)$. Then clearly $\rho^{-1}([\theta(x)])=[\theta(x)]\cap\mathcal{G}=\langle\theta(x)\rangle$. Note that $\rho$ is \textbf{not} a continuous map, since $\mathcal{G}$ is equipped with the $\tau$-topology, not the Stone subspace topology
    
    Recall that we chose $B$ so that every element of $S_\phi(B)$ determines a coset of $G/G^{00}_\phi$ and so that the associated map $\iota:S_\phi(B)\to G/G^{00}_\phi$ is continuous. Thus $\iota^{-1}(U)$ is an open subset of $S_\phi(B)$, and hence of form $\bigcup_i[\theta_i(x)]$ for $\theta_i$ some $\phi$-formulas with parameters in $B$. But we have $\pi=\iota\circ\rho$, and so $\pi^{-1}(U)=\rho^{-1}(\iota^{-1}(U))=\rho^{-1}(\bigcup_{i}[\theta_i(x)])=\bigcup_i\langle\theta_i(x)\rangle$.
\end{proof}

Combining the previous two facts gives us the following main consequence.

\begin{lemma}\label{nss consequence for g00}
    Suppose that the $\tau$-topology on $\mathcal{G}$ is Hausdorff and that the sets $\langle\theta\rangle\subseteq\mathcal{G}$ are Borel. Let $\pi:\mathcal{G}\to G/G^{00}_\phi$ be the natural map. Let $N$ be a closed normal subgroup of $(G/G^{00}_\phi)^0$ such that $(G/G^{00}_\phi)^0/N$ is a Lie group. Then there is a $\phi$-formula $\theta(x)$ with parameters in $B$ such that, if $d_\theta$ denotes the pseudometric on $\mathcal{G}^0$ associated to the Borel subset $\langle\theta\rangle\cap\mathcal{G}^0$, and $N_\theta$ is the intersection of all of the conjugates of $\ker(d_\theta)$ in $\mathcal{G}^0$, then $\pi$ induces a continuous epimorphism $\mathcal{G}^0/N_\theta\to (G/G^{00}_\phi)^0/N$.
\end{lemma}
\begin{proof} By Fact \ref{tau topology to logic topology}, $\pi$ is a continuous epimorphism, so by Lemma \ref{connected component of quotient} the restriction of $\pi$ to $\mathcal{G}^0$ yields a continuous epimorphism $\mathcal{G}^0\to (G/G^{00}_\phi)^0$, and thus we just need to show that there is some $\theta$ with $\pi(N_\theta)\leqslant N$. We will argue similarly as in Theorem \ref{general far theorem}. As in Theorem \ref{general far theorem}, by NSS of Lie groups we may find $U\subseteq (G/G^{00}_\phi)^0$ an open neighborhood of the identity such that any subgroup of $(G/G^{00}_\phi)^0$ contained in $UU^{-1}$ is contained in $N$. So we just need to find some $\theta$ with $\pi(N_\theta)\subseteq UU^{-1}$. Let $U'\subseteq G/G^{00}_\phi$ be open such that $U=U'\cap (G/G^{00}_\phi)^0$.

By Lemma \ref{preimage of logic-open is countable union}, there are $\phi$-formulas $\theta_i(x):i\in\omega$ with parameters in $B$ such that $\pi^{-1}(U')=\bigcup_{i\in\omega}\langle\theta_i\rangle$. Note that, since $\pi(\mathcal{G}^0)=(G/G^{00})_\phi^0$ and $U=U'\cap (G/G^{00}_\phi)^0$, we have $\pi^{-1}(U)=\pi^{-1}(U')\cap\mathcal{G}^0$ and thus $\pi^{-1}(U)=\bigcup_{i\in\omega}\langle\theta_i\rangle\cap\mathcal{G}^0$. By continuity of $\pi$, $\pi^{-1}(U)$ is a (non-empty) open subset of $\mathcal{G}^0$, so if $\eta_0$ denotes the normalized Haar measure on $\mathcal{G}^0$ then $\eta_0(\pi^{-1}(U))>0$. Since the $\langle\theta_i\rangle\cap\mathcal{G}^0$ are all Borel, by countable additivity there is thus some $i\in\omega$ with $\eta_0(\langle\theta_i\rangle\cap\mathcal{G}^0)>0$. 

As in Theorem \ref{general far theorem}, denote $\langle\theta_i\rangle_0=\langle\theta_i\rangle\cap\mathcal{G}^0$. So $\eta_0(\langle\theta_i\rangle_0)>0$. Arguing as in the proof of Theorem \ref{general far theorem}, it follows that $\ker(d_{\theta_i})\subseteq \langle\theta_i\rangle_0\langle\theta_i\rangle_0^{-1}$, and hence in particular $N_{\theta_i}\subseteq\langle\theta_i\rangle_0\langle\theta_i\rangle_0^{-1}$. But $\langle\theta_i\rangle_0\subseteq\pi^{-1}(U)$, so $N_{\theta_i}\subseteq\pi^{-1}(U)\pi^{-1}(U)^{-1}$, whence $\pi(N_{\theta_i})\subseteq UU^{-1}$ and we are done.
\end{proof}

On the other hand, by Lemma \ref{general local far result} we have the following.
\begin{lemma}\label{local ellis group lemma}
    Suppose that the $\tau$-topology on $\mathcal{G}$ is Hausdorff. Suppose also that $\theta(x)$ is a formula implying $x\in G$ and such that the right-stratification $\theta(xy)$ has VC-codensity at most $\delta$, and suppose that the subset $\langle\theta(x)\rangle$ of $\mathcal{G}$ is Borel. If $d_\theta$ is the pseudometric on $\mathcal{G}^0$ corresponding to the Borel subset $\langle\theta\rangle\cap\mathcal{G}^0$, and $N_\theta$ is the intersection of all of the conjugates of $\ker(d_\theta)$ by elements of $\mathcal{G}^0$, then $\mathcal{G}^0/N_\theta$ is an inverse limit of compact Lie groups of dimension at most $(4\delta)^2$.
\end{lemma}
\begin{proof}
    By the proof of Lemma \ref{theta is vc set in ellis group}, which does not actually use the global NIP assumption at all, the family of translates of $\{g\langle\theta\rangle:g\in\mathcal{G}\}$ has VC-density at most $\delta$. So, letting $\langle\theta\rangle_0=\langle\theta\rangle\cap\mathcal{G}^0$, the same applies to the family $\{g\langle\theta\rangle_0:g\in\mathcal{G}^0\}$. Now the claim follows by applying Lemma \ref{general local far result} to the Borel subset $\langle\theta\rangle_0$ of $\mathcal{G}^0$.
\end{proof}

The previous two lemmas give the main result.

\begin{theorem}\label{local g/g00 main theorem}
    Suppose $T$ is NIP. Then $G/G^{00}_\phi$ is of finite Archimedean rank. More precisely, if $\phi(x,y)$ has VC-codensity at most $\delta$, then $G/G^{00}_\phi$ is an inverse limit of compact Lie groups of dimension at most $(4\delta)^2$.
\end{theorem}
\begin{proof}By Lemma \ref{archimedean rank of connected component}, it suffices to show that $(G/G^{00}_\phi)^0$ is an inverse limit of compact Lie groups of dimension at most $(4\delta)^2$. By Peter-Weyl, $(G/G^{00}_\phi)^0$ is an inverse limit of compact Lie groups, so it suffices to show that, if $N$ is any closed normal subgroup of $(G/G^{00}_\phi)^0$ such that $(G/G^{00}_\phi)^0/N$ is a Lie group, then $(G/G^{00}_\phi)^0/N$ has dimension at most $(4\delta)^2$. So fix such an $N$.

Let $\pi:\mathcal{G}\to G/G^{00}_\phi$ be the natural map. Since $T$ is NIP, by Fact \ref{tau-topology is hausdorff} and Fact \ref{borel definability} the $\tau$-topology on $\mathcal{G}$ is Hausdorff and every set $\langle\theta\rangle$ is Borel. So by Lemma \ref{nss consequence for g00} there is a $\phi$-formula $\theta(x)$ such that, letting $d_\theta$ be the pseudometric on $\mathcal{G}^0$ associated to the Borel subset $\langle\theta\rangle\cap\mathcal{G}^0$, if $N_\theta$ is the intersection of all of the conjugates of $\ker(d_\theta)$ in $\mathcal{G}^0$ then $\pi$ induces a continuous epimorphism $\mathcal{G}^0/N_\theta\to(G/G^{00}_\phi)^0/N$. By Lemma \ref{vc codensity lemma}, the VC-codensity of the right-stratification $\theta(xy)$ is at most $\delta$, and so by Lemma \ref{local ellis group lemma} $\mathcal{G}^0/N_\theta$ is an inverse limit of compact Lie groups of dimension at most $(4\delta)^2$, and so by Lemma \ref{lie quotient of far group} we are done.
\end{proof}

Note that, by Fact \ref{tau topology to logic topology preliminary} and by Theorem \ref{FAR main theorem}, if $T$ has bounded VC-codensity then $G/G^{00}$ itself has finite Archimedean rank. The point of Theorem \ref{local g/g00 main theorem} is to get a `local' version. Due to this it is a bit unsatisfying that we make a global NIP assumption, and it would be much better to remove the global NIP assumption and hence answer Question \ref{local g/g00 question} in full generality. A likely approach to do this comes from the already-mentioned recent work of Devrim Pekmezci. Pekmezci has shown in \cite{devrim} that, in an arbitrary theory, one can define the `local Ellis group' associated to a bi-invariant formula $\phi(x,y)$, and that if $\phi(x,y)$ is NIP then the $\tau$-topology on the local Ellis group is Hausdorff. It seems promising to obtain a positive answer to Question \ref{local g/g00 question} using these local Ellis groups. Indeed the proof of Theorem \ref{local g/g00 main theorem} use the global NIP assumption only to have access to Fact \ref{tau-topology is hausdorff} and Fact \ref{borel definability}, of which \cite{devrim} give the local analogues. So the proof from this section should work just as well in the setting of the local Ellis groups. On the other hand, a stronger result one might hope for, in analogue with Theorem \ref{FAR main theorem}, is the following: \begin{question}\label{local ellis group question}
    Does the local Ellis group associated to a bi-invariant NIP formula $\phi(x,y)$ have finite Archimedean rank, depending only on the VC-codensity of $\phi$?
\end{question}Adapting the methods from the proof of Theorem \ref{FAR main theorem} to the local setting seems very challenging to me, since without the global NIP assumption we do not have access to Theorem \ref{restriction map ellis group theorem} and hence do not have access to the essential Corollary \ref{tau-open is countable union}. It is completely unclear to me if there is hope of modifying these results to the local setting, since the development of the machinery of wideness (which is crucial for Theorem \ref{restriction map ellis group theorem}) uses the global NIP assumption very heavily.

\appendix
\section{Haar measure in the quotient group}
In this section we record some basic facts about certain pseudometrics associated to Borel subsets of a compact Hausdorff group. This material is surely standard, and it is closely related to a technique that already appears implicitly in Hrushovski's proof of Theorem 7.12 in \cite{hrushovski} (though without details), but we could not find a reference, so we include it in this appendix. Let $G$ be a compact Hausdorff group, and let $\eta$ be the (unique, bi-invariant) normalized Haar measure on $G$. We define the following:

\begin{definition}\label{pseudometric definition}Given a Borel subset $B\subseteq G$, define a map $d_B:G\times G\to [0,1]$ by taking $d_B(g,h)=\eta(gB\triangle hB)$ for all $g,h\in G$.
\end{definition}

Now we have the following, eg by Theorem A in Chapter XII of \cite{halmos}:
\begin{fact}
    For any Borel subset $B\subseteq G$, the map $g\mapsto d_B(g,1)$ is continuous.
\end{fact}

So we get an easy consequence:
\begin{lemma}\label{pseudometric and closed subgroup lemma}For every Borel subset $B\subseteq G$, $d_B$ defines a left-invariant pseudometric on $G$, and the set $\ker(d_B):=\{g:d_B(g,1)=0\}$ is a closed subgroup of $G$.
\end{lemma}
\begin{proof}
     Left-invariance follows from left-invariance of the Haar measure, and as usual the triangle inequality follows from left-invariance of the Haar measure and the identity $B\triangle (gh)B\subseteq (B\triangle gB)\cup (gB\triangle ghB)=(B\triangle gB)\cup g(B\triangle hB).$
\end{proof}

We will need to use the interaction of these pseudometrics with the quotient groups of $G$. Let $N$ be a closed normal subgroup of $G$, let $\mu$ be the normalized Haar measure on $N$, and let $\nu$ be the normalized Haar measure on $G/N$. Since $G$, $N$, and $G/N$ are compact groups, these Haar measures are bi-invariant, and all bounded Borel functions on the respective groups are integrable. Also, (eg by Theorem C in Chapter XII of \cite{halmos}), if $\pi:G\to G/N$ denotes the quotient map then $\nu(C)=\eta(\pi^{-1}(C))$ for all Borel subsets $C\subseteq G/N$. Given a bounded Borel function $f:G\to\mathbb{C}$, we define $T(f)$ a bounded Borel function on $G/N$ by taking $T(f)(Ng)=\int_Nf(ng)d\mu(n)$. Then we have the `quotient integral formula' (eg Theorem 2.49 in \cite{folland}): for any $f\in L^1(G)$, $\int_G fd\eta=\int_{G/N}T(f)d\nu$. %Also, since compact groups are unimodular, the Haar measure is bi-invariant, and thus .

\begin{lemma}
    Suppose that $B\subseteq G$ is Borel and that $N$ is contained in $\ker(d_B)$, ie that $\eta(nB\triangle B)=0$ for all $n\in N$. Let $\mathbf{1}_B:G\to[0,1]$ be the indicator function of $B$. Then $T(\mathbf{1}_B)(Ng)=\mathbf{1}_B(g)$ for almost all $g\in G$.
\end{lemma}
\begin{proof}
    Define $f:G\to [0,1]$ by $f(g)=T(\mathbf{1}_B)(Ng)=\int_N\mathbf{1}_B(ng)d\mu(n)$; we wish to show that $f(g)=\mathbf{1}_B(g)$ for almost all $g\in G$. It suffices to show that $\int_G|f-\mathbf{1}_B|d\eta=0$. Since $\mu(N)=1$, we trivially have $\mathbf{1}_B(g)=\int_N\mathbf{1}_B(g)d\mu(n)$. So now:
    \begin{align*}\int_G|f(g)-\mathbf{1}_B(g)|d\eta(g) &=\int_G\left|\int_N\mathbf{1}_B(ng)d\mu(n)-\int_N\mathbf{1}_B(g)d\mu(n)\right|d\eta(g) \\
    &\leqslant\int_G\left[\int_N|\mathbf{1}_B(ng)-\mathbf{1}_B(g)|d\mu(n)\right]d\eta(g)\\
    &=\int_N\left[\int_G|\mathbf{1}_B(ng)-\mathbf{1}_B(g)|d\eta(g)\right]d\mu(n) \\
    &=\int_N\left[\int_G|\mathbf{1}_{n^{-1}B}(g)-\mathbf{1}_B(g)|d\eta(g)\right]d\mu(n)\\
    &=\int_N\eta(n^{-1}B\triangle B)d\mu(n)=0,
    \end{align*}where the third line is Fubini-Tonelli, the second-to-last last equality just uses the standard identity $|\mathbf{1}_{n^{-1}B}-\mathbf{1}_B|=\mathbf{1}_{n^{-1}B\triangle B}$, and the last equality uses $N\subseteq\ker(d_B)$.
\end{proof} In particular, $T(\mathbf{1}_B)(Ng)$ is equal to $0$ or $1$ for almost all $g\in G$, hence for almost all $Ng\in G/N$. The main result we will need is the following:

\begin{lemma}\label{haar measure descends}
    Suppose that $B\subseteq G$ is Borel and that $N$ is contained in $\ker(d_B)$. Then there is a Borel subset $C\subseteq G/N$ such that $d_C(gN,hN)=d_B(g,h)$ for all $g,h\in G$.
\end{lemma}
\begin{proof}
    By the previous lemma, $T(\mathbf{1}_B)(Ng)=\mathbf{1}_B(g)$ for almost all $g\in G$; as remarked above thus in particular $T(\mathbf{1}_B)(Ng)$ is $0$ or $1$ for almost all $Ng\in G/N$. Let $$C=T(\mathbf{1}_B)^{-1}(1)=\{Ng:T(\mathbf{1}_B)(Ng)=1\}.$$ Since $T(\mathbf{1}_B)(Ng)$ is $0$ or $1$ almost everywhere, we have $T(\mathbf{1}_B)(Ng)=\mathbf{1}_C(Ng)$ for almost all $Ng\in G/N$, and hence $\mathbf{1}_B(g)=\mathbf{1}_C(Ng)$ for almost all $g\in G$. The desired result now follows easily; indeed, for $g,h\in G$, we have
\begin{align*}
        d_B(g,h)&=\eta(gB\triangle hB) \\
        &=\int_G|\mathbf{1}_{gB}(x)-\mathbf{1}_{hB}(x)|d\eta(x) \\
        &= \int_G|\mathbf{1}_B(g^{-1}x)-\mathbf{1}_B(h^{-1}x)|d\eta(x)\\
        &=\int_{G}|\mathbf{1}_C(Ng^{-1}x)-\mathbf{1}_C(Nh^{-1}x)|d\eta(x) \\
        &=\int_{G}|\mathbf{1}_{(Ng)C}(Nx)-\mathbf{1}_{(Nh)C}(Nx)|d\eta(x), \\
        &=\int_{G/N}|\mathbf{1}_{(Ng)C}(Nx)-\mathbf{1}_{(Nh)C}(Nx)|d\nu(Nx) \\
        &=\nu((Ng)C\triangle (Nh)C)=d_C(Ng,Nh),
    \end{align*} where the fourth equality uses that $\mathbf{1}_B(g)=\mathbf{1}_C(Ng)$ for almost all $g\in G$ and the third-to-last equality is by the quotient integral formula.
\end{proof}

\newpage

%\bibliographystyle{plain}
%\bibliography{references}

\begin{thebibliography}{9}
%\bibitem{adler_nip}Hans Adler. Introduction to theories without the independence property. \textit{Archive of Mathematical Logic}, 2008.

\bibitem{adler_forking}Hans Adler. A geometric introduction to forking and thorn-forking. \textit{Journal of Mathematical Logic}, Vol. 9, No. 1 (2009), pp. 1-20.

\bibitem{adler_ntp2}Hans Adler. Kim's lemma for NTP$_2$ theories. Preprint (2014).

%\bibitem{bays_kaplan_simon}Martin Bays, Itay Kaplan, Pierre Simon. Density of compressible types and some consequences. \textit{Journal of the EMS}, 2022.

\bibitem{five_authors} Matthias Aschenbrenner, Alf Dolich, Deirdre Haskell, Dugald Macpherson, Sergei Starchenko. Vapnik-Chervonenkis density in some theories without the independence property, I. \textit{Transactions of the AMS}. Vol. 368, No. 8 (2016). pp 5889–5949.

\bibitem{auslander_book}Joseph Auslander. \textit{Minimal Flows and Their Extensions}. North Holland (1988).

\bibitem{basso_zucker}Gianluca Basso and Andy Zucker. Topological groups with tractable minimal dynamics. Preprint. arXiv:2412.05659.

\bibitem{basu_patel}Saugata Basu and Deepam Patel. VC density of definable families over valued fields. \textit{Journal of the EMS} 23 (2021), pp 2361–2403.

\bibitem{berarducci_otero_peterzil_pillay}Alessandro Berarducci, Margarita Otero, Ya'acov Peterzil, and Anand Pillay. A descending chain condition for
groups in o-minimal structures. \textit{Annals of Pure and Applied Logic} 134 (2005). pp 303–313.

\bibitem{boxall_kestner}Gareth Boxall and Charlotte Kestner. The definable $(p,q)$-theorem for distal theories. \textit{The Journal of Symbolic Logic}, Vol. 83, No. 1 (2018). pp. 123-127

\bibitem{chernikov}Artem Chernikov. Externally definable fsg groups in NIP theories. Preprint. arXiv:2506.23265.

\bibitem{chernikov_gannon_krupinski}Artem Chernikov, Kyle Gannon, and Krzysztof Krupi\'{n}ski. Definable convolution and idempotent Keisler measures III. Generic stability, generic transitivity, and revised Newelski's conjecture. \textit{Journal of the LMS}, 2026.

\bibitem{chernikov_kaplan}Artem Chernikov, Itay Kaplan. Forking and dividing in NTP$_2$ theories. \textit{The Journal of Symbolic Logic}, Vol. 77, No. 1 (2012), pp. 1-20.

\bibitem{chernikov_pillay_simon}Artem Chernikov, Anand Pillay, Pierre Simon. External definability and groups in NIP theories. \textit{Journal of the London Mathematical Society}, Vol. 90, Iss. 1 (2014), pp. 213-240.

\bibitem{chernikov_simon_externally_1}Artem Chernikov and Pierre Simon. Externally definable sets and dependent pairs. \textit{Israel Journal of Mathematics}, Vol. 194, no. 1 (2013). pp 409-425.

\bibitem{chernikov_simon_externally_2}Artem Chernikov and Pierre Simon. Externally definable sets and dependent pairs II. Transactions of the AMS 367 (2015). pp. 5217-5235.

\bibitem{chernikov_simon}Artem Chernikov, Pierre Simon. Definably amenable NIP groups. \textit{Journal of the AMS}, Vol. 31, No. 3 (2018), pp. 609-641.

\bibitem{conant_pillay}Gabriel Conant and Anand Pillay. Pseudofinite groups and VC-dimension. \textit{Journal of Mathematical Logic}, 21(2). (2021).

\bibitem{conant_pillay_terry}Gabriel Conant, Anand Pillay, Caroline Terry. Structure and regularity for subsets of groups with finite VC-dimension. \textit{Journal of the EMS}, Vol 24, No 2 (2022). pp 583–621.


\bibitem{conversano_1}Annalisa Conversano. Maximal compact subgroups in the o-minimal setting. \textit{Journal of Mathematical
Logic}, 13 (2013). pp. 1-15

\bibitem{conversano_2}Annalisa Conversano. A reduction to the compact case for groups definable in o-minimal structures.
\textit{Journal of Symbolic Logic}, 79 (2014). pp. 45-53.

\bibitem{conversano_pillay}Annalisa Conversano and Anand Pillay. Connected components of definable groups and o-minimality I. \textit{Advances in Mathematics}. 
Volume 231, Issue 2. (2012). pp.605-623

\bibitem{folland}Gerald Folland. A Course in Abstract Harmonic Analysis. CRC Press. (1995)

\bibitem{kyle_tomasz}Kyle Gannon and Tomasz Rzepecki. Upside down and backwards. preprint. arxiv:2512.03576.

\bibitem{gismatullin}Jakub Gismatullin. Model-theoretic connected components of groups. \textit{Israel Journal of Mathematics}, Vol. 184 (2011), pp. 251-274.

\bibitem{glasner_book}Eli Glasner. Proximal Flows. Springer (1976).

\bibitem{glasner_tameness}Eli Glasner. The structure of tame minimal dynamical systems. \textit{Ergodic Theory and Dynamics Systems}, Vol. 27 Iss. 6 (2007), pp. 1819-1837.

\bibitem{guingona_laskowski}Vince Guingona and Michael Laskowski. On VC-minimal theories and variants. \textit{Archive for Mathematical Logic}, 52, (2013). pp 743-758.

\bibitem{halmos}Paul Halmos. Measure Theory. Springer. (1950)

\bibitem{hofmann_morris}Karl Hofmann and Sidney Morris. The Structure of Compact Groups. De Gruyter. (2020).

\bibitem{hrushovski_lascar_group}Ehud Hrushovski. Beyond the Lascar Group. Preprint. arXiv: arXiv:2011.12009.

\bibitem{hrushovski_thesis}Ehud Hrushovski. Contributions to Stable Model Theory. PhD thesis.

\bibitem{hrushovski}Ehud Hrushovski. Approximate Equivalence Relations. \textit{Model Theory}, Vol. 3, No. 2 (2024).

\bibitem{hrushovski_stable_groups}Ehud Hrushovski. Stable group theory and approximate subgroups. \textit{Journal of the AMS}. Vol. 25, No. 1 (2012). pp. 189-243.

\bibitem{hkp_1}Ehud Hrushovski, Krzysztof Krupi\'{n}ski, and Anand Pillay. Amenability, connected components, and definable actions. \textit{Selecta Mathematica} (28), 16 (2022)

\bibitem{hkp_2}Ehud Hrushovski, Krzysztof Krupi\'{n}ski, and Anand Pillay. On first order amenability. \textit{Selecta Matheamtica} (32), 23 (2026)

\bibitem{hrushovski_peterzil_pillay}Ehud Hrushovski, Ya'acov Peterzil, Anand Pillay. Groups, measures, and the NIP. \textit{Journal of the AMS}, Vol. 21, No. 2 (2008), pp. 563-596.

\bibitem{hpp_2}Ehud Hrushovski, Ya'acov Peterzil, Anand Pillay. On central extensions and definably compact groups in o-minimal structures. \textit{Journal of Algebra}, Vol 327, Iss. 1 (2011), pp. 71-106.

\bibitem{hrushovski_pillay}Ehud Hrushovski, Anand Pillay. On NIP and invariant measures. \textit{Journal of the EMS}, Vol. 13, No. 4 (2011), pp. 1005-1061.

\bibitem{jagiella}Grzegorz Jagiella. The Ellis Group Conjecture and variants of definable amenability. \textit{The Journal of Symbolic Logic}, Vol. 83, No. 4 (2018). pp. 1376-1390 

\bibitem{kaplan}Itay Kaplan. A definable $(p,q)$-theorem for NIP theories. Itay Kaplan. \textit{Advances in Mathematics}, Vol 436 (2024).

\bibitem{kirk}Thomas Kirk. Definable Topological Dynamics of $\mathrm{SL}_2(\mathbb{C}((t)))$. Preprint. 	arXiv:1903.03570.

\bibitem{krupinski_pillay_early}Krzysztof Krupi\'{n}ski and Anand Pillay. Generalized Bohr compactification and model-theoretic connected components. \textit{Math. Proc. Cambridge} 163.2 (2017), pp. 219–249.

\bibitem{krupinski_pillay}Krzysztof Krupi\'{n}ski and Anand Pillay. Generalized locally compact models for approximate groups. \textit{Advances in Mathematics}, 2026.

\bibitem{tomasz_krupinski_pillay}Krzysztof Krupi\'{n}ski, Anand Pillay, and Tomasz Rzepecki. Topological dynamics and the complexity of strong types. \textit{Israel Journal of Mathematics} 228 (2018). pp 863–932.


\bibitem{krupinski_rzepecki}Krzysztof Krupi\'{n}ski and Tomasz Rzepecki. Galois groups as quotients of Polish groups. \textit{Journal of Math Logic}, Vol. 20 (2020).

\bibitem{amador_daniel_1} Amador Martín-Pizarro and Daniel Palacín. Stabilizers, measures and IP-sets.
\textit{Notre Dame Journal of Formal Logic}. Vol. 66 No. 2 (2025), pp. 189-204.

\bibitem{amador_daniel_2}Amador Martín-Pizarro and Daniel Palacín. Complete type amalgamation for nonstandard finite groups. \textit{Model Theory}. Vol. 3 No. 1 (2024).

\bibitem{montenegro_onshuus_simon}Samaria Montenegro, Alf Onshuus, Pierre Simon. Stabilizers, NTP$_2$ Groups with f-Generic, and PRC Fields. \textit{Journal of the Institute of Mathematics of Jussieu}, Vol. 19 Iss. 3 (2020). pp.821-853.

\bibitem{newelski_1}Ludomir Newelski. Topological dynamics of definable group actions. \textit{The Journal of Symbolic Logic}, Vol. 74, No. 1 (2009), pp. 50-72.

\bibitem{newelski_2}Ludomir Newelski. Model theoretic aspects of the ellis semigroup.
\textit{Israel Journal of Mathematics}, Vol. 190 (2011), pp. 477-507.

\bibitem{newelski_petrykowski}Ludomir Newelski and Marcin Petrykowski. Weak generic types and coverings of groups I. \textit{Fundamenta Matematicae}, Vol. 191 Iss 3 (2006), pp. 201-225.

\bibitem{palacin}Daniel Palacín. On compactifications and product-free sets. \textit{Journal of the LMS}, Vol. 101, Iss. 1 (2020). pp. 156-174.

\bibitem{devrim}Devrim Pekmezci. Local revised Ellis group conjecture. In progress.

\bibitem{peterzil}Ya'acov Peterzil. Pillay's conjecture and its solution—a survey. In \textit{Logic colloquium 2007}. Cambridge University Press (2011).


\bibitem{pillay_penazzi_gismatullin}Jakub Gismatullin, Davide Penazzi, and Anand Pillay. Some model theory of $\mathrm{SL}_2(\mathbb{R})$. \textit{Fundamenta Mathematicae}, Vol. 229, Iss. 2 (2015). pp 117-128


\bibitem{pillay_penazzi_yao}Davide Penazzi, Anand Pillay, and Ningyuan Yao. Some model theory and topological dynamics of $p$-adic algebraic groups. \textit{fundamenta mathematicae}, Vol. 247 (2019). pp 191-216.

\bibitem{pillay_invariant_ellis_group}Anand Pillay. Topological dynamics and definable groups. \textit{The Journal of Symbolic Logic} 78.2 (2013), pp. 657–666.

\bibitem{pillays_conjecture}Anand Pillay. Type-definability, compact Lie groups, and o-minimality. \textit{Journal of Mathematical Logic}, 4
(2004). pp 147–162.

\bibitem{pillay_groups}Anand Pillay. Model theory and groups, in `Groups and Model Theory, GAGTA Book 2'. \textit{de Gruyter} (2021).

\bibitem{pillay_simple}Anand Pillay. Definability and definable groups in simple theories. The Journal of Symbolic Logic. Vol. 63, No. 3 (1998). pp. 788-796.

\bibitem{pillay_stonestrom}Anand Pillay and Atticus Stonestrom. On forking and invariant measures in NIP theories. \textit{Fundamenta Mathematicae} 273 (2026). pp 81-92

\bibitem{rosenthal}Haskell Rosenthal. A Characterization of Banach Spaces Containing $\ell^1$. \textit{Proceedings of the National Academy of Science, USA}, Vol. 71, No. 6 (1974). pp 2411-2413.

\bibitem{tomasz_thesis}Tomasz Rzepecki. Bounded Invariant Equivalence Relations. PhD thesis. arXiv:1810.05113

\bibitem{shelah_forking}Saharon Shelah. Dependent first order theories, continued. \textit{Israel Journal of Mathematics}, Vol. 173, No. 1 (2009), pp. 1-60.

\bibitem{shelah_strong}Saharon Shelah. Strongly dependent theories. \textit{Israel Journal of Mathematics}, Vol. 204 (2014), pp. 1-83.

\bibitem{shelah_g00}Saharon Shelah. Minimal bounded index subgroup for dependent theories. \textit{Proceedings of the AMS}, Vol. 136, No. 3 (2008), pp. 1087-1091.

%\bibitem{shelah_g000}Saharon Shelah. Definable groups for dependent and 2-dependent theories. 2009.

\bibitem{simon_book}Pierre Simon. A guide to NIP theories. \textit{Cambridge University Press} (2015).

\bibitem{simon_invariant_types}Pierre Simon. Invariant types in NIP theories.
\textit{Journal of Math. Logic}, Volume 15-2 (2015).

\bibitem{simon_rosenthal}Pierre Simon. Rosenthal Compacta and NIP formulas. \textit{Fundamenta Mathematicae}. Vol. 231 Iss. 1 (2015). pp 81-92

\bibitem{simon_vc_sets}Pierre Simon. VC-sets and generic compact domination. \textit{Israel Journal of Mathematics}, Vol. 218 (2017). pp 27-41.

\bibitem{stonestrom}Atticus Stonestrom. On f-generic types in NIP groups. Accepted at \textit{Advances in Mathematics}.

\bibitem{tent_ziegler}Katrin Tent and Martin Ziegler. A Course in Model Theory. Cambridge University Press, 2012.

\bibitem{wagner}Frank Wagner. Stable Groups. \textit{Cambridge University Press} (2013).

\bibitem{yao_zhang}Ningyuan Yao and Zhentao Zhang. Newelski's conjecture for o-minimal and $p$-adic Groups. Preprint. arXiv:2602.01810

\bibitem{yao_long}Ningyuan Yao and Dongyang Long. Topological dynamics for groups definable in real closed field. \textit{Annals of Pure and Applied Logic}. Volume 166, Issue 3 (2015). pp 261-273

%\bibitem{simon_distal}Pierre Simon. Distal and non-distal NIP theories. \textit{Annals of Pure and Applied Logic}, 2013.

%\bibitem{simon_compressible}Pierre Simon. Type decomposition in NIP theories. \textit{Journal of the EMS}, 2016.

%\bibitem{usvyatsov}Alexander Usvyatsov. Morley sequences in dependent theories. Preprint. 2008.


\end{thebibliography}

%} %small commented out
\end{document}